\documentclass[11pt,reqno]{amsart}
\usepackage[sc]{mathpazo}
\usepackage[T1]{fontenc}

\usepackage{amsmath}
\usepackage{amssymb,latexsym}
\usepackage{amsthm}
\usepackage{bm}
\usepackage{bbm,stmaryrd}
\usepackage[all,cmtip]{xy}
\input{diagxy}
\usepackage{url}
\usepackage[colorlinks=true,linkcolor=blue,anchorcolor=blue,citecolor=blue,
     		filecolor=blue,urlcolor=blue]{hyperref}

\newcommand{\opcit}{\emph{op.cit.}}
\newcommand{\ie}{\emph{i.e.}}
\newcommand{\eg}{\emph{e.g.}}
\newcommand{\ibid}{\emph{ibid.}}

\newcommand{\cf}{\emph{cf.}}

\newcommand{\A}{\ensuremath{\mathbb{A}}}
\newcommand{\bbtwo}{\ensuremath{\mathbbm{2}}}
\newcommand{\C}{\ensuremath{\boxempty}}

\newcommand{\N}{\ensuremath{\mathbb{N}}}
\newcommand{\D}{\ensuremath{\mathbb{D}}}
\newcommand{\B}{\ensuremath{\mathbb{B}}}
\newcommand{\bbC}{\ensuremath{\mathbb{C}}}
\newcommand{\Z}{\ensuremath{\mathbb{Z}}}
\newcommand{\T}{\ensuremath{\mathbb{T}}}
\newcommand{\CC}{\ensuremath{\mathcal{C}}}
\newcommand{\WW}{\ensuremath{\mathcal{W}}}
\newcommand{\FF}{\ensuremath{\mathcal{F}}}

\newcommand{\EE}{\ensuremath{\mathcal{E}}}
\newcommand{\psh}[1]{\ensuremath{\mathsf{Set}^{#1^{\mathrm{op}}}}}
\newcommand{\Set}{\ensuremath{\mathsf{Set}}}
\newcommand{\Cat}{\ensuremath{\mathsf{Cat}}}

\newcommand{\cSet}{\ensuremath{\mathsf{cSet}}}
\newcommand{\cSetZ}{\ensuremath{\mathsf{cSet}/_{\!Z}}}

\newcommand{\cEE}{\ensuremath{\mathsf{c}\mathcal{E}}}
\newcommand{\y}{\ensuremath{\mathsf{y}}} 
\newcommand{\yon}{\ensuremath{\mathsf{y}}} 

\newcommand{\op}[1]{\ensuremath{{#1}^{\mathrm{op}}}}
\newcommand{\eval}{\ensuremath{\mathsf{eval}}}
\newcommand{\dom}{\ensuremath{\mathsf{dom}}}
\newcommand{\cod}{\ensuremath{\mathsf{cod}}}
\newcommand{\proj}{\ensuremath{\mathsf{p}}}

\newcommand{\hook}{\ensuremath{\hookrightarrow}}
\newcommand{\mono}{\ensuremath{\rightarrowtail}}
\newcommand{\ra}{\ensuremath{\rightarrow}}
\newcommand{\cof}{\ensuremath{\rightarrowtail}}
\newcommand{\fib}{\ensuremath{\twoheadrightarrow}}
\renewcommand{\to}{\ensuremath{\rightarrow}}
\newcommand{\too}{\ensuremath{\longrightarrow}}
\newcommand{\onto}{\ensuremath{\twoheadrightarrow}}
\newcommand{\pbh}[2]{#1\!\Rightarrow\!#2}
\newcommand{\gph}[1]{\ensuremath{\langle #1 \rangle}}

\newcommand{\Hom}{\ensuremath{\mathrm{Hom}}}
\renewcommand{\hom}{\ensuremath{\mathrm{Hom}}}
\newcommand{\Sub}[1]{\ensuremath{\mathsf{Sub}{#1}}}
\newcommand{\elem}[1]{\textstyle\int\!{#1}}

\newcommand{\I}{\ensuremath{\mathrm{I}}}
\newcommand{\J}{\ensuremath{\mathrm{J}}}
\newcommand{\II}{\ensuremath{\mathbb{I}}}
\renewcommand{\H}{\ensuremath{\Box}}

\newcommand{\del}{\ensuremath{\partial}}

\newcommand{\pair}[1]{\ensuremath{\langle #1\rangle}}

\newcommand{\U}{\ensuremath{\mathcal{U}}}
\newcommand{\UU}{\ensuremath{\,\dot{\mathcal{U}}}}
\newcommand{\V}{\ensuremath{\mathcal{V}}}
\newcommand{\VV}{\ensuremath{\dot{\mathcal{V}}}}
\newcommand{\SSet}{\ensuremath{\,\dot{\Set}}}
\newcommand{\Fib}{\ensuremath{\mathsf{Fib}}}
\newcommand{\FFib}{\ensuremath{\dot{\mathsf{Fib}}}}
\newcommand{\TFib}{\ensuremath{\mathsf{TFib}}}
\newcommand{\TTFib}{\ensuremath{\dot{\mathsf{TFib}}}}

\newcommand{\TCof}{\ensuremath{\mathsf{TCof}}}

\newtheorem{theorem}{Theorem}
\newtheorem*{theorem*}{Theorem}
\newtheorem{proposition}[theorem]{Proposition} 
\newtheorem{lemma}[theorem]{Lemma}
\newtheorem{corollary}[theorem]{Corollary} 

\theoremstyle{remark}
\newtheorem{remark}[theorem]{Remark} 
\newtheorem*{remarks*}{Remarks}
\newtheorem{example}[theorem]{Example}

\theoremstyle{definition}
\newtheorem{definition}[theorem]{Definition}

\usepackage{tikz}
\usepackage{pdfpages}
\usepackage{tikz-cd}
\newcommand{\pbmark}{\ar[dr, phantom, "\lrcorner" very near start, shift right=.5ex]}	

\newcommand{\pbcorner}[1][dr]{\save*!/#1-1.2pc/#1:(-1,1)@^{|-}\restore}

\begin{document}

\title{Cartesian cubical model categories}
\author{Steve Awodey}
\date{\today}


\begin{abstract}
\noindent The category of Cartesian cubical sets is introduced and endowed with a Quillen model structure
using ideas coming from recent constructions of cubical systems of univalent type theory.
\end{abstract}

\maketitle

\setcounter{tocdepth}{1}
\tableofcontents


\section*{Introduction}

Recent years have seen renewed interest in the cubical approach to abstract homotopy theory.  This contrasts with the more familiar and widespread simplicial approach, using which many sophisticated and powerful tools have been developed, such as simplicial model categories \cite{DwyerKan:1980sl}, quasi-categories \cite{Joyal:2008tq}, and higher toposes \cite{Lurie:2009ht}.  Of course, some early work like the original papers of D.\ Kan \cite{Kan:55,Kan:56} employed cubical sets, and some researchers such as R.\ Brown \cite{BROWN2018459} and R.\ Jardine \cite{Jardine:cubical} have developed such methods further in a more modern style, but they are swimming against the tide.  

The current interest in the cubical approach arises from connections with the  formal system of \emph{type theory} for the purpose of computerized proof checking \cite{AwodeyCoquand:2013}.  Unlike previous cubical models of homotopy theory, however, the cubes being used for this purpose are generally assumed to be closed under finite products; we call such cube categories \emph{Cartesian}.  This is a natural enough assumption to make for cubes, but one that has somehow escaped serious consideration---but for two notable exceptions: in A.~Grothendieck's famous letter to D.~Quillen, and the accompanying 600 page manuscript \emph{Pursuing stacks} \cite{G:1983}, such cubical sets make an appearance as \emph{test categories}, which model the homotopy category of spaces in a particular way.  In fact, the Cartesian cubes studied here are \emph{strict} test categories in the terminology of \opcit, meaning that the geometric realization functor preserves finite products (\cite{BuchholtzMoorehouse}).  The more familiar category of ``monoidal'' cubical sets used since Kan is also a strict test category provided one includes \emph{connections} \cite{Maltsiniotis:2009}, but this is not necessarily Cartesian.  The second source for Cartesian cubical sets is F.W.~Lawvere, who proposed them as a model for homotopy theory in lectures, and in public and private correspondence, but never (to my knowledge) published anything on the subject.   Among their advantages, he stressed the tinyness of the 1-cube, or ``interval'' $\I$, which indeed plays a role in the current theory---although perhaps not the one envisioned by him.

We can define \emph{the} Cartesian cube category $\Box$ to be the Lawvere algebraic theory of bipointed objects, the opposite of which is therefore the category of finite, strictly bipointed sets $\B = \op\Box$.  Thus $\Box$ is the free finite product category with a bipointed object $[0]\rightrightarrows [1]$.  Our homotopy theory will be based on the category of \emph{Cartesian cubical sets}, which is the category of presheaves on $\Box$,
\[
\cSet = \psh{\Box}
\]
and thus consists of all \emph{covariant} functors $\B\to\Set$. Among these, there is an evident distinguished one, namely that which ``forgets the points'', and it is represented by the generating $1$-cube $[1]$,
\[
\I = \Box(-,[1]) : \B \too\Set\,.
\]

In cubical sets, the bipointed object $1\rightrightarrows\I$ turns out to have the (non-algebraic) property that its two points have a trivial intersection.  
\[
\xymatrix{
0 \ar[d] \ar[r] \pbcorner & 1 \ar[d]  \\
1 \ar[r] & \I
}
\]
We call such an object in a topos an \emph{interval}, and in a sense to be made precise, this is the universal one.  Other categories of Cartesian cubical sets have a canonical comparison to this one, relating their respective homotopy theories.

For the purpose of homotopy theory, namely, this interval provides a good cylinder $X + X \cof \I\times X$ for every object $X$, as well as a good path object $X^\I \fib X\times X$ for every \emph{fibrant} object $X$.  The notion of fibrancy here is determined by the interval $\I$ in terms of \emph{paths} $\I\to X$, and is a generalization of the \emph{path-lifting} condition from classical homotopy theory, suitably modified for this setting. We formulate it using the now-standard notion of a Quillen model structure:

\begin{definition}[cf.\cite{Quillen:1967ha}]\label{def:qmsviaJT}
A \emph{Quillen model structure} on a (bicomplete) category $\EE$ consists of three classes of maps $\CC, \WW, \FF$ satisfying the conditions:\begin{enumerate}
\item $(\CC, \WW\cap\FF)$ and $(\CC\cap\WW, \FF)$ are weak factorization systems,
\item $\WW$ has the 3-for-2 property: if any two sides of the triangle $e = f\circ g$ are in $\WW$, so is the third.
\end{enumerate}
\end{definition}

For the interval $1\rightrightarrows \I$, we have the mono $\del:1+1\cof \I$ as one of two basic cofibrations $\CC$ giving rise to all the others, in a certain sense.  The other basic one is the diagonal $\delta : \I\cof\I\times\I$, which is a special cofibration that, together with $\del$, determines both $\FF$ and $\WW$ just from the conditions (1) and (2) in the definition (which we have restated in a form due to \cite{JT:intro}).  Condition (1) has recently been termed a \emph{premodel structure} by R.\ Barton \cite{Barton}, and its verification in our setting is fairly routine, occupying less than the first half of the paper.   Condition (2) is where all the work is, and where our treatment is most likely to be of interest to the expert. We shall summarize those aspects below, but let us say now that the model structure is not the one determined by the method of Cisinski \cite{cisinski-asterisque}, nor is it Reedy \cite{Reedy:1974ht}, although the Cartesian cube category $\Box$ is ``generalized Reedy'' in the sense of \cite{BergerMoerdijk:2008rc}.  

Having identified $\Box$ as a strict test category, why not simply use standard tools to determine the test model structure on $\cSet$, making it equivalent to the standard homotopy theory of spaces?  Because \emph{we are mainly interested in how the model structure relates to the interpretation of type theory}.  
Specifically, we wish to investigate the relationship between the ingredients of a Quillen model structure and certain standard constructions in type theory, in order to better understand the somewhat mysterious connection between the two.

The first models of homotopy type theory used the standard Kan-Quillen model structure on simplicial sets \cite{awodey-warren:homotopy-idtype,KLV:21}. Much subsequent work has also relied on classical methods, including M.~Shulman's \emph{tour de force} result that  every Grothendieck $\infty$-topos admits a model of HoTT with a univalent universe \cite{shulman2019infty1toposes}.  This means that all of the results in the Homotopy Type Theory book \cite{hottbook} hold not only in the standard model in ``spaces,'' \ie\ simplicial sets, but also in any such higher topos.  In particular, the univalence \emph{axiom} of V.~Voevodsky is actually \emph{true} in all such models.
There is, however, a mismatch between such models of the univalence axiom and the design and implementation of computer systems based on type theory.  Taken as an axiom, univalence blocks the normalization algorithm which forms the basis of type theoretic computation. Voevodsky recognized this, and conjectured (roughly) that the system with the univalence axiom admitted an interpretation into the system without it, in a way that would restore effective computation.  

A version of this ``homotopy canonicity conjecture'' was finally verified a decade later by T.~Coquand and collaborators \cite{BCH,CCHM:2018ctt}.  One key insight that apparently led to their success was the ``change of shape'' from simplicial to cubical sets.\footnote{As suggested by \cite{BC:Kripke2015}.  Whether this alone is essential is still a matter of debate; arguably, it was rather the algebraic aspect underlying the ``uniform Kan filling'' condition that made the break-through possible.  Whether the cubical shape is essential to \emph{that} will perhaps be determined by recent work on an algebraic simplicial approach by \cite{gambino-henry}.}  
Some aspects of Coquand's work were undoubtedly informed by homotopy theory, but much of it was driven by type-theoretic considerations: normalization, canonicity, constructivity, etc.   
Subsequent work on computational systems of univalent type theory (such as \cite{orton-pitts,LOPS18,ABCHFL}) also used intuitions from basic homotopy theory (and some of the jargon), but without bothering to verify the model category axioms.  Of course, this research had a very different aim, namely the provision of a constructive system of type theory with univalence, which would facilitate its implementation in a computer proof system.  Once that was accomplished, there was no need to determine whether a Quillen model structure was also lurking in the background; it simply remained a mystery that the ingredients required for a computational system of univalent type theory seemed to align with the basic concepts of abstract homotopy theory.  

It was C.~Sattler who first recognized that a computational implementation of univalent type theory contains everything required to determine a Quillen model structure \cite{Sattler:2017ee}. An earlier result in this direction had been given by \cite{gambino-garner:idtypewfs}, who showed that the basic system of type theory with identity types not only interpreted into a weak factorization system (as had been shown by \cite{awodey-warren:homotopy-idtype}), but that it actually \emph{required} such a structure for its sound interpretation---essentially by constructing a weak factorization system from the system of type theory itself (P.~LeFanu Lumsdaine subsequently used higher inductive types to construct a second weak factorization system within homotopy type theory \cite{Lumsdaine:HITCofibrations}, making another step toward a full model structure).  The relationship between the full system of univalent type theory and a full Quillen model structure is somewhat more subtle---and part of the present investigation---but the mystery of \emph{why} the tools of model category theory seemed to work so well for constructing systems of univalent type theory is at least partially resolved by the insight that the type theory is apparently describing the same  kind of structure as do certain model categories; namely, that of a higher topos. 
So while there was no reason to expect \emph{a priori} that the work on computer proof systems 
would have any relevance to homotopy theory, the methods developed for those purposes have now acquired such relevance nonetheless. 

These new methods include various species of cubical sets with different combinatorial and homotopical properties \cite{BuchholtzMoorehouse}, some still unknown, as well as various composition, filling, and uniformity conditions with as yet unclear relationships to homotopical algebra \cite{orton-pitts,CCHM:2018ctt,BCH,ABCHFL,CMS:2020}. It is worth noting, for those not familiar with both, that translating between the language of type theory and that of model categories is by no means routine, nor is the converse anything like typesetting a commutative diagram in LaTeX.  (Indeed, the limits of such translation are a matter of current investigation, with the question of how to handle the coherences arising in higher category theory in type theory at the very forefront of current research.) 

The particular category of Cartesian cubes considered here has been studied by the author, in lectures and papers, since 2013 \cite{Awodey:cubical-model,awodey19cmu}, with various different box-filling conditions.  The condition explored in the present work, which we call \emph{unbiased partial box-filling}, was apparently considered early on by Coquand \cite{Coq:cartesian} but abandoned in favor of a monoidal one in \cite{BCH}, and later modified to one depending on the presence of connections in \cite{CCHM:2018ctt}.  The unbiased approach was resurrected and studied intensely in type theory by R.~Harper and his students \cite{BL:2014,AngiuliHarperWilson17,AngiuliFavoniaHarper18,Angiuli:2019}, culminating in \cite{ABCHFL}.   These type theoretic constructions are analyzed in terms of model categories here for the first time, doing for the system of Cartesian cubical type theory roughly what Gambino and Sattler \cite{GS,Sattler:2017ee} did for the system in \cite{CCHM:2018ctt}.

Specifically, we ultimately show that the category of Cartesian cubical sets admits a Quillen model structure $\CC, \WW, \FF$ with the unbiased fibrations as the class $\FF$ and the cofibrations $\CC$ axiomatized to allow for variations, including additional structure on the basic cube category $\Box$ and adjustments in the filling conditions.  Since our proofs are given in elementary diagrammatic form, they will also hold in other categories of Cartesian cubical sets, including those with connections, reversals, etc.  Indeed, part of our motivation is to apply the results obtained here, \emph{mutatis mutandis}, in two other settings: realizability, and equivariant filling.  The former (underway in \cite{AAFS}) imposes a strict condition of constructivity, about which we will say a bit more shortly.   The latter (underway in \cite{ACCRS}) is based on an unpublished result due to Sattler showing that an additional equivariance condition on the unbiased fibrations suffices to turn this model structure into the test model structure already mentioned.  

The possibility of an entirely constructive verification of the Quillen model category axioms is a consequence of the constructive interpretation of univalent type theory labored over by Coquand and his collaborators, and it has applications for the homotopy theory of presheaves and sheaves that stand to be explored further (but cf.~\cite{Coq:stack2017}).   The important uniformity condition on the Kan filling operations is closely related to E.~Riehl's \emph{algebraic model structures} \cite{riehl-algebraic-model} and gives rise to a notion of structured fibration that admits classification, in the sense of classifying spaces, by means of what we here call \emph{classifying types}.   These classifiers are used to construct universal objects of various kinds: families, cofibrations, (trivial) fibrations, and ultimately a universal fibration $\UU\fib\U$, which acts like an object classifier in higher topos theory, but with a stricter universal property.  Our work shows that having such classifying types can be useful, \eg\ when ``changing the base'' from one slice category $\EE/_X$ to another $\EE/_Y$ along a map $f : Y\to X$, or along a more general geometric morphism $f^*\dashv f_* : \FF \to \EE$.

Another application of the constructivity of the model structure is the computation of homotopy invariants as a result of a constructive proof.  This was only a theoretical possibility until quite recently, when a breakthrough by  A.~Ljungstr\"om \cite{AxelLj:2022} finally allowed the computer system Cubical Agda \cite{VezzosiMortbergAbel19} (which is based on the results just mentioned of Coquand \emph{et al.}) to compute the value of $k\in \Z$ from a proof in homotopy type theory that $\pi_4(S^3)\cong \Z/_{\!k\Z}$, which had been done by hand 10 years earlier at the IAS by G.\ Brunerie \cite{GB:IAS2013}.  Realizability models of type theory based on constructively proven model structures should also have applications in computational homotopy theory.

One way to verify that our model structure is entirely constructive would be to formalize the proofs below in a proof assistant such as Agda.  While this could be of interest for the practice of translating model category proofs into type theory, in principle one would learn very little that is not already known, since the model structure given here is in a certain sense ``reverse-engineered'' from a computational interpretation of type theory that has already been fully formalized and verified (namely, that in \cite{ABCHFL}).  Although our definitions and proofs do not parallel those in \ibid\ in the way that a proper formalization would, the interpretation of type theory underlying them will be plainly visible to the experts---since the Quillen model category defined here was already found lurking, as it were, behind that system.

Let us now make this more explicit as we outline the contents of the paper (references to the literature occur at the corresponding points of the main text).
After defining the Cartesian cubical sets and establishing some basic facts about them in Section 1, Section 2  specifies the \emph{cofibrations} axiomatically, as a class of monomorphisms classified by a universal one $t:1 \cof \Phi$.  This permits using the associated polynomial endofunctor $P_t : \cSet\to \cSet$ (which is shown to be a monad by the axioms for cofibrations), to give an algebraic weak factorization system with the cofibrations as the left maps and the (retracts of) $P_t$-algebras as the maps on the right, which we define to be the \emph{trivial fibrations}.  Since the monad is fibered, the factorization system is stable under change of base, which we use to derive the familiar diagonal filling characterization of the trivial fibrations in algebraic form, and relate this to the uniform filling condition from type theory.  The polynomial monad $P_t : \cSet\to \cSet$ is related to the type theoretic partiality- or lifting-monad, and generalizes the partial map classifier from the early days of topos theory.

In Section 3, the \emph{fibrations} are defined in terms of the cofibrations via the Joyal-Tierney calculus of pushout-products and pullback-homs.  A ``biased'' version using the two endpoints $\delta_0, \delta_1: 1\rightrightarrows\I$ is given first, before specifying the ``unbiased''  version in terms of the generic point $\delta : 1 \to \I^*\I$  in the slice category $\cSet/_\I$, namely the diagonal $\I \to \I\times\I$.  Specifically, a map $f : X \to Y$ in $\cSet$ is defined to be an unbiased fibration if its pullback to $\cSet/_\I$ has the right lifting property against all maps of the form $c\, \otimes_I\, \delta$ where $c : C\mono Z$ is a cofibration over $\I$ and the pushout-product with $\delta$ is formed in $\cSet/_\I$. 
\begin{equation*}\label{diagram:unbiasedfibrationIntro}
\xymatrix{
Z +_C (C\times\I^*\I) \ar@{>->}[d]_{c\, \otimes_I\, \delta} \ar[r] & \I^*Y \ar[d]^{\I^*f} \\
Z\times \I^*\I \ar@{.>}[ru] \ar[r] & \I^*X
}
\end{equation*}
The two weak factorization systems of cofibrations and trivial fibrations, and trivial cofibrations and fibrations, are assembled formally into a Barton premodel structure in Section 4, where the \emph{weak equivalences} are determined and related to \emph{weak homotopy equivalences}:  maps that induce isomorphisms in the homotopy category by precomposition.  The 3-for-2 axiom is reduced to a technical condition dubbed the \emph{fibration extension property}, the proof of which is deferred.  This concludes Part 1, and attention shifts to establishing the fibration extension property.

Part 2, consisting of Sections 5-8, is essentially a 60 page proof of a lemma.  It seems entirely likely that a more direct proof could be given, dispatching the entire second part of the paper. Even in that event, however, the work done in Part 2 would remain worthwhile, for this is where an implicit construction of a model of (homotopy) type theory occurs: The Frobenius property in Section 5 establishes the interpretation of $\Pi$-types of fibrations along fibrations, and thus the right properness of the model structure, by an entirely new diagrammatic argument derived from one originally given in type theory.   In Section 6 we construct the classifying types for fibration structure and use them to give a new construction of a universal fibration $\UU\fib\U$.  This is where the tinyness of the interval $\I$ plays an unexpected role, and a related axiom on the cofibrations is discovered.

Sections 7 and 8 make implicit use of the model of type theory emerging in the background, and contribute new diagrammatic proofs of two fundamental facts about it.  In Section 7 an equivalence extension property is established which is closely related to the univalence of the universal fibration $\UU\fib\U$, and in Section 8 that property is used to finally establish the fibration extension property, which is seen to be equivalent to the statement that the base object $\U$ is fibrant.  In sum, then, the missing 3-for-2 property of the premodel structure from Part 1 is proven in Part 2 by constructing a fibrant, univalent universe of fibrant objects.

One thing that we learn from the exercise is that one can get quite far in constructing a model of type theory in a \emph{premodel} category,   without assuming a fibrant universe, its univalence, or even the presence of a universe at all!  Conversely, our results suggest that the presence of a fibrant, univalent universe in such homotopical semantics in a premodel structure is not just necessary for a full model of univalent type theory, but actually suffices for a full Quillen model structure. In this sense, a model of HoTT is \emph{equivalent} to a Quillen model structure of a certain kind---namely, one that presents a higher topos.

\subsection*{Acknowledgements.}

Foremost, I am indebted to Thierry Coquand for sharing his ideas in conversations at the IHES in Paris, online during the COVID-19 pandemic, at the CAS in Oslo, and on several other occasions going back to the IAS in Princeton.  In many of the same places, Andr\'e Joyal has provided patient advice as well as illuminating lectures.   The writings of Mike Shulman and Emily Riehl have also been extremely helpful, as have many discussions with, and especially some late comments from, the latter.  I have learned much from my past and present CMU colleagues Mathieu Anel, Marc Bezem, Ulrik Buchholtz, Jonas Frey, Bas Spitters, Andrew Swan, and the late Pieter Hofstra; I am especially grateful for their patience through many revisions and delays.  Many other people have given good advice over the long period of these investigations, including Bjorn Dundas, Peter Dybjer, Marcelo Fiore, Richard Garner, Nicola Gambino, Dan Licata, Peter LeFanu Lumsdaine, Per Martin-L\"of, Andy Pitts, Christian Sattler, Thomas Streicher, Benno van den Berg, and undoubtedly others that I am forgetting.
I am also grateful to the computational higher type theory group around Bob Harper at CMU, including his students Carlo Angiuli, Evan Cavallo, Favonia, and Jon Sterling, for challenging me to clarify my thoughts.  
For research stays during which some of this work was conducted, I thank in particular the Oslo Centre for Advanced Studies and the Institut des Hautes \'Etudes Scientifiques.   Finally, I must acknowledge my debt to the late Vladimir Voevodsky, whose profound contributions advanced the subject far beyond my original expectations.

This material is based upon work supported by the Air Force Office of Scientific Research under awards number FA9550-21-1-0009, FA9550-20-1-0305 and FA9550-15-1-0053, and the Army Research Office award number W911NF-21-1-0121 P00002.

\clearpage
\section{Cartesian cubical sets}\label{sec:cSet}

There are many different categories of cubes $\Box$ that can be taken as a site for homotopy theory \cite{Grandis:2003,BuchholtzMoorehouse}, and indeed several different ones have recently been explored in connection with cubical systems of (homotopy) type theory, including \cite{BCH,orton-pitts, CCHM:2018ctt,ABCHFL,CMS:2020}, to name a few.  The model structure developed here is intended to work with any of these, insofar as they are  \emph{Cartesian}, in the sense that the indexing cubes $[n]\in\Box$ are closed under finite products $[m]\times[n] = [m+n]$.   Rather than working axiomatically, though, we shall work in the initial such category, which we call \emph{the Cartesian cube category} $\C$, defined as the free finite product category on an interval $\delta_0, \delta_1 : 1\rightrightarrows \I$.

\begin{definition}
The objects $[n]$ of the \emph{Cartesian cube category} $\C$, called $n$-cubes, are finite sets of the form 
\[
[n] = \{0, x_1, ..., x_n, 1\}\,,
\]
 where the $x_1, ..., x_n,$ are arbitrary but distinct elements, and $0,1$ are further distinct, distinguished elements.
The arrows,
\[
f : [m] \ra [n]\,,
\]
are arbitrary bipointed maps $f' : [n]\ra [m]$ (note the variance!).  Thus $\B = \op\C$ is the category if finite, strictly bipointed sets.  
\end{definition}
As a Lawvere theory, the arrows $f : [m] \ra [n]$ in $\C$ may also be regarded as $n$-tuples of elements from the set $\{0, x_1, ..., x_m, 1\}$.   These can be generated under composition by faces, degeneracies, permutations, and diagonals (see \cite{parker:thesis} for further details).

\begin{definition}\label{def:cSet}
The category \cSet\ of \emph{Cartesian cubical sets} is the category of presheaves on the Cartesian cube category $\C$,
\[
\mathsf{cSet}\ =\ \psh{\C}.
\]
It is of course generated by the representable presheaves $\y{[n]}$, to be written \[
\I^n = \y{[n]}
\]
 and called the \emph{geometric $n$-cubes}.   
\end{definition}

Note that the representables $\I^n$ are also closed under finite products, $\I^m \times \I^n = \I^{m+n}$.  We write $\I$ for $\I^1$ and $1$ for $\I^0$, which is terminal.  We will need the following basic fact about the cubes $\I^n$ in $\cSet$.

\begin{proposition}[Lawvere]\label{prop:Itiny}
The $n$-cubes $\I^n$ are \emph{tiny}, in the sense that the endofunctor $X\mapsto X^{\I^n}$ is a left adjoint.
\end{proposition}
\noindent (See \cite{Lawvere:2004} on such ``amazing right adjoints''.)

\begin{proof}
It clearly suffices to prove the claim for $n=1$.   For any cubical set $X$, the exponential $X^\I$ is a ``shift by one dimension'', 
\[
X^\I(n) \cong \hom(\I^n, X^\I)\ \cong\  \hom(\I^{n+1}, X)\cong\ X(n+1).
\]
Thus $X^\I$ is given by precomposition with the ``successor'' functor $\C\to\C$ with $[n] \mapsto [n+1]$. Precomposition always has a right adjoint, which in this case we shall write as 
\[
(-)^\I\dashv (-)_\I 
\]
 and call $X_\I$ the $\I^{th}$-\emph{root} of $X$.
\end{proof}

The following is used to calculate the root $X_\I$. A similar fact holds for the generic object in the object classifying topos $\Set[X]=\Set^{\mathsf{Fin}}$ and related categories used in the theory of abstract higher-order syntax \cite{FiorePlotkinTuri:1999}.

\begin{lemma}\label{lemma:binomial}
For the representable functor $\I = \y[1]$ in $\cSet$, we have $\I^\I\ \cong\ \I+1$.
\end{lemma}
\begin{proof}
For any $[n] \in \H$ we have:
\[
(\I^\I)(n) \cong \I(n+1) \cong \Hom(\I^{(n+1)},\I)\cong \H([n+1],[1])\cong \B([1], [n+1])\cong n+3.
\]
On the other hand,
\[
(\I+1)(n) \cong \I(n) + 1(n) \cong \Hom(\I^n, \I) + 1 \cong \B([1],[n]) +1 \cong (n+2) +1.
\]
The isomorphism is natural in $n$.
\end{proof}

\begin{corollary} For any cubical set $X$,
\begin{align*}\textstyle
X_\I(n) &\cong \hom(\I^n, X_\I) \\
&\cong \hom((\I^n)^\I, X) \\
&\cong \hom((\I^\I)^n, X) \\
&\cong \hom((\I+1)^n, X) \\
&\cong \hom(\I^n +  {\textstyle{n\choose n-1}}\I^{n-1} + \dots + {\textstyle{n\choose 1}}\I+1, X) \\
&\cong X_n \times X_{n-1}^{\binom{n}{n-1}} \times \dots \times X_1^{\binom{n}{1}}\times X_0\,.
\end{align*}
\end{corollary}

The exponential  $X^\I$ will be called the \emph{pathobject} of $X$, and plays a special role.  As we have just seen, it classifies ``paths'' in $X$; so the $0$-cubes $p \in (X^I)_0$ in the pathobject correspond to $1$-cubes $p\in X_1$, the ``endpoints'' of which $p_0, p_1\in X_0$ are given by composing with the evaluation maps 
\[
\epsilon_0, \epsilon_1 : X^\I \rightrightarrows X
\]
at the points $\delta_0, \delta_1 : 1\rightrightarrows \I$. More generally, higher cubes $c\in X_{n+1}$ correspond to maps $c : \I^{n+1}\to X$, which are thus paths between the $n$-cubes $c_0, c_1 : \I^n \to X$, corresponding to $c_0, c_1 \in X_n$.  
Note that, as a left adjoint, the pathobject functor $X\mapsto X^\I$ preserves all \emph{co}limits.   

We shall also need the following two facts concerning the interaction of cubes $\I^n$, pathobjects $X^\I$, and the base change functors associated to a map $f : X\to Y$  in $\cSet$, namely,
\[
f_!\dashv f^* \dashv f_* : \cSet/_X \too \cSet/_Y\,.
\]

\begin{lemma}\label{lemma:pathspacepushforward}
The pushforward functor along any map $f : X\to Y$ preserves pathobjects; for any object $A \to X$ over $X$, the pathobject of the pushforward $f_*A$ is (canonically isomorphic over $Y$ to) the pushforward of the pathobject,
\[
(f_*A)^\I \cong f_*(A^\I)\,.
\]
\end{lemma}
\begin{proof}
This is true for any constant family $X^*C = X\times C \to X$ with $C$ in place of $\I$, as the reader can easily verify using the Beck-Chevalley condition.
\end{proof}

\begin{lemma}\label{lemma:tinyslicedI}
The pulled-back interval $\I^*\I = \I\times\I\to \I$ in $\cSet/_\I$ is also tiny.
\end{lemma}
\begin{proof}
Since the interval $\I = \y[1]$ is representable, the slice category $\cSet/_\I$ is also a category of presheaves, namely over the sliced cube category $\C/_{[1]}$\,,
\[
\cSet/_\I \ =\ {\psh{\C}}\!/_{\y[1]} \ \cong\ \psh{ (\C/_{[1]}) }\,.
\]
However, since $\C$ does not have all finite limits, the sliced index category does not have all finite products, and so we cannot simply repeat the proof from Proposition \ref{prop:Itiny}.  But as in that proof, we do have a ``successor'' functor 
\[
s_{[1]} : \C/_{[1]} \to \C/_{[1]}\,,
\]
resulting from the ``predecessor'' natural transformation $s \Rightarrow 1_\C$ given by the  projection $\I\times X \to X$. Evaluating $s$ at each object $f : [n] \to [1]$ in  $\C/_{[1]}$, we obtain a commutative diagram:
\begin{equation}\label{diag:snatural}
\begin{tikzcd}
s[n]  \ar[d,swap, "{sf}"] \ar[r, "{\cong}" ] & {[1]}\!\times\! {[n]}  \ar[r, "p_n"] & {[n]} \ar[d, "f"] \\  
s[1]  \ar[r, swap, "{\cong}" ] & {[1]}\!\times\! {[1]}  \ar[r, swap, "p_1"] & {[1]}
 \end{tikzcd}
 \end{equation}
We can then set $s_{[1]}(f) = p_1\circ sf = f\circ p_n$. As in the foregoing proof, we can then calculate the values of the adjoints on presheaves, associated to $s_{[1]}$,
 \[
 {s_{[1]}}_! \dashv {s_{[1]}}^* : \widehat{\C/_{[1]}} \too \widehat{\C/_{[1]}} \]
  to be, successively,
 \begin{align*}
 {s_{[1]}}_! (X) &= \I^*\I \times X\,, \\
  {s_{[1]}}^* (X) &= X^{\I^*\I} \,.
 \end{align*}
The first equation follows from the observation that the diagram \eqref{diag:snatural} is a pullback, and so the object $s_{[1]}(f)  : s[n]\to [1]$ of $\widehat{\C/_{[1]}}$ given by the evident composite is just $\I^*\I \times f $, and the diagram itself represents the counit map $(\I^*\I \times f) \to f$ over $\I$.
 The second line then follows by adjointness, as does the fact that we have a further right adjoint, namely, the ${\I^*\I}^{th}$-\emph{root}:
 \begin{equation*}
 {s_{[1]}}_* (X) =: X_{\I^*\I} \,.
 \end{equation*}
\end{proof}

\section{The cofibration weak factorization system}\label{sec:cofibrations}

To build a model structure on the presheaf category of cubical sets, one can simply take as the cofibrations \emph{all} of the monomorphisms in $\cSet$; but for some purposes, it is convenient to know what is actually required of them (see \eg\ Appendix A). Thus, to begin, the following axioms are assumed.
 
\begin{definition}[Cofibrations]\label{def:cofibration}
The \emph{cofibrations} are a class $\mathcal{C}$ of monomorphisms satisfying the following conditions:
\begin{enumerate}
\item[(C0)] The map $0\to C$ is always a cofibration.
\item[(C1)] All isomorphisms are cofibrations.
\item[(C2)] The composite of two cofibrations is a cofibration.
\item[(C3)] Any pullback of a cofibration is a cofibration.
\end{enumerate}
We also require the cofibrations to be classified by a subobject $\Phi \hook \Omega$ of the standard subobject classifier $\top: 1 \to \Omega$ of $\mathsf{cSet}$:
\begin{enumerate}
\item[(C4)] There is a terminal object $t:1\to\Phi$ in the category of cofibrations and cartesian squares.
\end{enumerate}
Further axioms for cofibrations will be added later as needed: two early in Section \ref{sec:biasedfibration}, one later in Section \ref{subsec:unbiasedfibration}, and a final one in Section \ref{sec:realignment} (see Appendix A for a summary).   Cofibrations will be written 
 \[
 c : A \mono B\,.
 \]
\end{definition}

\subsection*{The cofibrant partial map classifier.}
Consider the polynomial endofunctor $P_t : \cSet\to \cSet$ determined by the cofibration classifier $t : 1 \mono \Phi$ (see  \cite{gambino-kock}).  We will write the value of this functor at an object $X$ as 
\begin{equation}\label{eq:partialmapclassifier}
X^+\ :=\ \Phi_!\,t_*(X)\ =\ \sum_{\varphi: \Phi}X^{[\varphi]}\,.
\end{equation}

The reader familiar with type theory will recognize the similarity to the ``partiality'' or ``lifting''  monad \cite{Moggi:91}.  When all monos are cofibrations, so that $\Phi = \Omega$, the object $X^+$ agrees with the partial map classifier $\widetilde{X}$ from topos theory \cite{JohnstoneTT}.  We may therefore regard $X^+$ as the object of \emph{cofibrant partial elements} of $X$, as we now explain.

Since $t: 1\cof \Phi$ is monic, $t^*t_*\cong 1$, so $X^+$ fits into the  pullback square
\begin{equation}\label{diagram:etadef}
\xymatrix{
X \ar[d]\pbcorner \ar@{>->}[r] & X^+ \ar[d]^{t_*X}\\
1 \ar@{>->}[r]_{t} & \Phi.
}
\end{equation}
Let $\eta : X\mono X^+$ be the indicated top horizontal map; we call this the \emph{cofibrant partial map classifier} of $X$.  By a \emph{cofibrant partial map} (from an object $Z$) into $X$ we mean a span $(c,x): Z\leftarrowtail C\ra X$ with a cofibration on the left.  The object $X^+$ is a \emph{classifying type} for such cofibrant partial maps, in that it has the following universal property.
 
\begin{proposition}\label{prop:cofparclass}
Let $\eta : X\mono X^+$ be as defined in \eqref{diagram:etadef}. 
\begin{enumerate}
\item The map $\eta : X\mono X^+$ is a cofibration.
\item For any object Z and any partial map $(c,x): Z\leftarrowtail C\ra X$, with $c : C\mono Z$ a cofibration, there is a unique $\chi : Z\ra X^+$ fitting into a pullback square as follows.
\[
\xymatrix{
C \ar@{>->}[d]_{c} \pbcorner \ar[r]^x & X \ar@{>->}[d]^{\eta}\\
Z \ar[r]_\chi & X^+
}
\]
The map $\chi : Z\ra X^+$ is said to \emph{classify} the partial map 
\[
(c,x): Z\leftarrowtail C\ra X\,.
\]
\end{enumerate}
\end{proposition}

\begin{proof}
The map $\eta : X\mono X^+$ is a cofibration since it is a pullback of the universal cofibration $t : 1\cof \Phi$. Observe that $(\eta, 1_X) : X^+\leftarrowtail X\ra X$ is therefore a cofibrant partial map into $X$.  The second statement is just the universal property of $X^+$ as a polynomial (see \cite{A:natural}, prop.~7). 
\end{proof}


\begin{proposition}\label{prop:plusmonad}
The pointed endofunctor  $\eta_X : X\mono X^+$ has a natural multiplication $\mu_X : X^{++} \ra X^+$ making it a monad.
\end{proposition}

\begin{proof}
Since the cofibrations are closed under composition, the monad structure on $X^+$ follows as in \cite{A:natural}, Lemma 5.  Explicitly, $\mu_X$ is determined by proposition \ref{prop:cofparclass} as the unique map making the following a pullback diagram.
\[
\xymatrix{
X \ar@{>->}[d]_{\eta_X} \ar[r]^= & X \ar@{>->}[dd]^{\eta}\\
X^+ \ar@{>->}[d]_{\eta_{X^+}} & \\
X^{++} \ar@{.>}[r]_\mu & X^+
}
\]
\end{proof}

\subsection*{Relative partial map classifier.}

For any object $X\in\cSet$ the pullback functor 
\[
X^* : \cSet \ra \cSet/_X\,,
\]
taking any $A$ to the (say) first projection $X\times A \ra X$, not only preserves the subobject classifier $\Omega$, but also the cofibration classifier $\Phi \hook \Omega$, where a map in $\cSet/_X$ is defined to be a cofibration if it is one in \cSet\ (under the forgetful functor $\cSet/_X \to \cSet$). Thus in $\cSet/_X$ we can define the \emph{(relative) cofibration classifier} to be the map
\[
X^*t : X^*1 \too X^*\Phi	\quad\text{over $X$}\,,
\]
which we may also write $t_X : 1_X \ra \Phi_X$.  Like $t : 1\ra \Phi$, this map determines a polynomial endofunctor  
\[
+_X : \cSet/_X \ra \cSet/_X\,,
\]
 which commutes (up to natural isomorphism) with $+ : \cSet \ra \cSet$ and $X^* : \cSet \ra \cSet/_X$ in the expected way, namely:
\begin{equation}\label{diag:+fibered}
\xymatrix{
\cSet/_X \ar[r]^{+_X}& \cSet/_X \\
\cSet \ar[u]^{X^*} \ar[r]_+ & \ar[u]_{X^*}\cSet 
}
\end{equation}
The endofunctor $+_X$ is also pointed $\eta_Y : Y \ra Y^+$ and has a natural monad multiplication $\mu_Y : Y^{++} \ra Y^+$, for any $Y\ra X$, for the same reason that $+$ has this structure.  Summarizing, we may say:

\begin{proposition}\label{prop:fiberedpolymonad}
The polynomial monad $+ : \cSet \ra \cSet$ of \emph{cofibrant partial elements} is indexed (or fibered) over \cSet.
\end{proposition}

\begin{definition}\label{def:+alg}
A \emph{$+$-algebra} in \cSet\ is an algebra for the pointed endofunctor $+ : \cSet\ra \cSet$.  Explicitly, a $+$-algebra is a cubical set $A$ together with a retraction $\alpha : A^+\ra A$ of the unit $\eta_A : A \ra A^+$.  Algebras for the monad $(+, \eta, \mu)$ will be referred to explicitly as \emph{$(+, \eta, \mu)$-algebras}, or \emph{$+$-monad algebras}. 

A \emph{relative $+$-algebra} in \cSet\ is a map $A \ra X$, together with an algebra structure over the codomain $X$ for the pointed endofunctor 
\[
+_X : \cSet/_X \too \cSet/_X\,.
\]
\end{definition}

\subsection*{The cofibration weak factorization system.}

The following proposition generalizes one in \cite{bourke-garner-I}.
\begin{proposition}
There is an (algebraic) weak factoriation system on \cSet\ with the cofibrations as the left class, and as the right class, the maps underlying relative $+$-algebras.  Thus a right map is one $p : A\ra X$ for which there is a retract $\alpha : A'\ra A$ over $X$ of the canonical map $\eta : A\ra A'$,
\[
\xymatrix{
A\ar[rd]_{p} \ar[r]^{\eta} \ar@/^5ex/ [rr]^= & A' \ar[r]^{\alpha} \ar[d]^-{p^+} & \ar[ld]^{p} A \\
& X. &
}
\]
\end{proposition}
(Note that the domain of $p^+ : A' \to X$ is not $A^+$, unless of course $X = 1$.)
\begin{proof}
The factorization of a map $f : Y\ra X$ is given by applying the relative $+$-functor over the codomain,
\[
\xymatrix{
Y\ar[rd]_{f} \ar@{>->}[r]^{\eta_f} & Y' \ar[d]^-{f^{+}} \\
& X. 
}
\]
We know by proposition \ref{prop:cofparclass} that the unit $\eta_f$ is always a cofibration, and since $f^{+}$ is the free algebra for the relative $+$-monad, it is in particular a $+$-algebra.

For the lifting condition, consider a cofibration $c : B\mono C$, and a right map $p:A\ra X$ with $+$-algebra structure map $\alpha: A' \ra A$ over $X$, and a commutative square as indicated below.
\[
\xymatrix{
B \ar@{>->}[dd]_{c} \ar[rr]^{a}  && A \ar[dd]_p \ar[rd]_{\eta} & \\
 &&& A' \ar@/_4ex/ [lu]_\alpha  \ar[ld]^{p^+} \\
C \ar[rr]_{x} && X &
}
\]
Thus in the slice category over $X$, we have
\[
\xymatrix{
B\ar@{>->}[d]_{c} \ar[r]^{a} & A \ar[d]^-{\eta} \\
C \ar@{.>}[ru]_{d} & A^+ \ar@/_4ex/ [u]_\alpha
}
\]
and we seek a diagonal filler $d$ as indicated.  (Note that we are writing $A^+$ for the map $p^+ : A'\to A$ regarded as an object over $X$, and similarly $C$ for $x : C\to X$ and $B$ for $xc : B\to X$ and $A$ for $p:A\to X$.)
Since $(c,a) : B \leftarrowtail C \ra A$ is a cofibrant partial map into $A$, by the universal property of $\eta : A \cof A^+$ (Proposition \ref{prop:cofparclass}) there is a unique classifying map $\varphi : C \ra A^+$ (over X) making a pullback square,
\[
\xymatrix{
B\ar@{>->}[d]_{c} \ar[r]^{a} & A \ar[d]^-{\eta} \\
C \ar@{.>}[r]_{\varphi} & A^+. \ar@/_4ex/ [u]_\alpha
}
\]
We can set $d := \alpha\circ \varphi : C \ra A$ to obtain the required diagonal filler, since $dc = \alpha\varphi{c} = \alpha\eta{a} = a$, because $\alpha$ is a retract of $\eta$.

The closure of the cofibrations under retracts follows from their classification by a universal object $t : 1 \cof \Phi$, and the closure of the right maps under retracts follows from their being the algebras for a pointed endofunctor underlying a monad (cf.~\cite{Riehl}).  Algebraicity of this weak factorization system is immediate, since $+$ is a monad.
\end{proof}

Summarizing, we have an algebraic weak factorization system $(\mathcal{C}, \mathcal{C}^\pitchfork)$ on the category \cSet\ of cubical sets, where:
\begin{align*}
\mathcal{C}\ &=\ \text{the cofibrations}\\
\mathcal{C}^\pitchfork\ &=\  \text{the maps underlying relative $+$-algebras}
\end{align*}
We shall call this the \emph{cofibration weak factorization system}. 
The right maps will be called \emph{trivial fibrations}, and the class of all such denoted
\[
\mathsf{TFib} = \mathcal{C}^\pitchfork\,.
\]

The cofibration algebraic weak factorization system is a generalization of one defined in \cite{bourke-garner-I} and mentioned in \cite{GS}.

\subsection*{Uniform filling structure.}

It will be useful to relate relative $+$-algebra structure to the more familiar diagonal filling condition of cofibrantly generated weak factorization systems, and specifically the special ones occuring in \cite{CCHM:2018ctt} under the name \emph{uniform filling structure} (this notion is also closely related to that of an \emph{algebraic weak factorization system}, \cf\ \cite{garner:small-object-argument,riehl-algebraic-model}).

Consider a generating sub\emph{set} of cofibrations consisting of those with representable codomain $c : C \mono \I^n$, and call these the \emph{basic cofibrations}.
\begin{equation}\label{eq:basiccof}
\mathsf{BCof} = \{c : C\mono \I^n\,|\ c\in \mathcal{C}, n\geq 0 \}.
\end{equation}

\begin{proposition}\label{prop:uniformstructequiv} For any object $X$ in \cSet\ the following are equivalent:
\begin{enumerate}
\item $X$ admits a \emph{$+$-algebra structure:} a retraction $\alpha :X^+ \ra X$ of the unit $\eta : X\ra X^+$.
\item $X\to 1$ is a \emph{trivial fibration:} it has the right lifting property with respect to all cofibrations,
\[
\mathcal{C}\, \pitchfork\,X.
\]
\item\label{smalluniformfilling} $X$ admits a \emph{uniform filling structure:} 
for each basic cofibration $c : C \mono \I^{n}$ and map $x : C\ra X$ there is given an extension $j(c,x)$,
\begin{equation}\label{diagram:uniformbasic}
\xymatrix{
C \ar@{>->}[d]_{c} \ar[r]^{x} & X, \\
\I^{n}\ar@{.>}[ru]_{j(c,x)} &
}
\end{equation}
and the choice is \emph{uniform in $\I^n$} in the following sense.  

Given any cubical map $u : \I^m \ra \I^n$, the pullback $u^*c : u^*C\mono \I^m$, which is again a basic cofibration, fits into a commutative diagram of the form
\begin{equation}\label{diagram:uniformbasic2}
\xymatrix{
u^*C \ar@{>->}[d]_{u^*c} \ar[rr]^{c^*u} \pbcorner &&  C \ar@{>->}[d]_{c} \ar[r]^{x} & X. \\
\I^{m} \ar[rr]_{u} && \I^{n} \ar@{.>}[ru]_{j(c,x)} &
}
\end{equation}
For the pair $(u^*c,\, x\circ c^*u)$ in \eqref{diagram:uniformbasic2}, the chosen extension $j(u^*c,x\circ c^*u): \I^m \ra X$, is required to be equal to  $j(c,x)\circ u$,
\begin{equation}\label{eq:uniformfillers}
j(u^*c,x\circ c^*u) = j(c,x) \circ u.
\end{equation}
\end{enumerate}
\end{proposition}

\begin{proof}
Let $(X, \alpha)$ be a \emph{$+$-algebra} and suppose given the span $(c,x)$ as below, with $c$ a cofibration. 
\begin{equation*}
\xymatrix{
C \ar@{>->}[d]_{c} \ar[r]^{x} & X \\
Z &
}
\end{equation*}
Let $\chi(c,x): Z\ra X^+$ be the classifying map of the cofibrant partial map $(c,x) : Z \leftarrowtail C \to X$, so that we have a pullback square as follows.
\begin{equation}\label{diagram:defphi}
\xymatrix{
C \ar@{>->}[d]_{c} \ar[rr]^{x} \pbcorner && X \ar[d]^{\eta} \\
Z \ar[rr]_-{\chi(c,x)} && X^+
}
\end{equation}
Then set
\begin{equation}\label{def:phi}
j = \alpha\circ\chi(c,x) : Z\ra X
\end{equation}
to get a filler,
\begin{equation}\label{diagram:defphi2}
\xymatrix{
C \ar@{>->}[d]_{c} \ar[rr]^{x} && X \ar[d]^{\eta} \\
Z\ar@{.>}[rru]_{j} \ar[rr]_{\chi(c,x)} && X^+  \ar@/_4ex/ [u]_\alpha
}
\end{equation}
since 
\[
j\circ c = \alpha\circ\chi(c,x)\circ c = \alpha\circ\eta \circ x = x.
\]
Thus (1) implies (2).  To see that it also implies (3), observe that in the case where $Z=I^n$ and we specify, in \eqref{def:phi}, that
\begin{equation}\label{def:j}
j(c,x) = \alpha\circ\chi(c,x) : \I^n\ra X,
\end{equation}
then the assignment is natural in $\I^n$. Indeed,  given any $u : \I^m \ra \I^n$, we have
\begin{equation}\label{eq:proof,uniformfillers}
j(c',xu') = \alpha\circ\chi(c',xu') = \alpha\circ\chi(c,x)\circ u = j(c,x) u,
\end{equation}
by the uniqueness of the classifying maps.

It is clear that (2) implies (1), since if $\mathcal{C} \pitchfork X$ then we can take as an algebra structure $\alpha : X^+ \ra X$ any filler for the universal span
\[
\xymatrix{
X \ar@{>->}[d]_{\eta} \ar[r]^{=} & X .\\
X^+ \ar@{.>}[ru]_\alpha&
}
\]

To see that (3) implies (1), suppose that $X$ has a uniform filling structure $j$ and we want to define an algebra structure $\alpha : X^+ \ra X$. By Yoneda, for every $y : \I^n \ra X^+$ we need a map $\alpha(y) : \I^n \ra X$, naturally in $\I^n$, in the sense that for any $u : \I^m \ra \I^n$, we have
\begin{equation}\label{eq:proof,plusstructure}
\alpha(yu) = \alpha(y)u.
\end{equation}
Moreover, to ensure that $\alpha\eta = 1_X$, for any $x : \I^n \ra X$ we must have $\alpha(\eta\circ x) = x$. So take $y : \I^n \ra X^+$  and let 
\[
\alpha(y) = j(y^*\eta, y'),
\]
as indicated on the right below.
\begin{equation}\label{diagram:definingalpha}
\xymatrix{
u^*C \ar@{>->}[d]_{u^*y^*\eta} \ar[r]^{u'} \pbcorner &  C \ar@{>->}[d]_{y^*\eta} \ar[rr]^{y'}\pbcorner && X \ar[d]^\eta . \\
\I^{m} \ar[r]_{u} & \I^{n} \ar@{.>}[rru]_{j(y^*\eta,y')} \ar[rr]_{y} && X^+
}
\end{equation}
Then for any $u : \I^m \ra \I^n$, we indeed have 

\[
\alpha(yu) =  j\big( (yu)^*\eta, y'u' \big) = j(y^*\eta, y')\circ u = \alpha(y)u,
\]
 by the uniformity of $j$. Finally, if $y=\eta\circ x$ for some $x : \I^n\ra X$ then 
 \[
 \alpha(\eta x) =  j\big( (\eta x)^*\eta, (\eta x)'\big) = j(1_X, x) = x,
 \]
 because the defining diagram for $\alpha(\eta x)$, i.e.\ the one on the right in \eqref{diagram:definingalpha}, then factors as
 \begin{equation}\label{diagram:definingalphatwo}
\xymatrix{
\I^n \ar@{>->}[d]_{=} \ar[r]^{x} \pbcorner &  X \ar@{>->}[d]_{=} \ar[r]^{=}\pbcorner & X \ar[d]^\eta , \\
\I^{n} \ar[r]_{x} & X  \ar[r]_{\eta} & X^+
}
\end{equation}
and the only possible extension $j(1_X, x)$ for the span $(1_{\I^n}, x)$ is  $x$ itself.
 \end{proof}
 
 \begin{remark}\label{remark:largeuniformfilling}
 Observe that the uniformilty condition (\ref{smalluniformfilling}) can be extended to the \emph{class of all} cofibrations, in the form:
 
 \begin{enumerate}
 \item[4.]\label{largeuniformfilling} $X$ admits a \emph{(large) uniform filling structure:} 
for each cofibration $c : C \mono Z$ and map $x : C\ra X$ there is given an extension $j(c,x)$,
\begin{equation}\label{diagram:largeuniformbasic}
\xymatrix{
C \ar@{>->}[d]_{c} \ar[r]^{x} & X, \\
Z\ar@{.>}[ru]_{j(c,x)} &
}
\end{equation}
and the choice is \emph{uniform in $Z$} in the following sense:
Given any map $u :Y \ra Z$, the pullback $u^*c : u^*C\mono Y$, which is again a cofibration, fits into a commutative diagram of the form
\begin{equation}\label{diagram:largeuniformbasic2}
\xymatrix{
u^*C \ar@{>->}[d]_{u^*c} \ar[rr]^{c^*u} \pbcorner &&  C \ar@{>->}[d]_{c} \ar[r]^{x} & X. \\
Y \ar[rr]_{u} && z \ar@{.>}[ru]_{j(c,x)} &
}
\end{equation}
For the pair $(u^*c,\, x\circ c^*u)$ in \eqref{diagram:largeuniformbasic2}, the chosen extension $j(u^*c,x\circ c^*u): \I^m \ra X$, is required to be equal to  $j(c,x)\circ u$,
\begin{equation}\label{eq:largeuniformfillers}
j(u^*c,x\circ c^*u) = j(c,x) \circ u.
\end{equation}
\end{enumerate}
Indeed, the proof that (1) implies (2) and (3) works just as well to infer (4), which in turn implies (2) and (3) as special cases.
\end{remark}

The relative version of the foregoing is entirely analogous, since the $+$-functor is fibered over $\cSet$ in the sense of diagram \eqref{diag:+fibered}.  We can therefore omit the entirely analogous proof of the following.
 
 \begin{proposition}\label{prop:uniformstructequivrelative} For any map $f : Y\ra X$ in \cSet\ the following are equivalent:
\begin{enumerate}
\item $f:Y\ra X$ admits a \emph{relative $+$-algebra structure over $X$}, i.e.\ there is a retraction  $\alpha :Y' \ra Y$ over X of the unit $\eta : Y\ra Y'$, where $f^+ : Y' \ra X$ is the result of the relative $+$-functor applied to $f$, as in definition \ref{def:+alg}.
\item $f:Y\ra  X$ is a \emph{trivial fibration},
\[
\mathcal{C}\, \pitchfork\,f.
\]
\item $f:Y\ra  X$ admits a (small) \emph{uniform filling structure}: 
for each basic cofibration $c : C \mono \I^{n}$ and maps $x : C\ra X$ and $y : \I^n\ra Y$ making the square below commute, there is given a diagonal filler $j(c,x,y)$,
\begin{equation}\label{diagram:coffillers}
\xymatrix{
C \ar@{>->}[d]_{c} \ar[rr]^{x} && X \ar[d]^{f} \\
\I^{n}\ar@{.>}[rru]_{j(c,x,y)} \ar[rr]_{y} && Y,
}
\end{equation}
and the choice is \emph{uniform in $\I^n$} in the following sense: given any cubical map $u : \I^m \ra \I^n$, the pullback $u^*c : u^*C\mono \I^m$ is again a basic cofibration and fits into a commutative diagram of the form
\begin{equation}\label{diagram:coffillers2}
\xymatrix{
u^*C \ar@{>->}[d]_{u^*c} \ar[rr]^{c^*u} \pbcorner &&  C \ar@{>->}[d]_{c} \ar[rr]^{x} && X\ar[d]^{f} \\
\I^{m} \ar[rr]_{u} && \I^{n} \ar@{.>}[rru]_{j(c,x,y)} \ar[rr]_{y} && Y.
}
\end{equation}
For the evident triple $(u^*c,\, x\circ c^*u, y\circ u)$ in \eqref{diagram:coffillers2} the chosen diagonal filler 
\[
j(u^*c,x\circ c^*u,y\circ u): \I^m \ra X
\]
 is equal to  $j(c,x,y)\circ u$,
\begin{equation}\label{eq:uniformfillers1}
j(u^*c, x\circ c^*u,y\circ u) = j(c,x,y)\circ u.
\end{equation}
\end{enumerate}
And again, a large version of (3) with arbitrary cofibrations $c : C\cof Z$ is again equivalent to (1)-(3).
\end{proposition}

We next collect some basic facts about trivial fibrations that will be needed later: they have sections, they are closed under composition and retracts, and they are closed under pullback and pushforward along all maps.

\begin{corollary}\label{cor:plusalgprops} 
\begin{enumerate}
\item Every trivial fibration $A\ra X$ has a section $s : X\ra A$.
\item If $a:A\ra X$ is a trivial fibration and $b:B\ra A$ is a trivial fibration, then  $a\circ b:B\ra X$ is a trivial fibration.
\item If $a:A\ra X$ is a trivial fibration and $a':A'\ra X'$ is a retract of $a$ in the arrow category, then $a'$ is a trivial fibration.
\item For any map $f:X\ra Y$ and any trivial fibration $B\ra Y$,  the pullback $f^*B \ra X$ is a trivial fibration.  
\item For any map $f:X\ra Y$ and any trivial fibration $A\ra X$,  the pushforward $f_*A \ra Y$ is a trivial fibration.  
\end{enumerate}
\end{corollary}
\begin{proof}
(1) holds because all objects are cofibrant by (C0). (5) is a consequence of (C3), stability of cofibrations under pullback, by a standard argument using the adjunction $f^* \dashv f_*$.  The rest hold for the right maps in any weak factorization system.
\end{proof}

\begin{remark}\label{trivfibpushforward}
The structured notion of trivial fibration, \emph{vis}.\ relative +-algebra, can also be shown algebraically (i.e.\ not using Proposition \ref{prop:uniformstructequivrelative}) to be closed under composition and retracts and preserved by pullback and pushforward. We do just the case of pushforward as an example.   
Thus consider the following situation with $A\to X$ a +-algebra with structure $\alpha$, as indicated.
\begin{equation}\label{diagram:pushforwardplusalgebras}
\xymatrix{
A \ar[rd]_{\eta_A} \ar[dd]  && f_*A \ar[rd]_{\eta_{f_*A}} \ar[dd] & \\
& A^+ \ar[ld] \ar@/_3ex/ [ul]_\alpha && (f_*A)^+ \ar[ld]  \ar@{.>}@/_3ex/ [ul]_\beta \\
X \ar[rr]_{f} && Y &,
}
\end{equation}
A +-algebra structure for $f_*A \to Y$ would be a retract $\beta : (f_*A)^+ \to f_*A$ of $\eta_{f_*A} : f_*A \to (f_*A)^+$ over $Y$, which corresponds under $f^*\dashv f_*$ to a map $\tilde\beta : f^*((f_*A)^+) \to A$ over $X$ with 
\[
\tilde\beta \circ f^*\eta_{f_*A} = \epsilon_A
\]
as indicated below.
\begin{equation}\label{diagram:pushforwardplusalgebras2}
\xymatrix{
f^*f_*A \ar[rrdd]_{\epsilon_A} \ar[rr]_{f^*\eta_{f_*A}} 
	&&  f^*((f_*A)^+)  \ar@{.>}[dd]_{\tilde\beta} \\
&&\\
&& A \ar[r]_{\eta_A} & \ar@/_3ex/ [l]_\alpha A^+ .
}
\end{equation}
But since pullback $f^*$ commutes with $+$, there is a canonical iso $c : f^*((f_*A)^+)  \cong (f^*f_*A)^+$ with $ c\circ f^*\eta_{f_*A} = \eta_{f^*f_*A} $. So we can set $\tilde\beta := \alpha\circ(\epsilon_A)^+\circ c$.
\begin{equation}\label{diagram:pushforwardplusalgebras3}
\xymatrix{
f^*f_*A \ar[rrdd]_{\epsilon_A} \ar[rr]_{f^*\eta_{f_*A}}  \ar@/^4ex/ [rrr]^{\eta_{f^*f_*A}} 
	&&  f^*((f_*A)^+)  \ar@{.>}[dd]_{\tilde\beta} \ar[r]^{\sim}_{c}  & (f^*f_*A)^+  \ar[dd]^{(\epsilon_A)^+} \\
&&&\\
&& A \ar[r]_{\eta_{A}} & \ar@/_3ex/ [l]_\alpha A^+ 
}
\end{equation}
\end{remark}

\section{The fibration weak factorization system}\label{sec:fibrations}

We now specify a second weak factorization system, with a restricted class of ``trivial'' cofibrations on the left, and an expanded class of right maps, the \emph{fibrations}.  As explained in the introduction, we first recall from \cite{GS} what we shall call the ``biased'' notion of fibration, before giving the ``unbiased'' one appropriate to our more general setting.   The two versions are equivalent in the presence of \emph{connections} 
\[
\vee,\wedge : \I\times\I \too \I
\]
 on the cubes, which are used in \cite{Sattler:2017ee} to determine a model structure with biased fibrations.  In \cite{AGH} it is shown that the biased fibrations of \opcit\ agree with those specified in the ``logical style'' of \cite{CCHM:2018ctt,orton-pitts}.    Note that we do \emph{not} assume connections in the category $\Box$ of Cartesian cubical sets.

\subsection*{Partial box filling (biased version)}\label{sec:biasedfibration}

The \emph{generating biased trivial cofibrations} are all maps of the form
\begin{equation}\label{eq:genclassTCof}
c \otimes \delta_\epsilon : D \mono Z\times \I\,,
\end{equation}
where:
\begin{enumerate}
\item  $c : C \mono Z$ is an arbitrary cofibration,

\item $\delta_\epsilon : 1 \ra \I$ is one of the two \emph{endpoint inclusions}, for $\epsilon = 0,1$.

\item $c\otimes\delta_\epsilon$ is the \emph{pushout-product} indicated in the following diagram.
\begin{equation}\label{diagram:pushoutproduct}
\xymatrix{
C\times 1 \ar[d]_{C\times 1} \ar[r]^{C\times \delta_\epsilon} & C\times \I \ar[d] \ar@/^4ex/ [rdd]^{c\times \I}\\
Z\times 1 \ar@/_4ex/ [rrd]_{Z\times \delta_\epsilon} \ar[r] &  Z +_C (C\times\I) \ar@{..>}[rd]_{c\, \otimes\, \delta_\epsilon} \\
&& Z\times\I
}
\end{equation}

\item $D = Z +_C (C\times\I)$ is the indicated domain of the map $c \otimes \delta_\epsilon$.
\end{enumerate}

In order to ensure that such maps are indeed cofibrations, we assume two further axioms in addition to (C1)--(C4) from Definition \ref{def:cofibration}:
\begin{enumerate}\label{cofibration_axioms}
\item[(C5)] The endpoint inclusions $\delta_\epsilon : 1 \ra \I$ are cofibrations, for $\epsilon = 0,1$.
\item[(C6)] The cofibrations are closed under joins $A\vee B \mono C$ of subobjects $A, B \mono C$ of any object $C$.
\end{enumerate}
\begin{remark}\label{rem:somecofibs}
Note that since $\delta_0: 1\to\I$ and $\delta_1:1\to\I$ are disjoint, by (C5) and stability under pullbacks we have that $0 \ra 1$ is a cofibration, so by stability again $0\ra A$ is always a cofibration.  Thus (C0) is no longer required.  From (C6) it follows that cofibrations are closed under pushout-products ${a}\otimes{b}$ in the arrow category.  It also then follows from (C5) that the boundary $\del : 1+1 \to \I$ is a cofibration.
\end{remark}

\subsection*{Fibrations (biased version).}
Now let 
\[
\mathcal{C}\otimes \delta_\epsilon\ =\ \{ c \otimes \delta_\epsilon : D \mono Z \times \I\ |\ c \in\mathcal{C},\ \epsilon = 0,1 \}
\]
be the class of all generating biased trivial cofibrations.
The \emph{biased fibrations} are defined to be the right class of these maps,
\[
(\mathcal{C}\otimes \delta_\epsilon)^\pitchfork\ =\ \mathcal{F}\,.
\]
Thus a map $f : Y\ra X$ is a biased fibration just if for every commutative square of the form
\begin{equation}\label{diagram:biasedfillers}
\xymatrix{
Z +_C (C\times\I) \ar@{>->}[d]_{c\, \otimes\, \delta_\epsilon} \ar[r] & Y \ar[d]^f \\
Z\times \I \ar@{.>}[ru]_{j} \ar[r] & X
}
\end{equation}
with a generating biased trivial cofibration on the left, there is a diagonal filler $j$ as indicated. 

To relate this notion of fibration to the cofibration weak factorization system, fix any map $u : A \ra B$, and recall (\eg\ from \cite{JT:notes,Riehl}) that the pushout-product with $u$ is a functor on the arrow category 
\[
(-)\!\otimes u : \cSet^\mathbbm{2} \ra\cSet^\mathbbm{2}\,.
\]
This functor has a right adjoint, the \emph{pullback-hom}, which for a map $f : X\ra Y$ we shall write as
\[
(u \Rightarrow\! f) : Y^B \too (X^B \times_{X^A} Y^A) \,.
\]
The pullback-hom is determined as indicated in the following diagram.
\begin{equation}\label{diagram:pullbackhom}
\xymatrix{
Y^B \ar@/_4ex/[rdd]_{f^B} \ar@{.>}[rd]^{u\, \Rightarrow f} \ar@/^4ex/ [rrd]^{Y^u} && \\
& X^B \times_{X^A} Y^A \ar[d] \ar[r] & Y^A \ar[d]^{f^A} \\
& X^B \ar[r]_{X^u} &  X^A 
}
\end{equation}

The $\otimes\dashv\ \Rightarrow$ adjunction on the arrow category has the following useful relation to weak factorization systems (cf.~\cite{GS,Riehl,JT:notes}), where, as usual, for any maps $a : A \ra B$ and $f : X\ra Y$ we write 
\[
a\, \pitchfork\, f 
\]
to mean that for every solid square of the form
\begin{equation}\label{eq:defpitchfork}
\xymatrix{
A\ar[d]_{a} \ar[r] & X \ar[d]^f \\
B \ar@{.>}[ru]_j \ar[r] & Y
}
\end{equation}
there exists a diagonal filler $j$ as indicated. 

\begin{lemma}\label{lemma:Leibniz}
For any maps $a : A_0 \to A_1, b : B_0 \to B_1, c : C_0 \to C_1$ in $\cSet$,
\[
(a\otimes b)\, \pitchfork\, c\quad\text{iff}\quad a\, \pitchfork\, (b\Rightarrow\! c)\,.
\]
\end{lemma}

The following is now a direct corollary.

\begin{proposition}\label{prop:algequivfill}
An object $X$ is fibrant if and only if both of the endpoint  projections $X^\I \ra X$ from the pathspace are trivial fibrations. More generally, a map $f : Y\ra X$ is a fibration iff both of the maps 
\[
(\delta_\epsilon \Rightarrow f) : Y^I \ra X^I\times_X Y
\]
 are trivial fibrations (for $\epsilon = 0,1$).
\end{proposition}

\subsection*{Fibration structure (biased version).}
The $\otimes\dashv\ \Rightarrow$ adjunction determines the fibrations in terms of the trivial fibrations, which in turn can be determined by \emph{uniform} lifting against a \emph{small category} consisting of basic cofibrations and pullback squares between them, by proposition \ref{prop:uniformstructequivrelative}.  The fibrations are similarly determined by \emph{uniform} lifting against the \emph{small category} of basic, biased trivial cofibrations, consisting of all those $c \otimes \delta_\epsilon$ in $\mathcal{C}\otimes \delta_\epsilon$ where $c : C \mono \I^n$ is a \emph{basic} cofibration, \ie\ one with representable codomain.  
Thus the set of \emph{basic biased trivial cofibrations} is
\begin{equation}\label{eq:basicTCof}
\mathsf{BCof}\otimes \delta_\epsilon = \{c \otimes \delta_\epsilon : B \mono \I^{n+1}\ |\ c : C\mono \I^n,\,\epsilon = 0,1,\ n\geq 0 \},
\end{equation}
where the pushout-product $c\otimes\delta_\epsilon$ now takes the simpler form
\begin{equation}\label{diagram:basicpushoutproduct}
\xymatrix{
C \ar@{>->}[d] \ar[r] & C\times \I \ar[d] \ar@/^4ex/ [rdd]\\
\I^{n} \ar@/_4ex/ [rrd] \ar[r] &  \I^n +_C (C\times\I) \ar@{>->}[rd]_{c\, \otimes\, \delta_\epsilon} \\
&& \I^{n}\times\I
}
\end{equation}
for a basic cofibration $c : C\mono \I^n$, an endpoint $\delta_\epsilon:1 \ra \I$, and with domain $B = \big(\I^n +_C (C\times\I)\big)$.   These subobjects $B \mono \I^{n+1}$ can be seen geometrically as generalized open box inclusions.

For any map $f:Y\ra X$  a \emph{uniform, biased fibration structure} on $f$ is a choice of diagonal fillers $j_\epsilon(c,x,y)$,
\begin{equation}\label{diagram:directedfillers}
\xymatrix{
\I^n +_C (C\times\I) \ar[d]_{c\,\otimes\,\delta_\epsilon} \ar[rr]^-{x} && X \ar[d]^f \\
\I^{n}\times \I \ar@{.>}[rru]_{j_\epsilon(c,x,y)} \ar[rr]_y && Y,
}
\end{equation}
for each basic biased trivial cofibration $c \otimes \delta_\epsilon : B = (\I^n +_C (C\times\I)) \mono \I^{n+1}$ and maps $x : B\ra X$ and $y : \I^{n+1}\ra Y$, which is \emph{uniform in $\I^{n}$} in the following sense: Given any cubical map $u : \I^m \ra \I^n$, the pullback $u^*c : u^*C\mono \I^m$ of $c : C\mono \I^n$ along $u$ determines another basic biased trivial cofibration 
\[
u^*c \otimes \delta_\epsilon : B' = (\I^m +_{u^*C} (u^*C\times\I)) \mono \I^{m+1},
\]
which fits into a commutative diagram of the form
\begin{equation}\label{diagram:directedfillers2}
\xymatrix{
\I^m +_{u^*C} (u^*C\times\I) \ar[d]_{u^*c\,\otimes\,\delta_\epsilon} \ar[rr]^-{(u\times\I)'} && \I^n +_C (C\times\I) \ar[d]_{c\,\otimes\,\delta_\epsilon} \ar[r]^-{x} & X \ar[d]^f \\
\I^{m}\times \I  \ar[rr]_{u\times \I} && \I^{n}\times \I \ar@{.>}[ru]_{j_\epsilon(c,x,y)} \ar[r]_{y} & Y,
}
\end{equation}
by applying the functor $(-)\otimes\,\delta_\epsilon$ to the pullback square relating $u^*c$ to $c$.  For the outer rectangle in \eqref{diagram:directedfillers2} there is then a chosen diagonal filler 
\[
j_\epsilon(u^*c,x\circ(u\times\I)', y\circ(u\times\I)): \I^m\times\I\ra X\,,
\]
and for this map we require that
\begin{equation}\label{eq:uniformfillers2}
j_\epsilon(u^*c,x\circ (u\times\I)', y\circ(u\times\I)) = j_\epsilon(c,x,y)\circ(u\times \I).
\end{equation}
This can be seen to be a reformulation of the logical specification given in \cite{CCHM:2018ctt} (see \cite{AGH}).

\begin{definition}\label{def:uniform} A \emph{uniform, biased fibration structure} on a map $f: Y\ra X$ is a choice of fillers $j_\epsilon(c,x,y)$ as in \eqref{diagram:directedfillers} satisfying \eqref{eq:uniformfillers2} for all maps $u : \I^m\ra\I^n$.
\end{definition}

Finally, we have the analogue of proposition \ref{prop:uniformstructequiv} for fibrant objects. The analogous statement of proposition \ref{prop:uniformstructequivrelative} for fibrations is omitted, as is the entirely analogous proof.

\begin{corollary}\label{cor:uniformfibstructequiv}  For any object $X$ in \cSet\ the following are equivalent:
\begin{enumerate}
\item $X$ is \emph{biased fibrant}, in the sense that every map $D\to X$ from the domain of a generating biased trivial cofibration $D \mono Z \times \I$ extends to a total map $Z \times \I \ra X$,
\[
\mathcal{C}\otimes \delta_\epsilon\ \pitchfork\ X\,.
\]

\item The canonical maps $(\delta_\epsilon\Rightarrow{X}) : X^I \ra X$ are trivial fibrations.

\item $X\ra 1$ admits a \emph{uniform biased fibration structure}.  Explicitly, for each basic biased trivial cofibration $c \otimes \delta_\epsilon : B \mono \I^{n+1}$ and map $x : B\ra X$, there is given an extension $j_\epsilon(c,x)$,
\begin{equation}\label{diagram:directedfillers2.5}
\xymatrix{
B \ar@{>->}[d]_{c \otimes \delta_\epsilon} \ar[r]^{x} & X, \\
\I^{n+1}\ar@{.>}[ru]_{j_\epsilon(c,x)} &
}
\end{equation}
and, moreover,  the choice is \emph{uniform in $\I^n$} in the following sense: Given any cubical map $u : \I^m \ra \I^n$, the pullback  $u^*c \otimes \delta_\epsilon : B'\mono \I^{m}\times\I$ fits into a commutative diagram of the form
\begin{equation}\label{diagram:directedfillers3}
\xymatrix{
B' \ar@{>->}[d]_{u^*c \otimes \delta_\epsilon} \ar[rr]^{(u\times\I)'} \pbcorner &&  B \ar@{>->}[d]_{c \otimes \delta_\epsilon } \ar[r]^{x} & X. \\
\I^{m}\times\I \ar[rr]_{u\times\I } && \I^{n}\times\I \ar@{.>}[ru]_{j(c,x)} &
}
\end{equation}
For the pair $(u^*c \otimes \delta_\epsilon,\, x\circ (u\times\I)')$ in \eqref{diagram:directedfillers3} the chosen extension 
\[
j(u^*c \otimes \delta_\epsilon,x\circ (u\times\I)'): \I^m\times\I \ra X
\]
 is equal to  $j(c,x)\circ (u\times\I)$,
\begin{equation}\label{eq:uniformfillers3}
j(u^*c \otimes \delta_\epsilon,x\circ(u\times\I)') = j(c,x) (u\times\I).
\end{equation}
\end{enumerate}
\end{corollary}

\subsection*{Partial box filling (unbiased version)}\label{subsec:unbiasedfibration}

Rather than building a weak factorization system based on the foregoing notion of biased fibration (as is done in \cite{GS}), we shall first eliminate the ``bias'' with respect to the endpoints $\delta_\epsilon : 1 \ra \I$, for $\epsilon = 0,1$.  This will have the effect of adding more trivial cofibrations, and thus more weak equivalences, to our model structure. Consider first the simple path-lifting condition for a map $f : Y \to X$, which is a special case of \eqref{diagram:biasedfillers} with $c =\, ! : 0\mono 1$, so that $!\otimes\delta_\epsilon = \delta_\epsilon$.
\begin{equation*}
\xymatrix{
1 \ar@{>->}[d]_{\delta_\epsilon} \ar[r] & Y \ar[d]^f \\
\I \ar@{.>}[ru]_{j_\epsilon} \ar[r] & X
}
\end{equation*}

In topological spaces, for instance, rather than requiring lifts $j_\epsilon$ for each of the endpoints $\epsilon = 0,1$  of the real interval $\I = [0,1]$, one could equivalently require there to be a lift $j_i$ \emph{for each point} $i: 1\ra\I$. Such ``unbiased path-lifting'' can be formulated in \cSet\ by introducing a ``generic point'' $\delta : 1\ra \I$ by passing to $\cSet/_\I$ via the pullback functor $\I^* : \cSet\to \cSet/_\I$, and then requiring path-lifting for $\I^*f$ with respect to $\delta : \I \to \I\times \I$, regarded as a map $\delta : 1\to\I^*{\I}$ in $\cSet/_\I$. We shall therefore define $f$ to be an unbiased fibration just if $\I^*f$ is a $\delta$-biased fibration for the generic point $\delta$.  The following specification implements that  idea, while also adding cofibrant partiality, as in the biased case.  

We first replace axiom (C5) with the following stronger assumption.
\begin{enumerate}
\item[(C7)] The diagonal map $\delta : \I\ra\I\times\I$ of the interval $\I$  is a cofibration.  
\end{enumerate}

The unbiased notion of a fibration for $\cSet$ is now as follows.

\begin{definition}[unbiased fibration]\label{def:unbiasedfibration}
Let $\delta : \I\ra\I\times\I$ be the diagonal map.
\begin{enumerate}
\item\label{item:fibrant} An object $X$ is  \emph{unbiased fibrant} if the map 
\[
(\delta\Rightarrow\! X) = \langle\mathsf{eval}, p_2\rangle : X^\I \times \I \ra X\times \I
\]
is a trivial fibration. 
\item\label{item:fibration} A map $f : Y\ra X$ is an \emph{unbiased fibration} if the map 
\[
(\delta\Rightarrow\! f) = \langle f^\I\times \I, \langle \mathsf{eval},p_2 \rangle\rangle: Y^\I\times \I \ra (X^\I \times \I)\times_{(X\times \I)} (Y\times\I)
\]
is a trivial fibration.
\end{enumerate}
\end{definition}

Let us (temporarily) write $\II = \I^*\I$ for the pulled-back interval in the slice category $\cSet/_\I$, so that the generic point is written $\delta : 1 \to \II$.
Condition \eqref{item:fibrant} above (which of course is a special case of \eqref{item:fibration}) then says that evaluation at the generic point $\delta : 1\ra \II$, the map $(\I^*X)^\delta: (\I^*X)^\II \ra \I^*X$, constructed in the slice category $\cSet/_\I$, is a trivial fibration.  
Condition \eqref{item:fibration} says that the pullback-hom of the generic point $\delta : 1\ra \II$ with $\I^*f$, constructed in the slice category $\cSet/_\I$, is a trivial fibration.  Thus a map $f: Y\to X$ is an \emph{unbiased} fibration just if its base change $\I^*f$  is a $\delta$-\emph{biased} fibration in the slice category $\cSet/_\I$.  The latter condition can also be reformulated as follows.

\begin{proposition}
A map $f : Y\ra X$ is an unbiased fibration if and only if the canonical map $u$ to the pullback, in the following diagram in $\cSet$, is a trivial fibration.
\begin{equation}\label{diagram:unbiasedfibrationeval}
\xymatrix{
Y^\I\times \I \ar@/^3ex/ [rrrd]^{\mathsf{eval}} \ar@/_3ex/ [rdd]_{f^\I\times \I} \ar@{..>}[rd]_{u}  && \\
& Y_\mathsf{eval} \pbcorner  \ar[d] \ar[rr] && Y \ar[d]^f \\
& X^\I \times \I \ar[rr]_{\mathsf{eval}} && X.
}
\end{equation}
\end{proposition}
\begin{proof}
We interpolate another pullback into the rectangle in \eqref{diagram:unbiasedfibrationeval} to obtain
\begin{equation}\label{diagram:unbiasedfibrationeval2}
\xymatrix{
 Y_\mathsf{eval} \pbcorner  \ar[d] \ar[r] & Y\times \I \ar[d] \pbcorner \ar[r] & Y \ar[d]^f \\
 X^\I \times \I \ar[r] & X\times \I \ar[r] & X
}
\end{equation}
with the evident maps. The left hand square is therefore a pullback, so we indeed have that 
\[
Y_\mathsf{eval}\ \cong\ (X^\I \times \I)\times_{(X\times \I)} (Y\times\I) \cong\ (X^\I \times \I)\times_{X} Y
\]
and $u = (\delta\Rightarrow\! f)$.
\end{proof}

As a special case, we have:
\begin{corollary}\label{cor:unbiasedfibrant}
An object $X$ is unbiased fibrant if and only if the canonical map $u$ to the pullback, in the following diagram in $\cSet$, is a trivial fibration.
\begin{equation}\label{diagram:unbiasedfibrationeval3}
\xymatrix{
X^\I\times \I \ar@/^3ex/ [rrrd]^{\mathsf{eval}} \ar@/_3ex/ [rdd]_{p_2} \ar@{..>}[rd]_{u}  && \\
& \I\times X \pbcorner  \ar[d] \ar[rr] && X \ar[d] \\
& \I \ar[rr] && 1.
}
\end{equation}
\end{corollary}

Now we can run the proof of Proposition \ref{prop:algequivfill} backwards in order to determine a class of generating trivial cofibrations for the unbiased case. Consider pairs of maps $c : C\mono Z$ and $i:Z\ra\I$, where the former is a cofibration and the latter is regarded as an ``$\I$-indexing'', so that
\[
\xymatrix{
C \ar[rd] \ar@{>->}[r]^{c} & Z  \ar[d]^i \\
& \I
}
\]
is regarded as an ``$\I$-indexed family of cofibrations $c_i : C_i \mono Z_i$''.  We shall use the notation
\begin{equation}\label{eq:graphdef}
\gph{i} := \langle 1_Z, i\rangle : Z\too Z\times\I\,,
\end{equation}
for the graph of the indexing map $i : Z\ra \I$. Then write
\[
c \otimes_i\!\delta := [\gph{i}, c\times\I] : Z+_C(C\times\I) \too Z\times \I\,,
\]
which is easily seen to be well-defined on the indicated pushout below.
\begin{equation}\label{diagram:unbiasedpushoutproduct}
\xymatrix{
C \ar@{>->}[d]_{c} \ar[r]^{\gph{ic}} & C\times \I \ar[d] \ar@/^4ex/ [rdd]^{c\times\I}\\
Z \ar@/_4ex/ [rrd]_{\gph{i}} \ar[r] &  Z +_C (C\times\I) \ar@{.>}[rd]_{c\, \otimes_i \delta} \\
&& Z\times\I\,.
}
\end{equation}

\begin{remark}\label{rem:poprodoverI}
The specification \eqref{diagram:unbiasedpushoutproduct} differs from the similar \eqref{diagram:pushoutproduct} by using the graph $\gph{i} : Z\mono Z\times\I$ for the inclusion of $Z$ into the cylinder over $Z$, rather than one of the two ``ends'',
\begin{equation}\label{eq:grtaphoverI}
\gph{1_Z , \delta_\epsilon !}  : Z\cong Z\times 1 \stackrel{Z\times \delta_\epsilon}{\longrightarrow} Z\times\I
\end{equation}
arising from the endpoint inclusions $\delta_\epsilon : 1\ra\I$, for $\epsilon = 0,1$.   As an arrow over $\I$, the graph $\gph{i} : Z\to Z\times\I$ also takes the form \eqref{eq:grtaphoverI}, namely
\[
\gph{i} = \gph{1_Z , \delta !} : Z\to Z\times \I\,.
\]
If we also regard $c : C\to Z$ as an arrow over $\I$ via $i : Z\to\I$, and use the generic point $\delta : 1 \to \II$ over $\I$ in place of $\delta_\epsilon : 1 \to \I$, then  \eqref{diagram:unbiasedpushoutproduct}  agrees with  \eqref{diagram:pushoutproduct}, up to those changes.  Thus the indicated map ${c\, \otimes_i \delta}$ in  \eqref{diagram:unbiasedpushoutproduct} \emph{is the pushout-product constructed over $\I$} of the generic point $\delta$ with the map $c$ regarded as an $\I$-indexed family of cofibrations via the indexing $i : Z \to \I$.
\end{remark}

Observe that for any map $i : Z\to \I$, the graph $\gph{i} = \langle 1_Z, i \rangle : Z \to Z\times \I$ is a cofibration, since it is a pullback of the diagonal of $\I$ along $i\times \I$.  The subobject 
\[
c \otimes_i\!\delta \mono Z\times \I
\]
 constructed in \eqref{diagram:unbiasedpushoutproduct} is therefore a cofibration, since it is the join in the lattice $\Sub(Z\times \I)$ of the cofibrant subobjects $\gph{i} \mono Z\times \I$ and $C\times \I \mono Z\times \I$, where the latter is the ``cylinder over $C\mono Z$''.

\begin{definition}\label{def:genunbiasedtrivcof}
The maps of the form $c\otimes_i\delta : Z +_C (C\times\I) \mono Z\times\I$  now form the class of \emph{generating unbiased trivial cofibrations},
\begin{equation}\label{eq:generatingtrivialcofibrations}
\mathcal{C}\otimes\delta\ =\ \{ c \otimes_i \delta : D \mono Z \times \I\ |\ c : C\mono Z, i:Z\ra\I \}\,.
\end{equation}
\end{definition}

We can then show that the unbiased fibrations are exactly the right class of these maps,
\[
(\mathcal{C}\otimes\delta)^{\pitchfork} = \mathcal{F}.
\]
\begin{proposition}\label{prop:fibiffrlp}
A map $f: Y\ra X$ is an unbiased fibration iff for every pair of maps $c : C\mono Z$ and $i:Z\ra\I$, where the former is a cofibration, every commutative square of the following form has a diagonal filler, as indicated in the following.
\begin{equation}\label{diagram:unbiasedfibration}
\xymatrix{
Z +_C (C\times\I) \ar@{>->}[d]_{c\, \otimes_i\, \delta} \ar[r] & Y \ar[d]^f \\
Z\times \I \ar@{.>}[ru]_{j} \ar[r] & X.
}
\end{equation}
\end{proposition}

\begin{proof}
Suppose that for all $c : C\mono Z$ and $i:Z\ra\I$, we have $(c\otimes_i\delta) \pitchfork f$ in $\cSet$. Pulling $f$ back over $\I$, this is equivalent to the condition $c\otimes\delta \pitchfork \I^*f$ in $\cSet/_\I$, for all cofibrations $c : C\mono Z$ over $\I$, which is equivalent to $c\pitchfork(\delta\Rightarrow \I^*f)$ in $\cSet/_\I$ for all cofibrations $c : C\mono Z$.  But this in turn means that $\delta\Rightarrow\I^*f$ is a trivial fibration, which by definition means that $f$ is an unbiased fibration.
\end{proof}

\begin{remark}\label{rem:specialtrivcofs}
Note that the endpoints $\delta_\epsilon : 1 \ra \I$, in particular, are of the form $c \otimes_i\!\delta$ by taking $Z = 1$ and $i = \delta_\epsilon$ and $c =\ ! : 0 \ra 1$, so that the case of biased filling is subsumed.  Moreover, for any $i : Z\to \I$ the graph $\gph{i}: Z \mono Z\times \I$ is itself of the form $0 \otimes_i\!\delta$ for the cofibration $0 \cof Z$, so the graph of any ``$\I$-indexing'' map $i : Z\to\I$ is also a trivial cofibration.
\end{remark}

The following sanity check will be needed later.
\begin{proposition}\label{prop:sanitycheck}
Let $f:F\fib X$ be an unbiased fibration in $\cSet$. Then for the endpoints $\delta_0, \delta_1 : 1 \to \I$, the associated pullback-homs,
\begin{equation}\label{eq:unbiasedimpliesbiased}
\delta_\epsilon \Rightarrow f : F^\I \to X^\I \times_X F \qquad(\epsilon = 0,1)
\end{equation}
are also trivial fibrations. Thus unbiased fibrations are also $\delta_\epsilon$-biased fibrations, for $\epsilon=0,1$.
\end{proposition}

\begin{proof}
This follows from Remark \ref{rem:specialtrivcofs} and the ${\otimes} \dashv {\Rightarrow}$ adjunction, but we give a different proof.
Consider the case $X=1$, the general one $f:F\to X$ being analogous. Thus let $F$ be an unbiased fibrant object in $\cSet$. So by definition $(\I^*F)^\delta : (\I^*F)^\II \too \I^*F$  in $\cSet/_\I$ is a trivial fibration. 
Pulling back $\delta : 1\to \II$ in $\cSet/_\I$ along the base change $\delta_\epsilon : 1\to \I$ takes it to $\delta_\epsilon : 1\to \I$ in $\cSet$, by the universal property of the generic point $\delta : 1\to \II$; that is $\delta_\epsilon^*(\delta) = \delta_\epsilon : 1 \to \I$.  So $(\I^*F)^\delta : (\I^*F)^\II \too \I^*F$ is taken by $\delta_\epsilon^*$ to
\[
\delta_\epsilon^* \big( {(\I^*F)^\delta} \big) =
(\delta_\epsilon^*\I^*F)^{\delta_\epsilon^*\delta} =  F^{\delta_\epsilon} : F^\I \too F\,,
\]
as shown in the following.
\begin{equation}\label{diagram:unbiasedfibration1}
\xymatrix{
F^\I \ar[d]_{F^{\delta_\epsilon}} \ar[rr] \pbcorner && (\I^* F)^{\II} \ar[d]^{(\I^*F)^{\delta}} &  \\
 F \ar[d] \ar[rr] \pbcorner && \I^*F \ar[d]  \ar[r] \pbcorner & F \ar[d]  \\
 1 \ar[rr]_{\delta_\epsilon} &&  \I \ar[r] & 1
}
\end{equation}
And pullback preserves trivial fibrations. 
\end{proof}
\subsection*{Unbiased fibration structure.}\label{sec:unbiasedfibration}

As in the biased case, the fibrations can be determined by \emph{uniform} right-lifting against a  \emph{small category} of unbiased trivial cofibrations, now consisting of all those $c \otimes_i \delta$ in $\mathcal{C}\otimes \delta$ for which $c : C \mono \I^n$ is basic, \ie\ has representable codomain.  Call these maps the \emph{basic unbiased trivial cofibrations}, and let 
\begin{equation}\label{eq:basicunbiasedTCof}
\mathsf{BCof}\otimes \delta = \{c \otimes_i \delta : B \mono \I^{n+1}\ |\ c : C\mono \I^n,\, i : \I^n \ra \I,\,n\geq 0\}\,,
\end{equation}
where the pushout-product $c\otimes_i \delta$ now has the form
\begin{equation}\label{diagram:unbiasedbasicpushoutproduct}
\xymatrix{
C \ar@{>->}[d]_{c} \ar[r]^{\gph{ic}} & C\times \I \ar[d] \ar@/^4ex/ [rdd]^{c\times\I}\\
\I^n \ar@/_4ex/ [rrd]_{\gph{i}} \ar[r] &  \I^n +_C (C\times\I) \ar@{.>}[rd]_{c\, \otimes_i \delta} \\
&& \I^n\times\I
}
\end{equation}
for a basic cofibration $c : C\mono \I^n$ and an indexing map $i : \I^n \ra \I$, and with domain $B = \big(\I^n +_C (C\times\I)\big)$.   These subobjects $B \mono \I^{n+1}$ can again be seen geometrically as ``generalized open box inclusions", but now the floor and lid of the open box are generalized to the graph of an arbitrary map $i:\I^n\ra \I$.

For any map $f:Y\ra X$  a uniform, unbiased fibration structure on $f$ is then a choice of diagonal fillers $j(c,i,x,y)$,
\begin{equation}\label{diagram:basicunbiasedfillers}
\xymatrix{
B \ar@{>->}[d]_{c\,\otimes_i\delta} \ar[rr]^x && X \ar[d]^f \\
\I^{n}\times \I \ar@{.>}[rru]_{j(c,i,x,y)} \ar[rr]_y && Y,
}
\end{equation}
for each basic trivial cofibration $c \otimes_i \delta : B \mono \I^{n+1}$, which is \emph{uniform in $\I^n$} in the following sense: Given any cubical map $u : \I^m \ra \I^n$, the pullback $u^*c : u^*C\mono \I^m$ and the reindexing $iu : \I^m \ra \I^n\ra\I$ determine another basic trivial cofibration $u^*c \otimes_{iu} \delta : B' = (\I^m +_{u^*C} (u^*C\times\I)) \mono \I^{m+1}$, which fits into a commutative diagram of the form
\begin{equation}\label{diagram:basicunbiasedfillersuniformity}
\xymatrix{
B' \ar[d]_{u^*c\,\otimes_{iu}\delta} \ar[rr]^{(u\times \I)'} \pbcorner  && B \ar[d]_{c\,\otimes_i\delta} \ar[rr]^x && X \ar[d]^f \\
\I^{m}\times \I  \ar[rr]_{u\times \I} && \I^{n}\times \I \ar@{.>}[rru]_{j(c,i,x,y)} \ar[rr]_y && Y.
}
\end{equation}
For the outer rectangle in \eqref{diagram:basicunbiasedfillersuniformity} there is a chosen diagonal filler 

\[
j(u^*c,iu,x(u\times \I)', y(u\times \I)): \I^m\times\I\ra X,
\]
 and for this map we require that
\begin{equation}\label{eq:uniformunbiasedfillers}
j(u^*c,iu,x(u\times \I)', y(u\times \I)) = j(c,i,x,y)\circ (u\times \I).
\end{equation}

\begin{definition}\label{def:unbiasedfibstructure} A \emph{uniform, unbiased fibration structure} on a map 
\[
f: Y\ra X
\]
 is a choice of fillers $j(c,i,x,y)$ as in \eqref{diagram:basicunbiasedfillers} satisfying \eqref{eq:uniformunbiasedfillers} for all cubical maps $u : \I^m\to\I^n$.
\end{definition}

In these terms, we have the following analogue of corollary \ref{cor:uniformfibstructequiv}.

\begin{proposition}\label{prop:equivfibstruc} For any object $X$ in \cSet\ the following are equivalent:
\begin{enumerate}
\item $X$ is an unbiased fibrant object in the sense of Definition \ref{def:unbiasedfibration}: the canonical map $\delta\Rightarrow X : X^\I\times \I \ra X\times \I$ is a trivial fibration. 

\item $X$ has the right lifting property with respect to all generating unbiased trivial cofibrations,
\[
(\mathcal{C}\otimes\delta)\, \pitchfork\,X.
\]

\item $X$ has a uniform, unbiased fibration structure in the sense of Definition \ref{def:unbiasedfibstructure}.
\end{enumerate}
\end{proposition}

\begin{proof}
The equivalence between (1) and (2) is proposition \ref{prop:fibiffrlp}.  
So assume (1).
Then in $\cSet/_\I$, the evaluation at $\delta : 1\to\II$,
\[
(\I^*X)^\delta : (\I^*X)^\II \too X
\]
is a trivial fibration.  By Proposition \ref{prop:uniformstructequivrelative} it therefore 
has a uniform filling structure with respect to all basic cofibrations $c :C\mono \I^n$ over $\I$.  Transposing by the $\otimes\dashv\,\Rightarrow$ adjunction and unwinding then gives exactly a uniform fibration structure on $X$.
\end{proof}

A statement analogous to the foregoing also holds for maps $f:Y\ra X$ in place of objects $X$.  Indeed, as before, we have the following sharper formulation.

\begin{corollary}
Uniform, unbiased fibration structures on a map $f : Y\ra X$ correspond uniquely to relative $+$-algebra structures on the map $(\delta\Rightarrow{f})$ (cf.\ definition \ref{def:unbiasedfibration}),
\[
(\delta\Rightarrow{f}) : Y^I\times \I \too (X^I \times \I)\times_{(X\times \I)} (Y\times\I)\,.
\]
\end{corollary}

\subsection*{Factorization}\label{subsec:FWFS}

\begin{definition}\label{def:FibWFSclasses}
Summarizing the foregoing definitions, we have the following classes of maps:

\begin{itemize}
\item  The \emph{generating unbiased trivial cofibrations} were determined in \eqref{eq:generatingtrivialcofibrations} as
\begin{equation}\label{eq:genunbiasedTCof}
\mathcal{C}\otimes\delta = \{c \otimes_i \delta : D \mono Z\times\I\ |\  c : C\mono Z,\, i : Z \ra \I\}\,,
\end{equation}
where $D = \big(Z +_C (C\times\I)\big)$ and the pushout-product $c\otimes_i \delta$ has the form
\begin{equation}\label{diagram:unbiasedbasicpushoutproduct2}
\xymatrix{
C \ar@{>->}[d]_{c} \ar[r]^{\gph{ic}} & C\times \I \ar[d] \ar@/^4ex/ [rdd]^{c\times\I}\\
Z \ar@/_4ex/ [rrd]_{\gph{ic}} \ar[r] &  Z +_C (C\times\I) \ar@{.>}[rd]_{c\, \otimes_i \delta} \\
&& Z\times\I
}
\end{equation}
for any cofibration $c : C\mono Z$ and indexing map $i: Z \ra \I$.

\item The class $\mathcal{F}$ of \emph{unbiased fibrations}, which can be characterized as the right-lifting class of the generating unbiased trivial cofibrations,
\[
(\mathcal{C}\otimes\delta)^\pitchfork\, =\,\mathcal{F}.
\]

\item The class of \emph{unbiased trivial cofibrations} is then defined to be left-lifting class of the fibrations,
\[
\mathsf{TCof}\, =\, ^{\pitchfork}\mathcal{F}.
\]
\end{itemize}
\end{definition}

It follows that the classes $\mathsf{TCof}$ and $\mathcal{F}$ are closed under retracts and are mutually weakly orthogonal, $\mathsf{TCof}\ {\pitchfork}\ \mathcal{F}$.
Thus in order to have a weak factorization system $(\mathsf{TCof}, \mathcal{F})$ it just remains to show the following.

\begin{lemma}\label{lemma:factorization}
Every map $f: X\ra Y$ in \cSet\ can be factored as $f  = p\circ i$,
\begin{equation}
\xymatrix{
X \ar[rd]_{f} \ar@{>->}[r]^i & X'\ar@{->>}[d]^p\\
& Y
}
\end{equation}
with $i: X\mono X'$ an unbiased trivial cofibration and $p: X'\onto Y$ an unbiased fibration.
\end{lemma}
\begin{proof}
We can use a standard argument (the ``algebraic small object argument'', cf.~\cite{garner:small-object-argument,riehl-algebraic-model}), which can be further simplified using the fact that the codomains of the basic trivial cofibrations $c\, \otimes_i \,\delta : B \mono \I^{n+1}$ are not just representable, but \emph{tiny} in the sense of Proposition \ref{prop:Itiny}, and the domains are not merely ``small'', but \emph{finitely presented}.  The reader is referred to \cite{Awodey:cubical-model} for details in a similar case.
\end{proof}

\begin{remark}\label{remak:constructive}
The proof in \ibid\ actually produces a stronger result than we need, namely an \emph{algebraic} weak factorization system.  This follows from the small generating \emph{category} $\mathsf{BCof}\otimes \delta$  of basic unbiased trivial cofibrations (and pullback squares of the form on the left in \eqref{diagram:basicunbiasedfillersuniformity}).  The relationship between this stronger condition and the \emph{classifying types} used in Section \ref{sec:U} is studied in \cite{Swan:Wtypes}, which also gives an even more ``constructive'' proof of the factorization Lemma \ref{lemma:factorization}, not requiring quotients, exactness, or impredicativity.  With this modification, the present approach can also be used in a \emph{quasi}topos, as occurs in \eg\ realizability and sheaves.
\end{remark}

\begin{proposition}\label{prop:fibrationwfs}
There is a weak factorization system on the category \cSet\ in which the right maps are the unbiased fibrations and the left maps are the unbiased trivial cofibrations, both as specified in definition \ref{def:FibWFSclasses}.  This will be called the \emph{(unbiased) fibration weak factorization system}.
\end{proposition}

Hereafter, unless otherwise stated, all fibrations in $\cSet$ are assumed to be  unbiased.

\section{The weak equivalences}\label{sec:weq}

Our approach to proving that the classes $\CC$ and $\FF$ of cofibrations and fibrations, from Sections \ref{sec:cofibrations} and \ref{sec:fibrations}, determine a model structure will be to first identify a \emph{premodel structure} in the sense of \cite{Barton}, and then turn to the question of the 3-for-2 property for the resulting weak equivalences.

\begin{definition}[Weak equivalence]
A map $f: X\ra Y$ in \cSet\ is a \emph{weak equivalence} if it can be factored as $f  = g\circ h$,
\begin{equation*}
\xymatrix{
X \ar[rd]_{f} \ar[r]^h & W\ar[d]^g\\
& Y
}
\end{equation*}
with $h \pitchfork \FF$ and $\CC\pitchfork g$.  Accordingly, let 
\begin{align*}
\mathcal{W} &= \TFib\circ\TCof \\
	&= \{f: X\ra Y\, |\ f = g\circ h\ \text{for some $g\in\mathsf{TFib}$ and $h\in\mathsf{TCof}$} \}
\end{align*}
 be the class of weak equivalences.
\end{definition}

Observe first that every trivial fibration $f\in\TFib = \mathcal{C}^\pitchfork$ is indeed a fibration, because the generating trivial cofibrations $c\otimes_i\delta$ are cofibrations.  Moreover, every trivial fibration $f : X\to Y$ is also a weak equivalence $f = f\circ 1_X$, since the identity map $1_X$ is (trivially) a trivial cofibration $\TCof =\, ^{\pitchfork}\mathcal{F}$. Thus we have
\[
\mathsf{TFib} \subseteq (\mathcal{F} \cap \mathcal{W}).
\]
Similarly, because $\TFib\subseteq\FF$, we have $\TCof\subseteq\CC$.  Moreover, since identity maps are also trivial fibrations we have $\TCof\subseteq \TFib\circ\TCof  =\WW$.  Thus we also have
\[
\mathsf{TCof} \subseteq (\mathcal{C} \cap \mathcal{W}).
\]

\begin{lemma}
$(\mathcal{C} \cap \mathcal{W})  \subseteq \mathsf{TCof}.$
\end{lemma}
\begin{proof}
Let $c : A\mono B$ be a cofibration with a factorization 
\[
c = tf\circ tc : A \ra W\ra B
\]
where $tc\in\mathsf{TCof}$ and $tf\in\mathsf{TFib}$.  Let $f:X\onto Y$ be a fibration and consider a commutative diagram,
\begin{equation*}
\xymatrix{
A \ar@{>->}[d]_{c} \ar[r]^-{x}  & X \ar@{>>}[d]^{f} \\
B \ar[r]_{y} &  Y.
}
\end{equation*}
Inserting the factorization of $c$, from $tc \pitchfork f$ we obtain $j : W\ra X$ as indicated, with $j\circ tc = x$ and $f\circ j = y\circ tf$.
\begin{equation*}
\xymatrix{
A \ar@{>->}[dd]_{c} \ar[rd]_{tc} \ar[rr]^-{x}  && X \ar@{>>}[dd]^{f} \\
& W \ar[ld]_{tf} \ar@{.>}[ru]_{j} & \\
B \ar[rr]_{y} &&  Y.
}
\end{equation*}
Moreover, since $c\pitchfork tf$ there is an $i : B \ra W$ as indicated, with $i\circ c = tc$ and $tf\circ i = 1_B$.
\begin{equation*}
\xymatrix{
A \ar@{>->}[dd]_{c} \ar[rd]_{tc} \ar[rr]^-{x}  && X \ar@{>>}[dd]^{f} \\
& W \ar[ld]_{tf} \ar@{.>}[ru]_{j} & \\
B \ar[rr]_{y} \ar@{.>}@/_2ex/[ru]_--{i} &&  Y.
}
\end{equation*}
Let $k = j\circ i$. Then $k \circ c = j\circ i \circ c = j \circ tc = x$, and $f \circ k = f\circ j\circ i = y\circ tf\circ i = y$.
\end{proof}

The proof of the following is exactly dual.
\begin{lemma}
$(\mathcal{F} \cap \mathcal{W})  \subseteq \mathsf{TFib}.$
\end{lemma}

\begin{proposition}\label{prop:FWC}
The three classes of maps $\mathcal{C}, \mathcal{W}, \mathcal{F}$ in \cSet\ constitute a \emph{premodel structure} in the sense of \cite{Barton}.   In particular,  we have 
\begin{align*}
\mathcal{F}\cap\mathcal{W} &= \mathsf{TFib}, \\
\mathcal{C}\cap\mathcal{W} &= \mathsf{TCof},
\end{align*}
and therefore two interlocking weak factorization systems:
\[
(\mathcal{C},\, \mathcal{W}\cap\mathcal{F})\ \ ,\ \ (\mathcal{C}\cap\mathcal{W},\, \mathcal{F}).
\]
\end{proposition}

It now ``only'' remains to show that the weak equivalences $\WW$ satisfy the 3-for-2 axiom from Definition \ref{def:qmsviaJT} in order to verify that $(\mathcal{C}, \mathcal{W}, \mathcal{F})$ is a model structure.  Perhaps surprisingly, this will occupy the remainder of these lectures!  We shall follow roughly the approach of \cite{JT:notes}: the weak equivalences between fibrant objects are shown to be the usual \emph{homotopy equivalences}, which evidently  satisfy 3-for-2.   So we reduce to this case using the fact that $K^X$ is fibrant whenever $K$ is.  It suffices, namely, to show that the weak equivalences are those maps $w : X\to Y$  that induce homotopy equivalences $K^w : K^Y \simeq K^X$ for fibrant $K$.   Such maps are termed \emph{weak homotopy equivalences}  (Definition \ref{def:WHE}), and our task will therefore be to show that a map is a weak equivalence if and only if it is a weak homotopy equivalence.

\subsection*{Homotopy equivalence.}

\begin{definition}[Homotopy]\label{homotopy}
A \emph{homotopy} $\vartheta : f \sim g$ between maps $f, g: X\rightrightarrows Y$  is a map,
\[
\vartheta : \I\times{X} \too Y,
\]
such that $\vartheta \circ \iota_0 = f$ and $\vartheta \circ \iota_1 = g$, 
\begin{equation}\label{diagram:defhomotopy}
\xymatrix{
X \ar[r]^-{\iota_0} \ar[rd]_f & \I\times\!{X} \ar[d]^-{\vartheta} & X, \ar[l]_-{\iota_1} \ar[ld]^g \\
& Y &
}
\end{equation}
where $\iota_0, \iota_1$  are the canonical inclusions into the ends of the cylinder,
\[
\xymatrix{
\iota_\epsilon : X \cong 1\times X \ar[rr]^-{\delta_\epsilon\times X} && \I\times X\,,} \qquad \epsilon = 0,1 .
\]
\end{definition}

Note that each of the inclusions $\iota_\epsilon : X \cof  \I\times X$ is a cofibration, as is their join $X + X \cof  \I\times X$, by Remark \ref{rem:somecofibs}.  

\begin{proposition}\label{prop:homotopyintofibequivrel}
The relation of homotopy $f \sim g$ between maps $f, g: X\rightrightarrows Y$ is preserved by pre- and post-composition. If $Y$ is fibrant, then $f \sim g$ is an equivalence relation.
\end{proposition}
\begin{proof}
Inspecting \eqref{diagram:defhomotopy}, preservation of $f\sim g$ under post-composing with any $h : Y\to Z$ is obvious: we have $h\circ\vartheta : h\circ\,f \sim h\circ\,g$.  Now observe that a homotopy $f\stackrel{\vartheta}{\sim} g : X\times\I \ra Y$ determines a (unique) path $\tilde{\vartheta} : \I\ra Y^X$ in the function space, with endpoints $\vartheta_0 = \vartheta\circ\delta_0 = \tilde{f}: 1\ra Y^X$ and $\vartheta_1 = \vartheta\circ\delta_1 = \tilde{g}$.  
Precomposing maps $f, g: X\rightrightarrows Y$ with any $e : W \to X$ is induced by post-composing  $\tilde{f}, \tilde{g}: 1\ra Y^X$ with the map $Y^e : Y^X \to Y^W$, which then also takes the path $\tilde{\vartheta} : \I \ra Y^X$ to a path $\tilde{\varphi} = Y^e \circ \tilde{\vartheta} : \I \ra Y^W$ corresponding to a  (unique) homotopy $\varphi : f\circ e \sim g\circ e$.

Now note that $Y^X$ is fibrant if $Y$ is fibrant, since the generating trivial cofibrations $c\times_i \delta$ are preserved by the functor $X\times(-)$.  So we can use ``box-filling'' in $Y^X$ to verify the claimed equivalence relation.    
\begin{itemize}
\item Reflexivity $f\sim f$ is witnessed by the homotopy $\rho:\I \ra 1 \stackrel{f}{\ra} Y^X$.  
\item For symmetry $f\sim g\Rightarrow g\sim f$ take $\vartheta : \I\ra Y^X$ with $\vartheta_0 = f$ and $\vartheta_1 = g$ and we want to build $\vartheta' : \I\ra Y^X$ with $\vartheta'_0 = g$ and $\vartheta'_1 = f$. Take an open 2-box in $Y^X$ of the following form.
\begin{equation*}
\xymatrix{
g  & f  \\
f \ar[u]^{\vartheta} \ar[r]_\rho & f \ar[u]_\rho
}
\end{equation*}
This box is a map $b : \I+_1 \I +_1 \I \ra Y^X$ with the indicated components, and it has a filler $c : \I\times \I \ra Y^X$, i.e.\ an extension along the canonical map $\I+_1 \I +_1 \I \mono \I\times\I$, which is a trivial cofibration of the form $\del{\I} \otimes \delta_0$.  Let $t : \I\ra \I\times\I$ be the top face of the 2-cube (the bipointed map $\{0, x_1, x_2, 1\}\to \{0, x, 1\}$ that is constantly $1$).  We can set $\vartheta' = c\circ t : \I \ra Y^X$ to get a homotopy $\vartheta' : \I\ra Y^X$ with $\vartheta'_0 = g$ and $\vartheta'_1 = f$ as required.

\item For transitivity, $f\stackrel{\vartheta}{\sim} g\,,\, g\stackrel{\varphi}{\sim} h\Rightarrow f\sim h$, an analogous  construction will fill the open box:
\begin{equation*}
\xymatrix{
f  & h  \\
f \ar[u]^{\rho} \ar[r]_\vartheta & g \ar[u]_\varphi
}
\end{equation*}
\end{itemize}
\end{proof}

We then have the usual definition of homotopy equivalence: 
\begin{definition}[Homotopy equivalence]\label{def:homotopyequivalence} 
A \emph{homotopy equivalence} is a map $f : X\to Y$ together with a map $g: Y\ra X$ and homotopies $\vartheta : 1_X \sim g\circ f$ and $\varphi : 1_Y\sim f\circ g$.  We call $g$ a \emph{quasi-inverse} of $f$.
\end{definition}
Since these maps clearly compose and come with quasi-inverses, the following is then immediate.
\begin{lemma}\label{lemma:HE342}
The  homotopy equivalences satisfy the 3-for-2 condition. 
\end{lemma}

\begin{lemma}\label{lem:TFibisHE}
A fibration that is a weak equivalence is a homotopy equivalence.
\end{lemma}
\begin{proof}
Any trivial fibration $f : X \onto Y$ has a section $s: Y\to X$ by Corollary \ref{cor:plusalgprops}.
Consider the following lifting problem:
\begin{equation*}
\xymatrix{
X+X \ar@{>->}[d]_{[\iota_0, \iota_1]} \ar[r]^-{[sf, 1]}  & X\ar@{>>}[d]^{f} \\
\I\times X \ar[r]_{f\pi_2} & Y
}
\end{equation*}
Since the map on the left is a cofibration, a diagonal filler provides a homotopy $\vartheta : sf \sim 1_X$. 
Thus $f$ is a homotopy equivalence. 
\end{proof}

For the further comparison of the weak equivalences with the homotopy equivalences we need the following.

\subsection*{Weak homotopy equivalence.}

\begin{definition}[Connected components]
The functor 
\[
\pi_0 : \cSet\ra\Set
\]
 is defined on a cubical set $X$ as the coequalizer 
 \[
 X_1\rightrightarrows X_0\ra \pi_0X\,,
 \]
  where the two parallel arrows are the maps $X_{\delta_0}, X_{\delta_1} : X_1 \rightrightarrows X_0$ for the endpoints $\delta_0, \delta_1 : 1 \rightrightarrows \I$.   If $K$ is fibrant, then by the foregoing Proposition \ref{prop:homotopyintofibequivrel}, for any $X$ we have  
  \[
  \pi_0(K^X) = \hom(X,K)/\!\!\sim\,.
  \]
    That is, $\pi_0(K^X)$ is the set  $[X, K]$ of homotopy equivalence classes of maps $X\ra K$.
\end{definition}

\begin{remark}
One can show that in fact $\pi_0X = \varinjlim X_n$ where the colimit is taken over \emph{all} objects $[n]$ in the index category $\op\Box = \B$, rather than just the ``last'' two $[1]\rightrightarrows [0]$. Since the category $\B$ of finite strictly bipointed sets is sifted, 
the functor $\pi_0: \cSet\ra\Set$ preserves finite products. 
\end{remark}

\begin{definition}[Weak homotopy equivalence]\label{def:WHE}
A map $f: X\ra Y$ is called a \emph{weak homotopy equivalence} if for every fibrant object $K$, the canonical map $K^f : K^Y \ra K^X$ is bijective on connected components,  
\[
\pi_0(K^f) : \pi_0(K^Y) \cong \pi_0(K^X)\,.
\]
\end{definition}

\begin{lemma}\label{lemma:HEisWHE}
Every homotopy equivalence is a weak homotopy equivalence.
\end{lemma}
\begin{proof}
Let $f: X\ra Y$ be a homotopy equivalence.  Then $K^f : K^Y \ra K^X$ is also a homotopy equivalence for any $K$, since homotopy respects (post-) composition by all maps.  If $K$ is fibrant, then so is $K^X$ and $\pi_0$ is well defined on homotopy classes of maps, by Proposition \ref{prop:homotopyintofibequivrel}.  It clearly takes homotopy equivalences to isomorphisms of sets, since it identifies homotopic maps.
\end{proof}

\begin{lemma}\label{lemma:WHE342}
The weak homotopy equivalences also satisfy the 3-for-2 condition. 
\end{lemma}
\begin{proof}
This follows by applying the $\Set$-valued functors $\pi_0(K^{(-)})$, for all fibrant objects $K$, and the corresponding fact about bijections of sets.
\end{proof}

In virtue of Lemma \ref{lemma:WHE342} it now suffices to show that a map is a weak equivalence if and only if it is a weak homotopy equivalence.   The following characterization will be useful.

\begin{lemma}\label{lem:WHEunwound}
A map $f : X\ra Y$ is a weak homotopy equivalence just if it satisfies the following two conditions.
\begin{enumerate}
\item For every fibrant object $K$ and every map $x : X \ra K$ there is a map $y:Y\ra K$ such that $y\circ f \sim x$,
\[
\xymatrix{
X \ar[d]_{f} \ar[r]^{x}  & K.\\
Y \ar@{..>}[ru]_{y}^{\sim} &
}
\]
We say that $x$ ``extends along $f$ up to homotopy''.
\item For every fibrant object $K$ and maps $y, y' : Y \ra K$ such that $yf \sim y'f$, there is a homotopy $y\sim y'$,
\[
\xymatrix{
X \ar[d]_{f} \ar[r]  & K^\I\ar[d] \\
Y \ar@{..>}[ru] \ar[r]_-{\langle y,y'\rangle} & K\times K.
}
\]
\end{enumerate}
\end{lemma}
\begin{proof}
Condition (1) says exactly that the internal precomposition map $K^f : K^Y \ra K^X$ is surjective under connected components $\pi_0$, while (2) says just that it is injective under $\pi_0$.
\end{proof}

\begin{lemma}\label{lem:WEimpliesWHE}
Any weak equivalence is a weak homotopy equivalence.
\end{lemma}

\begin{proof}
By Lemma \ref{lem:TFibisHE} and Lemma \ref{lemma:HEisWHE}, a trivial fibration is also a weak homotopy equivalence. So it suffices to consider the trivial cofibrations, since weak homotopy equivalences are closed under composition, by Lemma \ref{lemma:WHE342}.   Thus let $f : X \mono Y$ be a trivial cofibration, and apply Lemma \ref{lem:WHEunwound}: condition (1) is immediate, and (2) follows because $K^\I \fib K\times K$ is a fibration when $K$ is fibrant, since $\del : 1+1 \cof \I$ is a cofibration (by Remark \ref{rem:somecofibs}).
\end{proof}

Our goal is now to show the converse of Lemma \ref{lem:WEimpliesWHE}(2), that a weak homotopy equivalence is a weak equivalence. We shall first restrict attention to maps $f : X\to K$ with a fibrant codomain $K$.  By factoring such maps, we can split into the cases of a fibration and a cofibration.  

\begin{lemma}\label{FibHETFib}
If $K$ is fibrant, then any fibration $f : X \onto K$ that is a homotopy equivalence is a weak equivalence.
\end{lemma}
\begin{proof}
This is a standard argument, which we just sketch.  It suffices to show that any diagram of the form
\begin{equation}\label{diagram:FHEisWE}
\xymatrix{
C \ar@{>->}[d]_{c} \ar[r]^x & X \ar@{>>}[d]^{f} \\
K\ar[r]_{=} & K,
}
\end{equation}
with  $c : C \mono X$  a cofibration, has a diagonal filler, for then $f$ is a trivial fibration.  
Since $f$ is a homotopy equivalence, it has a quasi-inverse $s:K\ra X$ with $\vartheta : fs\sim 1_K$, which we claim can be corrected to a section $s' : K\ra X$. 
Indeed, consider 
\begin{equation*}
\xymatrix{
K \ar@{>->}[d]_{\iota_0} \ar[r]^{s}  & X\ar@{>>}[d]^{f} \\
K\times \I \ar[r]_{\vartheta} \ar@{..>}[ru]^{\vartheta'}& K\\
K \ar@{>->}[u]_{\iota_1} \ar[ru]_=&
}
\end{equation*}
where $\vartheta' $ results from $\iota_0 \pitchfork f$. Let $s' = \vartheta' \iota_1$, so that $\vartheta' : s\sim s'$ and $fs' = 1_K$.

Thus we can assume that $s = s' : K\ra X$ is a section, which fills the diagram \eqref{diagram:FHEisWE} up to a homotopy in the upper triangle.
\begin{equation*}
\xymatrix{
C \ar@{>->}[d]_{c} \ar[r]^x & X \ar@{>>}[d]^{f} \\
K\ar[r]_{=} \ar[ru]_{s}^{\sim} & K
}
\end{equation*}
Now we can correct $s: K\ra X$ to a homotopic $t : K\ra X$ over $f$ by using the homotopy $\varphi : sc\sim x$  to get a map $\varphi : C\ra X^\I$ over $f$.  Since $f$ is a fibration, the projections $p_0, p_1:X^\I \ra X$ over $f$ are trivial fibrations, and so there is a lift $\varphi': K\ra X^\I$ for which $t:= p_1\varphi'$ has $tc= x$ and $ft=1_K$, and so is a filler for \eqref{diagram:FHEisWE}.
\end{proof}

\begin{lemma}\label{lemma:FibWHEfibCodTFib}
If $K$ is fibrant, then any fibration $f : X \onto K$ that is a weak homotopy equivalence is a weak equivalence.
\end{lemma}
\begin{proof}
Since $K$ is fibrant, so is $X$, and since $f$ is a weak homotopy equivalence, by lemma \ref{lem:WHEunwound}(1) there is then a map $s : K\ra X$ and a homotopy $\theta: sf \sim 1_X$.  Postcomposing with $f$ gives a homotopy $f\vartheta: fsf \sim f$, forming the outer commutative square in
\[
\xymatrix{
X \ar[d]_{f} \ar[r]^{f\vartheta}  & K^\I \ar[d] \\
K \ar@{..>}[ru]_{\varphi} \ar[r]_-{\langle fs, 1_K\rangle} & K\times K.
}
\]
By lemma \ref{lem:WHEunwound}(2) there is a diagonal filler $\varphi : fs\sim 1_K$, and so $f$ is a homotopy equivalence. Now apply lemma \ref{FibHETFib}.
\end{proof}

We now have the following.

\begin{proposition}\label{prop:CofbetweenFibWEWHEHE}
If $A$ and $K$ are both fibrant, then for any cofibration $c : A \cof K$ the following are equivalent.
\begin{enumerate}
\item $c : A \cof K$ is a weak equivalence.
\item $c : A \cof K$ is a homotopy equivalence.
\item $c : A \cof K$ is a weak homotopy equivalence.
\end{enumerate}
\end{proposition}
\begin{proof}
Suppose (1), so $c : A \cof K$ is a trivial cofibration.  Then since $A$ is fibrant, it has a retraction $r : K\to A$.
\begin{equation*}
\xymatrix{
A \ar@{>->}[d]_{c} \ar[r]^= & {A}. \\
{K}\ar[ru]_r & 
}
\end{equation*}
Since $K$ is fibrant, $K^\I \to K\times K$ is a fibration. So the following has a diagonal filler, which is a homotopy $1_K \sim cr$.
\begin{equation*}
\xymatrix{
A \ar@{>->}[d]_{c} \ar[r]^{c}  & K \ar[r]^{K^!} & K^\I  \ar[d]^{\pair{K^{d_0},K^{d_1}}} \\
K \ar[rr]_{\pair{1_K, cr}} \ar@{..>}[rru]_{\vartheta} &&  K\times K
}
\end{equation*}

$(2)\Rightarrow (3)$ is Lemma \ref{lemma:HEisWHE}.

Suppose (3), that $c : A \cof K$ is a weak homotopy equivalence.  Factor $c = f\circ tc$ with a trivial cofibration $tc : A \cof C$ followed by a  fibration $f : C \fib K$.  By parts (1) and (2), $tc : A \cof C$ is then a weak homotopy equivalence. By 3-for-2 for weak homotopy equivalences, Lemma \ref{lemma:WHE342}, $f : C \fib K$ is then also a weak homotopy equivalence. By Lemma \ref{lemma:FibWHEfibCodTFib}, $f : C \fib K$ is then a weak equivalence.
\end{proof}

\begin{proposition}\label{prop:fibobjWEHEWHE} For fibrations $f : X \fib K$ with fibrant codomain $K$, all three concepts coincide: weak equivalences, weak homotopy equivalences, and homotopy equivalences.
\end{proposition}

\begin{proof}
Let $K$ be fibrant and suppose that $f:X\fib K$ is a weak homotopy equivalence.   Then it is a weak equivalence by Lemma \ref{lemma:FibWHEfibCodTFib}.  By Lemma \ref{lem:WEimpliesWHE} any fibration weak equivalence is a homotopy equivalence, and by 
Lemma \ref{lemma:HEisWHE} any homotopy equivalence is a weak homotopy equivalence.
\end{proof}

\begin{corollary} 
For all maps $f : X \to Y$ between fibrant objects $X$ and $Y$, all three concepts coincide: weak equivalence, weak homotopy equivalence, and homotopy equivalence.
\end{corollary}

\begin{proof}
Let $X$ and $Y$ be fibrant and factor $f = tf\circ tc$ with a trivial cofibration $tc : X \cof F$ followed by a trivial fibration $tf : F \fib Y$.
 Then by Proposition \ref{prop:CofbetweenFibWEWHEHE}, $tc : X \cof F$ is a homotopy equivalence, and by Proposition \ref{prop:fibobjWEHEWHE} so is $tf : F \fib Y$, thus $f = tf\circ tc$  is a homotopy equivalence. Again by Lemma \ref{lemma:HEisWHE}, any homotopy equivalence is a weak homotopy equivalence, and weak homotopy equivalence between fibrant objects is clearly a weak equivalence, by factoring and using the foregoing Propositions \ref{prop:CofbetweenFibWEWHEHE} and \ref{prop:fibobjWEHEWHE}.
\end{proof}

\begin{lemma}\label{CofWHEfibCodTCof}
If $K$ is fibrant, then any cofibration $c : A \mono K$ that is a weak homotopy equivalence is a weak equivalence.
\end{lemma}
\begin{proof}
Let $c : A \mono K$ be a cofibration weak homotopy equivalence and factor it into a trivial cofibration $i : A\mono Z$ followed by a fibration $p: Z\onto K$.  By lemma \ref{lem:WHEunwound}, any trivial cofibration is clearly a weak homotopy equivalence.  So both 
$c$ and $i$ are weak homotopy equivalences, and therefore so is $p$ by 3-for-2 for weak homotopy equivalences.   Since $K$ is fibrant, $p$ is a trivial fibration by lemma \ref{lemma:FibWHEfibCodTFib}, and thus $c$ is a weak equivalence.  
\end{proof}

It now follows that a weak homotopy equivalence $f : X \to K$ with a fibrant codomain is a weak equivalence. To eliminate the condition on the codomain we use the following lemma due to D.-C.~Cisinski \cite{cisinski-asterisque}.

\begin{lemma}\label{lemma:CofWHEiffFibLift}
A cofibration $ c : A \mono B$ weak homotopy equivalence lifts against  any fibration $f : Y\onto K$ with fibrant codomain.
\end{lemma}
\begin{proof}
Let $c : A\mono B$ be a cofibration weak homotopy equivalence and  $f : Y\onto K$ a fibration with fibrant codomain $K$, and consider a lifting problem
\begin{equation*}
\xymatrix{
A \ar@{>->}[d]_{c} \ar[r]^-{a}  & Y \ar@{>>}[d]^{f} \\
B \ar[r]_{b} &  K.
}
\end{equation*}
Let $\eta : B\mono B'$ be a fibrant replacement of $B$, since $K$ is fibrant, $b$ extends along $\eta$ to give $b' : B'\ra K$ as shown below.   
\begin{equation*}
\xymatrix{
A \ar@{>->}[d]_{c} \ar[r]^-{a}  & Y \ar@{>>}[d]^{f} \\
B \ar[r]_{b} \ar[d]_\eta &  K\\
B' \ar[ru]_{b'} & 
}
\end{equation*}
Since $\eta$ is a trivial cofibration, it is a weak homotopy equivalence. So the composite $\eta c$ is also a weak homotopy equivalence.  But since $B'$ is fibrant, $\eta c$ is then a trivial cofibration by lemma \ref{CofWHEfibCodTCof}.  Thus there is a lift $j : B'\ra Y$, and therefore also one $k = j\eta : B\ra Y$. 
\end{proof}

To complete the proof that a weak homotopy equivalence is a weak equivalence, we shall make use of the following \emph{fibration extension property}, the proof of which is deferred to section \ref{sec:FEP}.

\begin{definition}[Fibration extension property]\label{def:fibextreplace}
For any fibration $ f : Y \onto X$ and any trivial cofibration $\eta: X\ra X'$,
there is a fibration $f' : Y' \onto X'$ that pulls back to $f$ along $\eta$, as shown below.
\begin{equation}\label{diagram:FEP}
\xymatrix{
Y \ar@{->>}[d]_{f} \ar[r]  \pbcorner & Y' \ar@{>>}[d]^{f'} \\
X \ar@{>->}[r]_{\eta} &  X'
}
\end{equation}
\end{definition}

\begin{lemma}\label{lemma:CofWEiffFibLift}
Assuming the fibration extension property, a cofibration that lifts against every fibration $f : Y\onto K$ with fibrant codomain is a weak equivalence.
\end{lemma}

\begin{proof}
Let $c : A\mono B$ be a cofibration and consider a lifting problem against an arbitrary fibration $f: Y\onto X$,
\begin{equation}\label{diagram:CofWHEiffFibLift1}
\xymatrix{
A \ar@{>->}[d]_{c} \ar[r]^-{a}  & Y \ar@{>>}[d]^{f} \\
B \ar[r]_{b} &  X.
}
\end{equation}
Let $\eta: X\ra X'$ be a fibrant replacement, so $\eta$ is a trivial cofibration and $X'$ is fibrant. 
By the fibration extension property of definition \ref{def:fibextreplace}, there is a fibration $f' : Y' \onto X'$ such that $f$ is a pullback of $f'$ along $\eta$. So we can extend diagram \eqref{diagram:CofWHEiffFibLift1} to obtain the following, in which the righthand square is a pullback.
\begin{equation}\label{diagram:CofWHEiffFibLift2}
\xymatrix{
A \ar@{>->}[d]_{c} \ar[r]^-{a}  & Y \ar@{>>}[d]^{f} {\pbcorner} \ar[r]^{y} & Y' \ar@{>>}[d]^{f'} \\
B \ar[r]_{b} &  X  \ar[r]_\eta &  X'.
}
\end{equation}
By assumption, there is a lift $j' : B\ra Y'$ with $f' j' = \eta b$ and $j'c = yb$.  Therefore, since $f$ is a pullback, there is a map $j : B\ra Y$ with $fj = b$ and $y j = j'$.  
\begin{equation}\label{diagram:CofWHEiffFibLift3}
\xymatrix{
A \ar@{>->}[d]_{c} \ar[r]^-{a}  & Y \ar@{>>}[d]_<<<<{f} {\pbcorner} \ar[r]^{y} & Y' \ar@{>>}[d]^{f'} \\
B \ar[r]_{b} \ar@{..>}[ru]^{j} \ar@{..>}[rru]_>>>>{j'} &  X  \ar[r]_\eta &  X'.
}
\end{equation}
Thus $yjc = j'c = ya$.  But as a trivial cofibration, $\eta$ is monic, and as a pullback of $\eta$, $y$ is also monic. So $jc=a$.
\end{proof}

\begin{corollary}\label{cor:CofWHEtoWE}
Assuming the fibration extension property,  
\begin{enumerate}
\item a cofibration $c : A \mono B$ weak homotopy equivalence is a weak equivalence,
\item a fibration $ f : Y \onto X$ weak homotopy equivalence is a weak equivalence.
\end{enumerate}
\end{corollary}

\begin{proof}
(1) follows immediately by combining the previous lemmas \ref{lemma:CofWHEiffFibLift} and \ref{lemma:CofWEiffFibLift}.

For (2), factor $ f : Y \onto X$ into a cofibration $i : Y\mono Z$ followed by a trivial fibration $p: Z\onto X$.  Then $f$ is itself a trivial fibration if $i\pitchfork f$, for then it is a retract of $p$.  Since $p$ is a trivial fibration, it is a weak homotopy equivalence by Lemma \ref{lem:WEimpliesWHE}.  Since $f$ is also a weak homotopy equivalence, so is $i$ by Lemma \ref{lemma:WHE342}.  Thus $i$ is a trivial cofibration by (1). Since $f$ is a fibration, $i\pitchfork f$ as required.
\end{proof}

We have now shown:

\begin{proposition}\label{prop:WHEiffWE}
Assuming the fibration extension property, a map $ f : X \to Y$ is a weak homotopy equivalence if and only if it is a weak equivalence.  The weak equivalences $\mathcal{W}$ therefore satisfy the 3-for-2 condition.
\end{proposition}

The results of this section are summarized in the following.

\begin{theorem}\label{theorem:QMSmodFEP}
Assume the fibration weak factorization system of Definition \ref{def:FibWFSclasses} satisfies the fibration extension property of Definition \ref{def:fibextreplace} (as will be shown  in Corollary \ref{cor:FEP}). 
Then the weak equivalences $\WW$ have the 3-for-2 property, and so by Proposition \ref{prop:FWC}, the classes $(\CC,\WW,\FF)$ form a Quillen model structure.
The weak equivalences $\WW$  are the \emph{weak homotopy equivalences}: those maps $f: X\ra Y$ for which $K^f : K^Y \to K^X$ is bijective on connected components whenever $K$ is fibrant.
\end{theorem}

The proof of the fibration extension property will be given in Section \ref{sec:FEP}.  It uses the equivalence extension property (Section \ref{sec:EEP}), a universal fibration (Section \ref{sec:U}), and the Frobenius condition (Section \ref{sec:Frobenius}), to which we now turn.

\section{The Frobenius condition}\label{sec:Frobenius}

In this section, we show that the (unbiased) fibration weak factorization system from Section \ref{sec:fibrations} satisfies what has been called the \emph{Frobenius condition}: the left maps are stable under pullback along the right maps (see \cite{van-den-berg-garner}).  This will imply the \emph{right properness} of our model structure: the weak equivalences are preserved by pullback along fibrations.  In the present setting, it then follows that the entire model structure is stable under such a base change.  The Frobenius condition will be used in the proof of the equivalence extension property in Section \ref{sec:EEP}.  

A proof of Frobenius in the related setting of cubical sets \emph{with connections} was given in \cite{GS} using conventional, functorial methods.  By contrast, the type theoretic approach of \cite{CCHM:2018ctt} provides a proof that is much more direct, and can also be modified to work without connections (as in \cite{ABCHFL}).  That approach proves the dual fact that the \emph{pushforward} operation, which is right adjoint to pullback and always exists in a topos, preserves fibrations when applied along a fibration.  This corresponds to the type-theoretic $\Pi$-formation rule, and the proof given in \opcit\ is entirely in type theory.  It also employs a reduction of box filling (in all dimensions) to an apparently weaker condition of \emph{Kan composition} (in all dimensions), which merely ``puts a lid on" the open box, rather than filling it.  This aspect of the type theoretic proof can also be described functorially, but is not used in the proof given here, and will therefore not be discussed further (see \cite{LOPS18} for a description of Kan composition with connections, and \cite{A:composiiton} for the same without connections).

Our proof takes the approach that was used to determine the unbiased fibrations, namely we first establish the result in the \emph{biased but generic} setting, and then transfer it to the unbiased setting by pulling back along the base change $\cSet\to\cSet/_\I$. We first give the second step as a conditional statement.

\begin{proposition}\label{prop:biasedFrobimpliesunbiasedFrob}
Suppose the $\delta$-biased fibrations in $\cSet/_\I$ satisfy the Frobenius condition. Then the unbiased fibrations in $\cSet$ also satisfy the Frobenius condition.
\end{proposition}

\begin{proof}
This follows almost immeditely from the fact that the pullback functor $\I^* : \cSet \to \cSet/_\I$ preserves the locally cartesian closed structure, takes unbiased fibrations to $\delta$-biased ones, and reflects $\delta$-biased fibrations to unbiased ones.  In detail, let unbiased fibrations $B \fib A$ and $A \fib X$ in $\cSet$ be given, and we wish to find $C\fib X$ and $e : A\times_X C \to B$ over $A$, universal in the way recalled in the diagram below.
\begin{equation}\label{diagram:biasedFrobenius2unbiasedFrobenius}
\begin{tikzcd}
A \times_X C \ar[d, dotted, "e"] \ar[r, dotted] & C \ar[dd, two heads, dotted]\\
B \ar[d, two heads] &   \\
A \ar[r, two heads]  & X 
\end{tikzcd}
\end{equation}
Take the pushforward $C := A_*B \to X$, and its associated map $e : A\times_X C \to B$, in the locally cartesian closed category $\cSet$.  Since fibrations are stable under (all) pullbacks, it then suffices to show that $C\to X$ is a fibration.  

By definition, $C\to X$ is an unbiased fibration in $\cSet$ just in case the base change $\I^*C \to \I^*X$ is a $\delta$-biased fibration in the slice category $\cSet/_\I$.  Since the pullback functor $\I^* : \cSet \to \cSet/_\I$ preserves all lcc structure, over $\I^*X$ we have an iso,
\[
\I^*C = \I^*(A_*B) \cong (\I^*A)_*\I^*B\,,
\]
where the pushforward $(\I^*A)_*\I^*B$ is taken in the topos $\cSet/_\I$.  But $\I^*B \to \I^*A$ and $\I^*A \to \I^*X$ are $\delta$-biased fibrations in $\cSet/_\I$ because $B \to A$ and $A \to X$ were assumed to be unbiased fibrations in $\cSet$.  Since we are assuming the Frobenius condition for $\delta$-biased fibrations in $\cSet/_\I$, the pushforward $\I^*C \cong (\I^*A)_*\I^*B \to \I^*X$ is also a $\delta$-biased fibration, as required.
\end{proof}
  
\subsection*{Frobenius for biased fibrations.}

The results proved in this section will be applied to the slice category $\cSet/_\I$ and the generic point $\delta : 1 \ra \I = \I^*\I$, but nothing depends on this particular case, and so we shall write simply $\delta : 1\to \I$ for a chosen pointed object in an arbitrary topos $\EE$.  (Indeed, in this section $\EE$ may even be just a locally cartesian closed category with a class of cofibrations in the sense of Appendix A.)

Recall from Definition \ref{def:unbiasedfibration} that a map $f:A\ra X$  is a $\delta$-biased fibration just if the map $\delta \Rightarrow f$ admits a relative +-algebra structure, and is therefore a trivial fibration.  The definition of the pullback-hom $\pbh{\delta}{f}$ is recalled below.
\begin{equation}\label{diagram:frobenius1}
\xymatrix{
A^\I \ar@/_3ex/ [rdd]_{f^\I} \ar@{.>}[rd]^{\delta\Rightarrow{f}} \ar@/^3ex/ [rrd]^{A^{\delta}}  && \\
& X^\I \times_{X} A \ar[d] \ar[r] & A \ar[d]^{f} \\
& X^\I \ar[r]_{X^{\delta}} &  X
}
\end{equation}
Let us write this condition schematically as follows:
\begin{equation}\label{diagram:frobenius2}
\xymatrix{
A^\I \ar[r]|| \ar[r]  & A_\epsilon\ar[d] \pbcorner \ar[r] & A \ar[d]^f \\
& X^\I \ar[r]_{\epsilon} &  X
}
\end{equation}
where $\epsilon = X^{\delta}$ and $A_\epsilon = X^\I \times_{X} A$, and the struck-through arrow indicates that it admits a +-algebra structure.

\begin{lemma}\label{lemma:fibrationspullback}
Let  $A \ra X$ be a $\delta$-biased fibration and $t: Y\ra X$ any map, then the pullback $t^*A \ra Y$ is also a $\delta$-biased  fibration.
\end{lemma}
\begin{proof}
This is clear from the fact that the $\delta$-biased fibrations can be made into the right class of a weak factorization system (by reasoning analogous to that for Proposition \ref{prop:fibrationwfs}), but it will be useful to see how the structure indicated in \eqref{diagram:frobenius1} is itself stable under pullback.  Indeed, consider the following commutative diagram, in which the front face of the cube is the pullback in question, and the right and left sides are the respective versions of the construction in \eqref{diagram:frobenius1}.
\begin{equation}\label{diagram:fibrationspullback}
\xymatrix{
 (t^*A)^\I \ar[rr] \ar[d] && A^\I \ar[d] & \\
  (t^*A)_\epsilon \ar[dd] \ar[rd] \ar@{.>}[rr] && A_\epsilon \ar[dd] \ar[rd] & \\
  & t^*A \ar[rr] \ar[dd]  && A \ar[dd]  \\
Y^\I \ar[rd]_\epsilon \ar[rr] & &  X^\I \ar[rd]_\epsilon & \\
 & Y \ar[rr]_t && X 
 }
\end{equation}
The rear square of solid arrows is the image of the front face under the pathobject functor and is therefore also a pullback. The base commutes by the naturality of the maps $\epsilon$, as does a corresponding top square involving further such $\epsilon$'s not shown.  Note that these naturality squares need not be pullbacks, but the vertical squares on the sides are, by construction.  It follows that there is a dotted arrow as shown, making the resulting lower rear square commute.  That lower square is then also a pullback, since the other vertical faces of the resulting cube are pullbacks, and thus finally, the upper rear square is also a pullback.  

Now if $A\to X$ is a $\delta$-biased fibration, then $A^\I \to A_\epsilon$ is a trivial fibration, and then so is its pullback $(t^*A)^\I \to (t^*A)_\epsilon$ since relative $+$-algebras are stable under pullback.  Therefore the pullback $t^*A\to Y$ is also a $\delta$-biased fibration.
\end{proof}

\begin{remark}
In this way we can show algebraically that the pullback of a $\delta$-biased fibration is again one by pulling back the structure that makes it so.  In Section \ref{sec:universalfibration}, the pullback stability of the fibration structure will be used in the construction of a universal fibration via a closely related argument.
\end{remark}

\begin{lemma}\label{lemma:fibrationscompose}
Let  $\alpha : A \ra X$ and $\beta: B\ra A$ be $\delta$-biased fibrations, then the composite $\alpha\circ\beta : B \ra X$ is also a $\delta$-biased fibration.
\end{lemma}
\begin{proof}
Again for maps in the right class of a weak factorization system this is immediate.  But let us see how the fibration structures also compose.  We have the following diagram for the fibration structures on $B\ra A$ and $A\ra X$ (with obvious notation).
\begin{equation}\label{diagram:fibcomposition1}
\xymatrix{
B^\I \ar[r]|| \ar[r]  & B_{\epsilon_A} \ar[rr]  \ar[d] \pbcorner & & B \ar[d]\\
&A^\I \ar[r]|| \ar[r] & A_{\epsilon_X} \ar[d] \pbcorner \ar[r] & A \ar[d]\\
&& X^\I \ar[r]_{\epsilon_X} &  X,
}
\end{equation}
Pulling back $B\ra A$ in two steps we therefore obtain the intermediate map $B_{\epsilon_X} \to A_{\epsilon_X}$  indicated in the following diagram. 
\begin{equation}\label{diagram:fibcomposition2}
\xymatrix{
B^\I \ar[r]|| \ar[r]   & B_{\epsilon_A} \ar[r]  \ar[d] \pbcorner & B_{\epsilon_X}  \ar[d] \pbcorner \ar[r] & B \ar[d]\\
&A^\I \ar[r]|| \ar[r]  & A_{\epsilon_X} \ar[d] \pbcorner \ar[r] & A \ar[d]\\
&& X^\I \ar[r]_{\epsilon_X} &  X
}
\end{equation}
Now use the fact that a trivial fibration structure (\ie\ a +-algebra structure) has a canonical pullback along any map, and that two such structures have a canonical composition (cf.\ Remark \ref{trivfibpushforward}), to obtain a trivial fibration structure for the indicated composite map $B^\I \ra B_{\epsilon_X}$, which is then a fibration structure for the composite $B\to A\to X$.
\end{proof}

\begin{proposition}[$\delta$-Biased Frobenius]\label{prop:Frobenius}
If $\alpha : A \ra X$ and $\beta: B\ra A$  are $\delta$-biased fibrations, then the pushforward $\alpha_*\beta : \Pi_AB \ra X$ is also a $\delta$-biased fibration.
\end{proposition}

\begin{proof}
Given $\delta$-biased fibrations $\alpha : A \ra X$ and $\beta: B\ra A$, let $a : A^\I \ra A_\epsilon$ and $b : B^\I \ra a^*B_\epsilon$ be the associated trivial fibrations, so that we have the situation of diagram \eqref{diagram:fibcomposition2}, with all three squares pullbacks.
\begin{equation}\label{diagram:frobenius1.5}
\xymatrix{
B^\I \ar[r]|| \ar[r]^{b} \ar[rd]_{\beta^\I}  & a^*B_\epsilon \ar[r]  \ar[d]  & B_{\epsilon}  \ar[d]   \ar[r] & B \ar[d]^\beta \\
&A^\I \ar[r]|| \ar[r]^{a} \ar[rd]_{\alpha^\I}  & A_{\epsilon} \ar[d]   \ar[r] & A \ar[d]^\alpha \\
&& X^\I \ar[r]_{\epsilon} &  X.
}
\end{equation}
Taking the pushforward of the righthand vertical column gives a map, 
\[
\gamma:= \alpha_*\beta : \Pi_A{B} \to X\,,
\]
and placing it underneath, along with the corresponding construction from \eqref{diagram:frobenius1}, we then have the following commutative diagram.
 \begin{equation}\label{diagram:frobenius2.5}
\xymatrix{
B^\I \ar[r]|| \ar[r]^{b} \ar[rd]_{\beta^\I}  & a^*B_\epsilon \ar[r]  \ar[d]  & B_{\epsilon}  \ar[d]   \ar[r] & B \ar[d]^\beta \\
& A^\I \ar[r]|| \ar[r]^{a} \ar[rd]_{\alpha^\I}  & A_{\epsilon} \ar[d]   \ar[r] & A \ar[d]^\alpha \\
&& X^\I \ar[r]_{\epsilon} &  X \\
& (\Pi_AB)^\I  \ar[r]_c  \ar[ru]^{\gamma^\I} & (\Pi_AB)_\epsilon  \ar[u] \ar[r] & \Pi_AB \ar[u]_\gamma 
}
\end{equation}
We wish to show that the indicated  map $c : (\Pi_AB)^\I \ra (\Pi_AB)_\epsilon$
admits a +-algebra structure. This we will do by showing that it is a retract of a known +-algebra.
Namely, we can apply the pushforward along the map $\alpha^\I:A^\I \ra X^\I$ to the +-algebra $b : B^\I \ra a^*B_\epsilon$ regarded as an arrow over $A^\I$.  We obtain an arrow over $X^\I$ of the form
\begin{equation}\label{plusalgretract}
\Pi_{A^\I}\,b :  \Pi_{A^\I}\,B^\I \too \Pi_{A^\I}\,a^*B_\epsilon 
\end{equation}
which is indeed a +-algebra, since these are preserved under pushing forward, by Remark \ref{trivfibpushforward}.

Next, observe that by the Beck-Chevalley condition for the central pullback, for the codomain of $c$ we have an isomorphism
\[
(\Pi_AB )_\epsilon\ \cong\ \Pi_{A_\epsilon} B_\epsilon \qquad \text{over $X^\I$.}
\]
And since $\Pi_{A^\I} \cong \Pi_{A_\epsilon} \circ a_*$, for the codomain of our $+$-algebra $\Pi_{A^\I}\,b$ from \eqref{plusalgretract} we also have
\[
 \Pi_{A^\I}\,a^*B_\epsilon\ \cong\ \Pi_{A_\epsilon} a_*a^* B_\epsilon \,.
\]
Thus the image of the unit $\eta : B_\epsilon \ra a_*a^* B_\epsilon$ under $\Pi_{A_\epsilon}$ provides a map 
$\sigma := \Pi_{A_\epsilon}\eta$ over $X^\I$ of the form:
\begin{equation}\label{diagram:frobenius3}
\xymatrix{
& X^\I  \\
 (\Pi_AB)^\I  \ar[r]_c   \ar[ru] & \Pi_{A_\epsilon}B_\epsilon  \ar[u] \ar[d]^{\sigma}\\
 \Pi_{A^\I}B^\I \ar[r]_-{\Pi_{A^\I}\,b} & \Pi_{A_\epsilon} a_*a^* B_\epsilon
}
\end{equation}
Our goal is now to determine further arrows $\varphi, \psi,\tau$ as indicated below, exhibiting $c$ as a retract of $\Pi_{A^\I}\,b$ in the arrow category over $X^\I$.
\begin{equation}\label{diagram:frobenius4}
\xymatrix{
& X^\I  \\
 (\Pi_AB)^\I  \ar[r]_c  \ar[ru] \ar@{..>}[d]_\varphi & \Pi_{A_\epsilon}B_\epsilon  \ar[u] \ar[d]^{\sigma}\\
 \Pi_{A^\I}B^\I \ar[r]_-{\Pi_{A^\I}\,b} \ar@{..>}[d]_\psi & \Pi_{A_\epsilon} a_*a^* B_\epsilon \ar@{..>}[d]^\tau\\
  (\Pi_AB)^\I  \ar[r]_c  & \Pi_{A_\epsilon}B_\epsilon
}
\end{equation}

\smallskip

\noindent $\bullet$ For $\varphi$, we require a map 
\[
\varphi : (\Pi_AB)^\I \ra \Pi_{A^\I}B^\I \qquad \text{over $X^\I$.}
\]

Consider the following diagram, which is based on \eqref{diagram:frobenius2}.
\begin{equation}\label{diagram:frobenius7}
\xymatrix{
B^\I \ar[r]|| \ar[r]_{b} \ar[rd]_{\beta^\I}  & a^*B_\epsilon \ar[r]  \ar[d]  & B_{\epsilon}  \ar[d]   \ar[r] & B \ar[d]^\beta & \\
& A^\I \ar[r]|| \ar[r]_{a}  \ar[rd]_{\alpha^\I}  & A_{\epsilon} \ar[d]   \ar[r] & A \ar[d]^\alpha &  \ar[d] \ar[l] \ar@{..>}[lu]_{e}\Pi_AB\times_{X} A\\
 (\Pi_AB \times_{X} A)^\I \ar@{..>}[uu]^{e^\I} \ar[ru] \ar[rd] && X^\I \ar[r]_{\epsilon} &  X &   \Pi_AB \ar[l]\\
& (\Pi_AB)^\I \ar[ru] \ar@{..>}[r]_\varphi  & \Pi_{A^\I}B^\I   \ar[u] & &
}
\end{equation}
The map $e$ is the counit at $\beta: B\ra A$ of the pullback-pushforward adjunction along $\alpha: A\ra X$. The right-hand side of the diagram, including $e$ and the associated pullback square, reappears (mirrored) on the left under the functor $(-)^\I$, which preserves the pullback. Thus we can take $\varphi$ to be the transpose of $e^\I$ under the pullback-pushforward adjunction along $\alpha^\I: A^I\ra X^\I$,
\[
\varphi\, :=\, \widetilde{e^\I}\,.
\]
An easy diagram chase involving the pullback-pushforward adjunction along $A_\epsilon\ra X^\I$ shows that the upper square in \eqref{diagram:frobenius4} then commutes.

\smallskip

\noindent $\bullet$ For $\tau$: referring to the diagram \eqref{diagram:frobenius2}, since $a : A^\I \ra A_\epsilon$ is a trivial fibration, it has a section $o :  A_\epsilon \ra A^\I$ by lemma \ref{cor:plusalgprops}.  Pulling  $a^*B_\epsilon \ra A^\I$ back along $o$ results in an iso,
\[
o^*a^* B_\epsilon \cong B_\epsilon\quad\text{over $A_\epsilon$}
\]
and so by the adjunction $o^*\!\dashv o_*$ there is an associated map,
\[
a^* B_\epsilon \ra o_* B_\epsilon\quad\text{over $A^\I$}
\]
to which we can apply $a_*$ to obtain a map,
\[
t : a_*a^* B_\epsilon \ra a_*o_*B_\epsilon \cong B_\epsilon\quad \text{over $A_\epsilon$\,.}
\]
This map $t$ is evidently a retraction of the unit $\eta : B_\epsilon \ra a_*a^* B_\epsilon$ over $A_\epsilon$.  Applying the functor $ \Pi_{A_\epsilon}$ therefore gives the desired retraction  of $\sigma$, 
\[
\tau\, :=\, \Pi_{A_\epsilon}t :  \Pi_{A_\epsilon}a_*a^* B_\epsilon \ra \Pi_{A_\epsilon}B_\epsilon\,.
\]

\medskip
\noindent $\bullet$ For $\psi$, we require a map 
\[
\psi:\Pi_{A^\I}B^\I \ra (\Pi_AB)^\I \qquad \text{over $X^\I$.}
\]
Consider the following diagram resulting from combining \eqref{diagram:frobenius2} and \eqref{diagram:frobenius4}, in which all solid arrows are those already introduced. The dotted arrow labelled $p$ is the evident composite.
\begin{equation}\label{diagram:frobenius8}
\xymatrix{
& X^\I \ar[r]^\epsilon & X \\
 (\Pi_AB)^\I  \ar[r]  \ar[ru] \ar[d] & \Pi_{A_\epsilon}B_\epsilon  \ar[u] \ar[d] \ar[r] & \Pi_{A}B \ar[u]\ar[dd]^= \\
 \Pi_{A^\I}B^\I \ar[r] \ar@{..>}[rrd]^>>>>>>>>>>>{p} & \Pi_{A_\epsilon} a_*a^* B_\epsilon \ar[d] & \\
  (\Pi_AB)^\I  \ar[r]  & \Pi_{A_\epsilon}B_\epsilon \ar[r] & \Pi_{A}B 
}
\end{equation}
The lower horizontal composite is the evaluation of the pathobject $(\Pi_AB)^\I$ at the point $\delta : 1 \to \I$, 
\[
\epsilon_{\Pi_AB} = (\Pi_AB)^\delta : (\Pi_AB)^\I \too (\Pi_AB)^1 \cong \Pi_AB\,.
\] 
This is constructed from the (cartesian closed) evaluation,
\[
\eval : \I \times (\Pi_AB)^\I \too \Pi_AB
\]
which is the counit of $\I\times(-) \dashv (-)^\I$, as the composite shown below.
\begin{equation}\label{diagram:biasedfrobenius}
\xymatrix{
(\Pi_AB)^\I \ar[d]_{\cong} \ar[rr]^{\epsilon_{\Pi_AB}} && \Pi_AB  \\
1\times (\Pi_AB)^\I \ar[rr]_{\delta\times{(\Pi_AB)^\I}} && \ar[u]_{\eval}  \I \times (\Pi_AB)^\I 
}
\end{equation}
Let us analyse this evaluation at $\delta$ further, in terms of the \emph{locally} cartesian closed structure associated to the base changes along the section $\delta : 1\to \I$  and retraction $\I \to 1$ in~$\EE$.
Since $\mathsf{id}  \cong \delta^*\I^* : \EE \to \EE/_\I \to \EE$, the map $\epsilon_{\Pi_AB}$ can be rewritten as follows.
\begin{equation}\label{diagram:biasedfrobenius2}
\xymatrix{
(\Pi_AB)^\I \ar[d]_{\cong} \ar[rr]^{\epsilon_{\Pi_AB}} && \Pi_AB  \ar[d]^{\cong} \\
\delta^*\I^*((\Pi_AB)^\I )  \ar[d]_{\cong}  \ar[rr]^{\delta^*\I^*\epsilon_{\Pi_AB}}  &&  \delta^*\I^*\Pi_AB  \ar[d]^{=} \\
\delta^*\I^*\I_*\I^*\Pi_AB \ar[rr]  \ar[rr]_{\delta^*\varepsilon}  & &  \delta^*\I^*\Pi_AB 
}
\end{equation}
where the map $\delta^*\varepsilon$ across the bottom is the counit of the adjunction $\I^*\dashv\I_*$, taken at $\I^*\Pi_AB$, and then pulled back along $\delta : 1\to\I$.  
Before taking the pullback, we therefore have the following iso over $\I$ between that counit $\varepsilon_{\I^*}$ and the image under $\I^*$ of the previously considered evaluation $\epsilon : (\Pi_AB)^\I \to \Pi_AB$ from \eqref{diagram:biasedfrobenius}.
\begin{equation}\label{diagram:biasedfrobenius3}
\xymatrix{
\I^*((\Pi_AB)^\I) \ar[d]_{\cong}  \ar[rrr]^{\\I^*\epsilon} &&& \I^*\Pi_AB  \ar[d]^{=} \\
\I^*\I_*\I^*\Pi_AB   \ar[rrr]_{\varepsilon_{\I^*}}  &&&  \I^*\Pi_AB\,.
}
\end{equation}
Now let us apply $\I^*$ to  \eqref{diagram:frobenius8} to get the map $\I^*p$ in the diagram below, which therefore factors (up to \eqref{diagram:biasedfrobenius3}) through the counit $\varepsilon_{\I^*}$ as $\varepsilon_{\I^*}\circ\I^*(\widetilde{\I^*p})$, where $\widetilde{\I^*p}$ is the adjoint transpose of $\I^*p$, as shown.
\begin{equation}\label{diagram:biasedfrobenius4}
\xymatrix{
\I^*\Pi_{A^\I}B^\I \ar[d]_{\I^*(\widetilde{\I^*p})} \ar[r] \ar@{..>}[rrd]^>>>>>>>>>>>{\I^*p} & \I^*\Pi_{A_\epsilon} a_*a^* B_\epsilon \ar[d] & \\
  \I^*\I_*\I^*\Pi_AB  \ar[r] \ar@/_4ex/ [rr]_{\varepsilon_{\I^*}} & \I^*\Pi_{A_\epsilon}B_\epsilon \ar[r] & \I^*\Pi_{A}B 
}
\end{equation}
We can therefore set 
\[
\psi\, := \, \widetilde{\I^*p}\,,
\]
and we obtain $\epsilon\circ \psi = p$, from which it follows that the square in \eqref{diagram:biasedfrobenius4} commutes by the definition of $\Pi_{A_\epsilon}B_\epsilon$ as a pullback. The same square without $\I^*$ then also commutes by applying the retraction $\delta^*$.  

We have now defined all the maps indicated below, the squares involving $\varphi$ and $\psi$ commute, and the  composite of $\sigma$ and $\tau$  is the identity.
\begin{equation}\label{diagram:frobenius9}
\xymatrix{
& X^\I \ar[r] & X \\
 (\Pi_AB)^\I  \ar[r]  \ar[ru] \ar[d]_{\varphi} 
 	& \Pi_{A_\epsilon}B_\epsilon  \ar[u] \ar[d]^\sigma \ar[r] & \Pi_{A}B \ar[u]\ar[dd]^= \\
 \Pi_{A^\I}B^\I \ar[r] \ar[d]_{\psi} \ar@{..>}[rrd]^>>>>>>>>>>>{p} 
 	& \Pi_{A_\epsilon} p_*p^* B_\epsilon \ar[d]^\tau & \\
  (\Pi_AB)^\I  \ar[r] \ar@/_4ex/ [rr]_\epsilon & \Pi_{A_\epsilon}B_\epsilon \ar[r] & \Pi_{A}B 
}
\end{equation}
To see that $\psi\circ\varphi = 1$, an easy chase through the diagram \eqref{diagram:frobenius9} shows that
\[
\epsilon\circ \psi\circ\varphi = p \circ \varphi = \epsilon\,.
\]
Thus by applying $\I^*$ and using \eqref{diagram:biasedfrobenius3} we have $
\varepsilon_{\I^*} \circ \I^*(\psi \circ\varphi) = \varepsilon_{\I^*}$, 
and so $\psi \circ\varphi = \widetilde{\varepsilon_{\I^*}} = 1$.
\end{proof}

From Proposition \ref{prop:biasedFrobimpliesunbiasedFrob} we then have:

\begin{corollary}[Unbiased Frobenius]\label{cor:unbiasedFrobenius}
The unbiased fibration weak factorization system on $\cSet$ satisfies the Frobenius condition.
\end{corollary}

\begin{corollary}\label{cor:unbiasedPi}
Unbiased fibrations are closed under pushforward along unbiased fibrations.  Thus given unbiased fibrations $X\fib Z$ and $Y\fib Z$ over any base $Z$, the relative exponential $Y^X = X_*X^* Y \to Z$, formed in the slice over $Z$, is again an unbiased fibration.
\end{corollary}

\begin{remark}\label{remark:unbiasedFrobeniusgeneralizes}
We note in passing that the proof just given for the $\delta$-biased case of Frobenius, Proposition \ref{prop:Frobenius}, made no use of the fact that $\delta : 1\to\I$ is generic, nor even that we were working in the slice category over $\I$.  Indeed the same algebraic argument works for $p$-biased fibrations for any point $p : 1\to\I$ of any object $\I$, in any (quasi-)topos $\EE$.
\end{remark}

\section{A universal fibration}\label{sec:U}

We shall construct a \emph{universal small fibration} $\dot{\U}\ra \U$, which is a classifier for small fibrations.  It will be shown in Section \ref{sec:FEP} that the base object $\U$ is fibrant, using the fact to be proved in Section \ref{sec:EEP} that the map $\dot{\U}\ra \U$ itself is \emph{univalent}, in a sense to be made precise.

Our construction of $\dot{\U}\ra \U$ makes use, first of all, of a new description of the well-known Hofmann-Streicher universe in a category $\widehat{\bbC} = [\op{\bbC}, \Set]$ of presheaves on a small category~$\bbC$, which was used in \cite{HS:1997} to interpret dependent type theory. See \cite{awodey:HSuniverse} for further details. 

\subsection*{Classifying families}\label{sec:Ufam}

\begin{definition}[\cite{HS:1997}]\label{def:HSuniverse}
Let $\bbC$ be a small category.  A (type-theoretic) \emph{universe}  $(U, {\mathsf{E}l})$  consists of 
$U\in\widehat{\bbC}$ and $\textstyle{\mathsf{E}l} \in \widehat{\int_\bbC U}$ with: 
 \begin{align}
	U(c)\ &=\ \Cat\big(\op{\bbC/_c}, \Set\big) \label{eq:universeob}\\ 
 	{\mathsf{E}l}(c, A)\ &=\ A(id_c) \label{eq:universeel}
 \end{align}
with the evident associated action on morphisms.  
\end{definition}

A few comments are required: 
\begin{itemize}
\item In contrast to \cite{HS:1997}, in \eqref{eq:universeob}  we take the underlying set of objects of the functor category $\widehat{\bbC/_c}=[\op{\bbC/_c}, \Set]$.

\item As in \cite{HS:1997}, \eqref{eq:universeel} adopts the ``categories with families'' point of view in describing an arrow $E\to U$ in $\widehat{\bbC}$ equivalently as a presheaf on the category of elements $\int_{\bbC}U$, using 
\begin{equation}\label{eq:elements}\textstyle
\widehat{\bbC}/_U\ \simeq\  \widehat{\int_{\bbC}U}\,
\end{equation}
where
\[
E(c)\ =\ {\textstyle \coprod_{A\in U(c)}{\mathsf{E}l}(c, A)}.
\]
The argument $(c, A) \in \int_{\bbC}U$ in \eqref{eq:universeel} thus consists of an object $c\in\bbC$ and an element $A\in U(c)$.
\item To account for size issues, the authors of \cite{HS:1997} assume a Grothendieck universe $u$ in $\Set$, the elements of which are called \emph{small}. The category $\bbC$ is assumed to be small, as are the values of the presheaves, unless otherwise stated.  
\end{itemize}

The presheaf $U$, which is not small, is then regarded as the Grothendieck universe $u$ ``lifted'' from $\Set$ to $[\op{\bbC}, \Set]$.  We first analyse this specification of $(U, {\mathsf{E}l})$ from a different perspective, in order to establish its basic property as a classifier for small families in $\widehat\bbC$. 

\subsubsection*{A realization-nerve adjunction.}

For a presheaf $X$ on $\bbC$, recall that the category of elements is the comma category,
\[\textstyle
\int_\bbC X\ =\ \yon_\bbC/_X\,,
\] 
where $\yon_\bbC : \bbC \to \psh\bbC$ is the Yoneda embedding, which we sometimes supress and write simply $\bbC/_X$ for $\yon_\bbC/_X$. 

\begin{proposition}[\cite{G:1983}, \S{28}]
The category of elements functor 
\[
\textstyle \int_\bbC : \widehat\bbC \too \Cat
\]
 has a right adjoint,
\[
\nu_\bbC : \Cat \too \widehat\bbC\,.
\]
For a small category $\A$, we shall call the presheaf $\nu_\bbC(\A)$ the \emph{($\bbC$-)nerve} of $\A$.
\end{proposition}
\begin{proof}
The adjunction $\int_\bbC\! \dashv \nu_\bbC$ is an instance of the usual ``realization/nerve'' adjunction, here with respect to the covariant slice category functor $\bbC/- : \bbC\to\Cat$, as indicated below.
\begin{equation}\label{eq:nerve}\textstyle
\begin{tikzcd}
	 \widehat\bbC \ar[rr, swap,"\int_\bbC"] &&  \ar[ll, swap,bend right=20, "{ \nu_\bbC}"] \Cat\\  
	 \\
	\bbC \ar[uu, hook, "\yon"] \ar[rruu, swap,"{\bbC/_{-}}"] &&
 \end{tikzcd}
 \end{equation}
In detail, for  $\A\in\Cat$ and $c\in\bbC$, let $\nu_{\bbC}(\A)(c)$ be the Hom-set of functors,
\begin{align*}
\nu_\bbC(\A)(c) &= \Cat\big( {\bbC/_c}\,,\, \A \big)\,,
\end{align*}
with contravariant action on $h : d\to c$ given by pre-composing a functor $P : {\bbC/_c}\to\A$  with the post-composition functor
\[
{\bbC/_h} : {\bbC/_d}\too {\bbC/_c} \,.
\]
For the adjunction, observe that the slice category $\bbC/_c$ is the category of elements of the representable functor $\y{c}$\,,
\[\textstyle
\int_\bbC\y{c}\ \cong\ \bbC/_c\,.
\]
 Thus for representables $\y{c}$\,, we have the required natural isomorphism
 \[\textstyle
 \widehat\bbC\big( \y{c}\,,\, \nu_\bbC(\A) \big)\ \cong\ \nu_\bbC(\A)(c)\  =\ \Cat\big( {\bbC/_c}\,,\, \A \big)\ \cong\ \Cat\big( \int_\bbC\y{c}\,,\, \A \big)\,.
  \]
For arbitrary presheaves $X$, one uses the presentation of $X$ as a colimit of representables over the index category $\int_\bbC X$, and the easy to prove fact that $\int_\bbC$ itself preserves colimits.  Indeed, for any category $\D$, we have an isomorphism in $\Cat$,
\[
\varinjlim_{d\in\D}\,\D/_d \ \cong\ \D\,.
\]
\end{proof}

When $\bbC$ is fixed, we may omit the subscript in the notation $\yon_\bbC$ and  $\int_\bbC$ and $\nu_\bbC$.  The unit and counit maps of the adjunction $\int \dashv \nu$, 
\begin{align*}\textstyle
\eta :&\ \textstyle  X \too \nu{\elem{X}}\,, \\
\epsilon :&\ \textstyle  \elem\nu\A \too \A\,,
\end{align*}
 are then as follows.  At $c\in\bbC$, for $x : \y{c}\ra X$, the functor $(\eta_X)_c(x) : \bbC/_c \to \bbC/_X$ is just composition with $x$, 
\begin{equation}\label{eq:eta}
(\eta_X)_c(x) = \bbC/_x : \bbC/_c \too \bbC/_X\,.
\end{equation}
For $\A\in\Cat$, the functor $ \epsilon : \int\nu\A \to \A$ takes a pair $(c\in\bbC, f : \bbC/_c \to \A)$ to the object $f(1_c) \in \A$,
\[
\epsilon(c,f) = f(1_c).
\]
\begin{lemma}\label{lemma:natpb}
For any $f : Y\to X$, the naturality square below is a pullback.
\begin{equation}\label{eq:naturality}\textstyle
\begin{tikzcd}
	 Y \ar[d, swap,"f"] \ar[r, "{\eta_Y}"] & \nu{\int\!{Y}} \ar[d, "{ \nu{\int\!{f}}}"]\\  
	X \ar[r, swap,"{\eta_X}"] &   \nu{\int\!{X}}.
 \end{tikzcd}
 \end{equation}
\end{lemma}

\begin{proof}
It suffices to prove this for the case $f : X\ra 1$.  Thus consider the square 
\begin{equation}\label{eq:naturalityobject}\textstyle
\begin{tikzcd}
	 X \ar[d] \ar[r, "{\eta_X}"] & \nu{\int\!{X}} \ar[d]\\  
	1\ar[r, swap,"{\eta_1}"] &   \nu{\int\! 1}.
 \end{tikzcd}
 \end{equation}
Evaluating at $c\in\bbC$ and applying \eqref{eq:eta} gives the following square in $\Set$.
\begin{equation}\label{eq:naturalityobjecteval}\textstyle
\begin{tikzcd}
	 Xc \ar[d] \ar[r, "{\bbC/_{-}}"] & \Cat\big( {\bbC/_c}\,,\, {\bbC/_X}\, \big) \ar[d]\\  
	1c\ar[r, swap, "{\bbC/_{-}}"] &   \Cat\big( {\bbC/_c}\,,\, \bbC/_1 \big)
 \end{tikzcd}
 \end{equation}
The image of $*\in 1c$ along the bottom is the forgetful functor $U_c : \bbC/_c\to \bbC$, and its fiber under the map on the right is the set of functors $F : {\bbC/_c}\to {\bbC/_X}$ such that $U_X\circ F = U_c$, where $U_X : \bbC/_X\to \bbC$ is also a forgetful functor. But any such $F$ is uniquely of the form $\bbC/_{x}$ for $x = F(1_c) : \y{c} \to X$.
\end{proof}

\subsubsection*{A universal family.}

For the terminal presheaf $1\in\widehat{\bbC}$ we have an iso $\elem{1} \cong\bbC$, so for every $X\in\widehat{\bbC}$ there is a canonical projection  $\elem X \ra\bbC$, which is a discrete fibration.  It follows that for any map $Y\to X$ of presheaves, the associated map $\elem Y \to \elem X$ is also a discrete fibration. 
Ignoring size issues temporarily, recall that discrete fibrations in $\Cat$ are classified by the forgetful functor $\op{\dot{\Set}}\to \op{\Set}$ from (the opposites of) the category of pointed sets to that of sets (cf.~\cite{W:2007}).  For every presheaf $X\in\widehat{\bbC}$, we therefore have a pullback diagram in $\Cat$,
\begin{equation}\label{eq:classifyuniversecat}\textstyle
\begin{tikzcd}
	 \elem X \ar[d] \ar[r] \pbmark & \op{\dot{\Set}} \ar[d]\\  
	\bbC \ar[r,swap,"X"] &  \op{\Set}.
 \end{tikzcd}
 \end{equation}
Using $\bbC\cong\elem{1}$ and transposing by the adjunction $\int \dashv \nu$ then gives a commutative square in $\widehat{\bbC}$ of the form:
\begin{equation}\label{eq:classifyuniversetype}\textstyle
\begin{tikzcd}
	 X \ar[d] \ar[r] & \nu\op{\dot{\Set}} \ar[d]\\  
	1 \ar[r,swap,"\tilde{X}"] &  \nu\op{\Set}.
 \end{tikzcd}
 \end{equation}

\begin{lemma}
The square \eqref{eq:classifyuniversetype} is a pullback in $\widehat{\bbC}$. More generally, for any map $Y\ra X$ in $\widehat{\bbC}$, there is a canonical pullback square 
\begin{equation}\label{eq:classifyuniversefamily}\textstyle
\begin{tikzcd}
	 Y \ar[d] \pbmark \ar[r] & \nu\op{\dot{\Set}} \ar[d] \\  
	X \ar[r] &  \nu\op{\Set}\,.
 \end{tikzcd}
 \end{equation}
\end{lemma}

\begin{proof}
Apply the right adjoint $\nu$ to the pullback square \eqref{eq:classifyuniversecat} and paste the naturality square \eqref{eq:naturality} from Lemma \ref{lemma:natpb} on the left, to obtain the transposed square \eqref{eq:classifyuniversefamily} as a pasting of two pullbacks.
\end{proof}

Let us write $\VV \to \V$ for the vertical map on the right in \eqref{eq:classifyuniversefamily}, setting
\begin{align}\label{eq:universedef}\textstyle
\VV\, &:=\, \nu\op{\dot{\Set}}\\  
\V\, &:=\, \nu\op{\Set}.\notag
 \end{align}
 
 We summarize our results so far as follows.

 \begin{proposition}\label{prop:Vclassifies}
The nerve $\VV\to\V$  of the classifier for discrete fibrations $\op\SSet\to\op\Set$, as defined in \eqref{eq:universedef}, classifies natural transformations $Y\to X$ in $\widehat{\bbC}$, in the sense that there is always a pullback square,
\begin{equation}\label{eq:classifyuniversefamily2}\textstyle
\begin{tikzcd}
	 Y \ar[d] \pbmark \ar[r] & \VV \ar[d] \\  
	X \ar[r,swap, "\tilde{Y} "] &  \V.
 \end{tikzcd}
 \end{equation}
The classifying map $\tilde{Y} : X\to \V$ is determined by the adjunction $\int \dashv \nu$ as the transpose of the classifying map of the discrete fibration $\elem Y\to\elem X$.  
\end{proposition}

Given a natural transformation $Y\to X$, the classifying map $\tilde{Y} : X\to \V$ is of course not in general unique. Nonetheless, we can  use the construction of $\VV\to\V$ as the nerve of the discrete fibration classifier $\op\SSet\to\op\Set$, for which classifying functors $\bbC \to \op\Set$ are unique up to natural isomorphism, to infer the following proposition, which will be required below (cf.~\cite{Shu:15,GSS:22}).

\begin{proposition}[Realignment for families]\label{prop:realignment}
Given a monomorphism $c : C\cof X$ and a family $Y\to X$, let $y_c : C \to \V$ classify the pullback $c^*Y\to C$.  Then there is a classifying map $y: X \to \V$ for $Y\to X$ with $y\circ c = y_c$.
\begin{equation}\label{diagram:presheafrealignment}
\begin{tikzcd}
c^*Y \ar[dd] \ar[rd] \ar[rr] && \VV \ar[dd] \\
& Y \ar[dd] \ar[ru, dotted] & \\
C  \ar[rd, tail,swap, "c"] \ar[rr, near start, "y_c"] && \V  \\
& X \ar[ru, dotted, swap, "y"] &
\end{tikzcd}
\end{equation}
\end{proposition}
\begin{proof}
Transposing the realignment problem \eqref{diagram:presheafrealignment} for presheaves across the adjunction $\int\dashv \nu$ results in the following realignment problem for discrete fibrations.
\begin{equation}\label{diagram:presheafrealignment2}
\begin{tikzcd}
\elem  c^*Y \ar[dd] \ar[rd] \ar[rr] && \op{\dot{\Set}}  \ar[dd] \\
&\elem  Y \ar[dd] \ar[ru, dotted] & \\
\elem  C  \ar[rd, tail,swap, "{\elem  c}"] \ar[rr, near start, "\widetilde{y_c}"] && \op{{\Set}}   \\
& \elem  X \ar[ru, dotted, swap, "\tilde{y}"] &
\end{tikzcd}
\end{equation}
The category of elements functor $\int $ is easily seen to preserve pullbacks, hence monos; thus let us consider the general case of a functor  $C : \bbC \mono \D$ which is monic in $\Cat$, a pullback of discrete fibrations as on the left below, and a presheaf $E : \bbC \to  \op{{\Set}}$ with $\elem E \cong \mathbb{E}$ over $\bbC$. 
\begin{equation}\label{diagram:presheafrealignment3}
\begin{tikzcd}
\mathbb{E} \ar[dd] \ar[rd] \ar[rr] && \op{\dot{\Set}}  \ar[dd] \\
& \mathbb{F}  \ar[dd] \ar[ru, dotted] & \\
\bbC  \ar[rd, tail,swap, "{C}"] \ar[rr, near start, "E"] && \op{{\Set}}   \\
& \D \ar[ru, dotted, swap, "F"] &
\end{tikzcd}
\end{equation}
We seek $F : \D \to  \op{{\Set}}$ with $\elem F \cong \mathbb{F}$ over $\D$ and $F\circ C = E$.  Let $F_0 : \D \to  \op{\Set}$ with $\elem F_0 \cong \mathbb{F}$ over $\D$, which exists since $\mathbb{F}\to\D$ is a discrete fibration.  Since $F_0\circ C$ and $E$ both classify $\mathbb{E}$, there is a natural iso $e : F_0\circ C \cong E$.
Consider the following diagram
\begin{equation}\label{diagram:presheafrealignment4}
\begin{tikzcd}
\bbC \ar[dd, tail, swap, "{C}"]  \ar[rr, "e"] && {\op{(\Set^{\cong})}}  \ar[dd,"{p_1}"] \ar[r,swap,"{p_2}"] &  {\op{\Set}} \\
&&&\\
\D  \ar[rr,swap, "F_0"]  \ar[rruu, dotted, swap, "f"] && \op{\Set} &  \\
\end{tikzcd}
\end{equation}
where $\Set^{\cong}$ is the category of isos in $\Set$, with $p_1, p_2$ the (opposites of the) domain and codomain projections.  There is a well-known weak factorization system on $\Cat$ (part of the ``canonical model structure'') with injective-on-objects functors on the left and isofibrations on the right.  Thus there is a diagonal filler $f$ as indicated.  The functor $F := p_2\circ f : \D \to \op{\Set}$ is then  the one we seek.
\end{proof}

\subsubsection*{Small maps.}
 
Of course, as defined in \eqref{eq:universedef}, the classifier $\VV\to\V$ cannot be a map in $\widehat{\bbC}$, for reasons of size; we now address this.  
Let $\alpha$ be a cardinal number, and call the sets strictly smaller than it $\alpha$-\emph{small}.  Let $\Set_\alpha\hook\Set$ be the full subcategory of $\alpha$-small sets.  
Call a presheaf $X : \op{\bbC} \to \Set$ $\alpha$-small if all of its values are $\alpha$-small sets, and thus if, and only if, it factors through $\Set_\alpha\hook\Set$. Call a map $f:Y\to X$ of presheaves $\alpha$-small if all of the fibers $f_c^{-1}\{ x\} \subseteq Yc$ are $\alpha$-small sets (for all $c\in\bbC$ and $x\in Xc$). The latter condition is of course equivalent to saying that, in the pullback square over the element $x:\y{c} \to X$, 
\begin{equation}\label{eq:smallmap}\textstyle
\begin{tikzcd}
	 Y_x \ar[d] \pbmark \ar[r] & Y \ar[d, "f"] \\  
	\y{c} \ar[r,swap,"x"] &  X,
 \end{tikzcd}
 \end{equation}
the presheaf $Y_x$ is $\alpha$-small.

Now let us restrict the specification \eqref{eq:universedef} of $\VV\to\V$ to the $\alpha$-small sets:
\begin{align}\label{eq:universedefalpha}\textstyle
\VV_\alpha\, &:=\, \nu \dot{\Set^{\mathsf{op}}_\alpha}\\  
\V_\alpha\, &:=\, \nu \Set^{\mathsf{op}}_\alpha. \notag
 \end{align}
Then the evident forgetful map $\VV_\alpha\to\V_\alpha$ \emph{is} a map in the category $\widehat{\bbC}$ of presheaves, and it is in fact $\alpha$-small. Moreover, it has the following basic property, which is just a restriction of the basic property of $\VV\to\V$ stated in Proposition \ref{prop:Vclassifies}.

 \begin{proposition}\label{prop:familyclassifier}
The map $\VV_\alpha\to\V_\alpha$ classifies $\alpha$-small maps $f:Y\to X$ in $\widehat{\bbC}$, in the sense that there is always a pullback square,
\begin{equation}\label{eq:classifyuniversefamilyalpha}\textstyle
\begin{tikzcd}
	 Y \ar[d] \pbmark \ar[r] & \VV_\alpha \ar[d] \\  
	X \ar[r,swap, "\tilde{Y}"] &  \V_\alpha.
 \end{tikzcd}
 \end{equation}
The classifying map $\tilde{Y} : X\to \V_\alpha$ is determined by the adjunction $\int \dashv \nu$ as (the factorization of) the transpose of the classifiyng map of the discrete fibration $\elem X\to\elem Y$. 
\end{proposition}

\begin{proof} If $Y\to X$ is $\alpha$-small, its classifying map $\tilde{Y} : X\to\V$ factors through $\V_\alpha \hook \V$, as indicated below, 
\begin{equation}\label{eq:classifyuniversetype2}\textstyle
\begin{tikzcd}
	 Y \ar[d] \ar[rr, bend left] \ar[r] & \nu\op{\dot{\Set_\alpha}} \ar[d] \ar[r,hook] & \nu\op{\dot{\Set}} \ar[d]\\  
	X \ar[rr, bend right, swap,"\tilde{Y}"] \ar[r] &  \nu\op{\Set_\alpha} \ar[r,hook] &  \nu\op{\Set},
 \end{tikzcd}
 \end{equation}
in virtue of the following adjoint transposition,
\begin{equation}\label{eq:adjointtranspose}\textstyle
\begin{tikzcd}
	 \elem Y \ar[d] \ar[rr, bend left] \ar[r] & \op{\dot{\Set_\alpha}} \ar[d] \ar[r,hook] & \op{\dot{\Set}} \ar[d]\\  
	 \elem X \ar[rr, bend right, ] \ar[r]  &  \op{\Set_\alpha} \ar[r,hook]  &  \op{\Set}.
 \end{tikzcd}
  \end{equation}
Note that the square on the right is evidently a pullback, and so the one on the left is, too, because the outer rectangle is the classifying pulback of the discrete fibration $\elem Y \to \elem X$, as stated.  Thus the left square in \eqref{eq:classifyuniversetype2} is also a pullback.
\end{proof}

\subsubsection*{Examples of universal families $\VV_\alpha \too \V_\alpha$.}\label{examples:universalfamilies}

\begin{enumerate}
\item Let $\alpha = \kappa$ a strongly inaccessible cadinal, so that $\mathsf{ob}({\Set_\kappa})$ is a Grothendieck universe.  Then the Hofmann-Streicher universe of Definition \ref{def:HSuniverse} is recovered as the $\kappa$-small map classifier
\begin{equation*}
E\, \cong\, \VV_\kappa \too \V_\kappa\, \cong\, U
\end{equation*}
 in the sense of Proposition \ref{prop:familyclassifier}.  Indeed, for $c\in\bbC$, we have 
 \begin{align}
  \V_{\kappa}{c}\ &=\ \nu(\Set^{\mathsf{op}}_\kappa)(c) = \Cat\big( {\bbC/_c}\,,\, \Set^{\mathsf{op}}_\kappa \big)\  =\ \mathsf{ob}(\widehat{\bbC/_c})\ =\ U{c} \,.
   \end{align} 
For $\VV_{\kappa}$ we then have,
   \begin{align}\label{eq:veedotc}
   \VV_{\kappa}{c}\ =\ \nu(\SSet^{\mathsf{op}}_\kappa)(c)\ &=\ \Cat\big( {\bbC/_c}\,,\, \SSet^{\mathsf{op}}_\kappa \big) \notag \\ 
   &\cong\ {\textstyle \coprod_{A\in\V_{\kappa}{c}}\Cat_{{\bbC/_c}}\big( {\bbC/_c}\,,\, A^*\Set^{\mathsf{op}}_\kappa \big)}
   \end{align}
   where the $A$-summand in \eqref{eq:veedotc} is defined by taking sections of the  pullback indicated below.
   \begin{equation}\label{eq:pbforindexing}\textstyle
\begin{tikzcd}
	A^*\Set^{\mathsf{op}}_\kappa \ar[d] \ar[r] \pbmark & \SSet^{\mathsf{op}}_\kappa \ar[d]\\  
	\bbC/_c \ar[r,swap,"A"] \ar[u, bend left, dotted] \ar[ur, dotted] &  \Set^{\mathsf{op}}_\kappa
 \end{tikzcd}
 \end{equation}
 But $A^*\Set^{\mathsf{op}}_\kappa\ \cong\ {\textstyle \int_{\bbC/_c}\!A}$ over $\bbC/_c\,$, and sections of this discrete fibration in $\Cat$ correspond uniquely to natural maps $1\to A$ in $\widehat{{\bbC/_c}}$.  Since $1$  is representable in $\widehat{{\bbC/_c}}$ we can continue \eqref{eq:veedotc} by
  \begin{align*}
   \VV_{\kappa}{c}\ &\cong\ {\textstyle \coprod_{A\in \V_{\kappa}{c}}\Cat_{{\bbC/_c}}\big( {\bbC/_c}\,,\, A^*\Set^{\mathsf{op}}_\kappa \big)}\\
   	&\cong\ {\textstyle \coprod_{A\in \V_{\kappa}{c}} \widehat{{\bbC/_c}}(1, A)}\\
	&\cong\ {\textstyle \coprod_{A\in \V_{\kappa}{c}} A(1_c) } \\
	& =\ {\textstyle \coprod_{A\in \V_{\kappa}{c}} {\mathsf{E}l}(\langle c, A\rangle)}\\
	& =\  E c\,.
   \end{align*}
 
\item By functoriality of the nerve $\nu : \Cat \to \widehat{\bbC}$, a sequence of Grothendieck universes 
\[
\Set_\alpha \subseteq \Set_{\beta} \subseteq ...
\]
 in $\Set$ gives rise to a (cumulative) sequence of type-theoretic universes 
 \[
 \V_\alpha \mono {\V_\beta} \mono ...
 \]
  in $\widehat{\bbC}$. More precisely, there is a sequence of  cartesian squares,
\begin{equation}\label{eq:Vhierarchy}\textstyle
\begin{tikzcd}
	 \VV_\alpha \ar[d] \ar[r,tail] \pbmark & {\VV_\beta} \ar[d] \ar[r,tail] \pbmark & \dots \\  
	 \V_\alpha  \ar[r, tail]  &  {\V_\beta} \ar[r,tail]  & \dots\,,
 \end{tikzcd}
  \end{equation}
in the image of $\nu : \Cat\too\widehat\bbC$, classifying small maps in $\widehat\bbC$ of increasing size, in the sense of Proposition \ref{prop:familyclassifier}.

\item\label{universeexample:sliceuniverse} Let $\alpha = 2$ so that $1\to 2$ is the subobject classifier of $\Set$, and 
\[
{\mathbb{1}} = \SSet^{\mathsf{op}}_2 \too  \Set^{\mathsf{op}}_2 = \mathbbm{2}
\]
is then a classifier in $\Cat$ for \emph{sieves}, i.e.\ full subcategories $\mathbb{S}\hook\A$ closed under the domains of arrows $a\to s$ for $s\in\mathbb{S}$.  The nerve $\VV_{2}  \to \V_{2}$ is then the usual subobject classifier $1\to\Omega$ of $\widehat\bbC$,
\begin{equation}\label{eq:SOCnerve}\textstyle
\begin{tikzcd}
	 \VV_{2}  \ar[d] \ar[r,equals] & \nu \mathbbm{1} \ar[d] \ar[r,"\sim"]   & 1 \ar[d] \\  
	 \V_{2} \ar[r,equals] & \nu \mathbbm{2} \ar[r,"\sim"]   &  \Omega 
 \end{tikzcd}
  \end{equation}

\item For any $X\in \widehat{\bbC}$, we have an equivalence 
\[
\widehat{\bbC}/_X\ 
\simeq\ \widehat{\textstyle \int_{\bbC}X}\ \simeq\ \mathsf{dFib}/_{\!\int_{\bbC}X}
\]
where, generally, $\mathsf{dFib}/_{\mathbb{D}}$ is the category of discrete fibrations over a category $\mathbb{D}$.
This equivalence commutes with composition along discrete fibrations, in the sense that the forgetful functor 
\[
{X_!}: \widehat{\bbC}/_X \to \widehat{\bbC}
\]
 given by composition along $X \to 1$ agrees (up to canonical isomorphism) with the base change $(p_X)_! \dashv (p_X)^*$ of presheaves along the projection  ${\textstyle  p_X : \int_{\bbC}X \to  \bbC}$, and with composition along the discrete fibration $p_X$, as indicated in:
\begin{equation}\label{diagram:sliceuniversepullsback}
\begin{tikzcd}
\widehat{\bbC}/_X  \ar[d,swap,"{{X_!}}"] \ar[r, "{\sim}"] 
	& \widehat{\textstyle \int_{\bbC}X} \ar[d,swap,"{{(p_X)_!}}"] \ar[r, "{\sim}"] 
		& \mathsf{dFib}/_{\!\int_{\bbC}X} \ar[d,"p_X\circ(-)"]  \\
 \widehat{\bbC}  \ar[r,swap, "{\sim}"]   & \widehat{\bbC}  \ar[r,swap, "{\sim}"] 
	&  \mathsf{dFib}/_{\!\bbC}.
\end{tikzcd}
\end{equation}

It follows that the pullback functor $X^* : \widehat{\bbC} \to \widehat{\bbC}/_X$ commutes with the corresponding right adjoints (one of which is the nerve), and therefore preserves the respective universes, 
\[\textstyle
 X^* \V_{\bbC} \ 
   \cong\ (p_X)^* \nu_{\bbC}(\Set^{\mathsf{op}}) \ 
   \cong\  \nu_{\int_{\bbC}X}(\Set^{\mathsf{op}}) \ 
   \cong\  \V_{ \int_{\bbC}X} \,.
 \]
\end{enumerate}

 \begin{corollary}\label{prop:familyclassifierforslices}
Let $\VV_\alpha\to\V_\alpha$ classify $\alpha$-small maps in $\widehat{\bbC}$, as in Proposition \ref{prop:familyclassifier}.  Then for any $X\in \widehat{\bbC}$, the pullback $X^*\VV_\alpha\to X^*\V_\alpha$ classifies 
$\alpha$-small maps in $\widehat{\bbC}/_X$\,.
\end{corollary}

\subsection*{Classifying trivial fibrations}

Returning now to the presheaf category $\cSet = \psh{\Box}$ of cubical sets, recall from section \ref{sec:cofibrations} that (uniform) trivial fibration structures on a map $A\ra X$ correspond bijectively to relative +-algebra structures over $X$ (definition \ref{def:+alg}).  A relative $+$-algebra structure on $A \ra X$ is an algebra structure for the pointed polynomial endofunctor $+_X : \cSet/X \too \cSet/X$, where recall from \eqref{eq:partialmapclassifier}, 
\[
A^+ = \sum_{\varphi: \Phi} A^{[\varphi]}\quad\text{over $X$}.
\]
A +-algebra structure is then a retract $\alpha : A^+\ra A$ over $X$ of the canonical map $\eta_A : A\ra A^+$,
\begin{equation}\label{eq:Aplus}
\xymatrix{
A\ar[rd] \ar[r]^{\eta_A} \ar@/^6ex/ [rr]^= & A^+ \ar[r]^{\alpha} \ar[d]& \ar[ld]A \\
& X. &
}
\end{equation}
In more detail, let us write $A\ra X$ as a family $(A_x)_{x\in X}$, so that $A=\sum_{x:X}A_x \ra X$. Since the +-functor acts fiberwise, the object $A^+$ in \eqref{eq:Aplus} is then the indexing projection
\[
\sum_{x:X}A^+_x \ra X.
\]
Working in the slice  $\cSet/X$, the (relative) exponentials (internal Hom's) $[A^+, A]$ and $[A, A]$ and the ``precomposition by $\eta_A$'' map $[\eta_A, A]$,  fit into the following pullback diagram 
\begin{equation}\label{diag:plualgstru}
\xymatrix{
+\mathsf{Alg}(A)\ar[d] \ar[r] \pbcorner & [A^+, A] \ar[d]^{[\eta_A, A]}\\
1 \ar[r]_{'{\mathsf{id}_A}'} & [A, A].
}
\end{equation}
The constructed object $+\mathsf{Alg}(A) \ra X$ over $X$ is then the \emph{object of +-algebra structures on $A\ra X$}, in the sense that sections $X \ra +\mathsf{Alg}(A)$ correspond uniquely to +-algebra structures on $A\ra X$. Moreover, $+\mathsf{Alg}(A) \ra X$ is stable under pullback, in the sense that for any $f:Y\ra X$, we have two pullback squares,
\begin{equation}\label{diagram:pbplus}
\xymatrix{
f^*A \ar[d] \ar[r]  & A \ar[d]\\
Y \ar[r]_{f} &X\\
+\mathsf{Alg}(f^*A)\ar[u] \ar[r] & +\mathsf{Alg}(A)\ar[u].
}
\end{equation}
because the +-functor, exponentials and pullbacks occurring in the construction of $+\mathsf{Alg}(A) \ra X$ are themselves all stable. 

It then follows from Proposition \ref{prop:familyclassifier} that, if $A\ra X$ is small, then $+\mathsf{Alg}(A) \ra X$ is itself a pullback of the analogous object $+\mathsf{Alg}(\VV) \ra \V$ constructed from the universal small family $\VV\ra\V$ of Proposition \ref{prop:familyclassifier}, so there are two pullback squares:
\begin{equation}\label{diagram:tfib1}
\xymatrix{
A \ar[d] \ar[r]  & \VV \ar[d]\\
X \ar[r]_{\chi_A} & \V\\
+\mathsf{Alg}(A)\ar[u] \ar[r] & +\mathsf{Alg}(\VV)\ar[u].
}
\end{equation}

\begin{proposition}\label{prop:classTFib}
There is a \emph{universal small trivial fibration}  
\[
\TTFib\ra\TFib.
\]
 Every small trivial fibration $A \ra X$ is a pullback of $\TTFib\ra\TFib$ along a canonically determined classifying map $X\ra \TFib$.
\begin{equation}\label{diagram:classifytf}
\xymatrix{
A \ar[d] \ar[r]  \pbcorner & \TTFib\ar[d]\\
X \ar[r] & \TFib
}
\end{equation}
\end{proposition}

\begin{proof}
We can take 
\[
\TFib := +\mathsf{Alg}(\VV),
\]
 which comes with its projection $+\mathsf{Alg}(\VV) \ra \V$ as in diagram \eqref{diagram:tfib1}.  Now define $p_t:\TTFib\ra\TFib$ by pulling back the universal small family,
\[
\xymatrix{
\TTFib \ar[d]_{p_t} \ar[r]  \pbcorner & \VV\ar[d]^p\\
\TFib \ar[r] & \V.
}
\]
Consider the following diagram, in which all the squares (including the distorted ones) are pullbacks, with the outer one coming from proposition \ref{prop:familyclassifier} and the lower one from \eqref{diagram:tfib1}.
\begin{equation}\label{diagram:classifytf2}
\xymatrix{
&A \ar[ddd] \ar[rrr]^{q_A}   \ar@{.>}[rrd] &&& \VV\ar[ddd]^p\\
& && \TTFib \ar[d]_{p_t} \ar[ru]  &\\
\TFib(A) \ar[rd] \ar[rrr] |<<<<<<<<<<<<\hole  &&& \TFib \ar[rd] &\\
&X \ar[rrr]_{\chi_A} \ar@{.>}@/^1pc/[lu]^\alpha \ar@{.>}[rru]_{\alpha'} &&& \V.
}
\end{equation}
A trivial fibration structure $\alpha$ on $A\ra X$ is a section the object of $+$-algebra structures on $A$, occurring in the diagram as 
\[
\TFib(A) := +\mathsf{Alg}(A),
\]
 the pullback of $\TFib = +\mathsf{Alg}(\VV)$ along the classifying map $\chi_A : X \to \V$ for the small family $A\to X$.  Such sections correspond uniquely to factorizations $\alpha'$ of $\chi_A$ as indicated, which in turn induce pullback squares of the required kind \eqref{diagram:classifytf}.

Note that the map $p_t : \TTFib\ra\TFib$ has a canonical trivial fibration structure. Indeed, consider the following diagram, in which both squares are pullbacks.
\begin{equation}\label{diagram:fibisfib}
\xymatrix{
\TTFib \ar[d]_{p_t} \ar[r]  & \VV \ar[d]\\
\TFib \ar[r] & \V\\
\TFib(\TTFib) \ar[u] \ar[r] & \TFib(\VV)\ar[u].
}
\end{equation}
$\TFib(\VV)$ is the object of trivial fibration structures on $\VV\ra\V$, and its pullback $\TFib(\TTFib)$ is therefore the object of trivial fibration structures on $p_t : \TTFib\ra\TFib$.  Thus we seek a section of $\TFib(\TTFib) \ra \TFib$.  But recall that $\TFib = \TFib(\VV)$ by definition, so the lower pullback square is the pullback of $\TFib(\VV)\ra \V$ against itself, which does indeed have a distinguished section, namely the diagonal
\[
\Delta : \TFib(\VV) \ra \TFib(\VV)\times_\V\TFib(\VV).
\]
\end{proof}

We record the following notation and corresponding fact from the foregoing proof for future reference:

\begin{lemma}\label{lemma:TFibstable}
The classifying type $\TFib(A) := +\mathsf{Alg}(A) \to X$ for trivial fibration structures on a map $A\to X$  is stable under pullback, in the sense that for any $f:Y\ra X$, we have two pullback squares,
\begin{equation}\label{diagram:TFibstable}
\xymatrix{
f^*A \ar[d] \ar[r]  & A \ar[d]\\
Y \ar[r]_{f} &X\\
\TFib(f^*A)\ar[u] \ar[r] & \TFib(A)\ar[u].
}
\end{equation}
\end{lemma}

Since the universal small trivial fibration $\TTFib\ra\TFib$ in $\cSet$ from Proposition \ref{prop:classTFib} was constructed as $\TFib = \TFib(\VV)$ for the universal small family $\VV\to \V$, which in turn is stable under pullback by Corollary \ref{prop:familyclassifierforslices}, we also have: 

\begin{corollary}\label{por:classTFibslice}
The base change of the universal small trivial fibration  
\[
\TTFib\ra\TFib
\]
in $\cSet$ along $\I^* : \cSet \to \cSet/_\I$ is a universal small trivial fibration in $\cSet/_\I$.
\end{corollary}

\subsection*{Classifying fibrations}\label{sec:universalfibration}

In order to classify fibrations $A\fib X$, we shall proceed as for trivial fibrations by constructing, for any map $A\ra X$, an object $\Fib(A)\to X$ of fibration structures which, moreover, is stable under pullback.  We then apply the construction to the universal small family $\VV\ra\V$ of Proposition \ref{prop:familyclassifier} to obtain a universal small fibration.    Here we will of course need to distinguish between biased and unbiased fibrations.  In Lemma \ref{lemma:classtypebiasedfibstruct}, we first construct a stable classifying type $\Fib(A)\ra X$ for $\delta$-biased fibration structures on any map $A\to X$ in $\cSet/_\I$ where $\delta$ is the generic point.    In Lemma \ref{lemma:classtypeunbiasedfibstruct}
we then transfer the construction along the base change $\I^* : \cSet \to \cSet/_\I$ to obtain a classifier $\Fib(A)\ra X$ for unbiased fibration structures on any $A\to X$ in~$\cSet$.

The construction of $\Fib(A)\ra X$ for biased fibration structures with respect to a point $\delta : 1\to \I$ is already a bit more involved than was that of $\TFib(A)\ra X$.  In particular, it requires the codomain $\I$ of $\delta$ to be \emph{tiny}, which is indeed the case for the generic point $\delta : 1 \to  \I^*\I$ in $\cSet/_\I$ by Lemma~\ref{lemma:tinyslicedI}.

\subsubsection*{The classifying type of biased fibration structures.}

A classifying type $\Fib(A)\ra X$ of (uniform, $\delta$-biased) fibration structures on a map $p: A\ra X$, 
as defined in Section \ref{sec:biasedfibration}, can be constructed as follows.

 \begin{enumerate}

\item First form the pullback-hom $\pbh{\delta}{p} : A^\I \ra X^\I \times_X A$ with the point $\delta : 1\to \I$, as indicated in the following diagram.
\begin{equation}\label{diagram:fibU1}
\xymatrix{
A^\I \ar@/_4ex/ [rdd]_{(p_A)^\I} \ar@{.>}[rd]^{\delta\Rightarrow{p}} \ar@/^4ex/ [rrd]^{A^\delta} && \\
& X^\I \times_{X} A \ar[d] \ar[r] & A \ar[d]^{p} \\
& X^\I \ar[r]_{X^\delta} &  X 
}
\end{equation}

\item A fibration structure on $p : A\ra X$ is then a relative +-algebra structure on $\delta\Rightarrow p$ in the slice category over its codomain $X^\I \times_X A$. 
To construct a classifier for such structures, let us first relabel the objects and arrows in diagram \eqref{diagram:fibU1} as follows:
\begin{align*}
\epsilon &:= X^\delta : X^\I \ra X \\
A_\epsilon &:= X^\I \times_X A\\
\epsilon_A &:= \delta\Rightarrow p
\end{align*}
so that the working part of \eqref{diagram:fibU1} becomes:
\begin{equation}\label{diagram:fibU2}
\xymatrix@=2em{
A^\I \ar[rd]_{\epsilon_A} && \\
& A_\epsilon\ar[d]_{p_\epsilon} \ar[r] \pbcorner & A \ar[d]^{p} \\
& X^\I \ar[r]_{\epsilon} &  X 
}
\end{equation}

\item Now a relative +-algebra structure on $\epsilon_A$ (Definition \ref{def:+alg}) is a retract $\alpha$ over $A_\epsilon$ of the unit $\eta$, as indicated below, where $D$ is simply the domain of the map $(\epsilon_A)^+$ resulting from applying the relative +-functor in the slice category over $A_\epsilon$ to the object $\epsilon_A$.
\begin{equation}\label{diagram:fibU3}
\xymatrix{
A^\I \ar[rd]_{\epsilon_A} \ar[r]_\eta & \ar@{.>}@/_1pc/[l]_\alpha D \ar[d]^{(\epsilon_A)^+} & \\
& A_\epsilon\ar[d]_{p_\epsilon} \ar[r] \pbcorner & A \ar[d]^{p} \\
& X^\I \ar[r]_{\epsilon} &  X 
}
\end{equation}

\item As in the construction \eqref{diag:plualgstru}, there is an object $\TFib(\epsilon_A) = +\mathsf{Alg}(\epsilon_A)$ over  $A_\epsilon$ of relative +-algebra structures on $\epsilon_A$, the sections of which correspond uniquely to relative +-algebra structures on $\epsilon_A$ (and thus to fibration structures on $A$).
\begin{equation}\label{diagram:fibU4}
\xymatrix{
&A^\I \ar[d]_{\epsilon_A} \ar[r]_\eta & \ar@{.>}@/_1pc/[l]_\alpha D \ar[ld]^<<<<<{(\epsilon_A)^+}  \\
\TFib(\epsilon_A) \ar[r] & A_\epsilon\ar[d]_{p_\epsilon} \ar[r] \pbcorner & A \ar[d]^{p} \\
& X^\I \ar[r]_{\epsilon} &  X 
}
\end{equation}

\item Sections of $\TFib(\epsilon_A)\too A_\epsilon$ then correspond to sections of its push-forward along $p_\epsilon$, which we shall call $F_A$: 
\[
F_A := (p_\epsilon)_*\TFib(\epsilon_A)\,.
\]

\begin{equation}\label{diagram:fibU5}
\xymatrix{
&A^\I \ar[d]_{\epsilon_A} \ar[r]_\eta & \ar@{.>}@/_1pc/[l]_\alpha D \ar[ld]^<<<<<{(\epsilon_A)^+}  \\
\TFib(\epsilon_A) \ar[r] & A_\epsilon\ar[d]_{p_\epsilon} \ar[r] \pbcorner & A \ar[d]^{p} \\
F_A \ar[r] & X^\I \ar[r]_{\epsilon} &  X 
}
\end{equation}

\item One might now try taking another pushforward of $F_A \ra X^\I$ along $\epsilon : X^\I\ra X$ to get the object $\Fib(A) \ra X$ that we seek, but unfortunately, this would not be stable under pullback along arbitrary maps $Y\ra X$, because the evaluation $\epsilon = X^\delta : X^\I \ra X$ is not stable in that way.  Instead we use the \emph{root} functor, i.e.\ the right adjoint of the pathspace, $(-)^\I \dashv (-)_\I$ (Proposition \ref{prop:Itiny}). 

Let $f : F_A \ra X^\I$ be the map $(p_\epsilon)_*\TFib(\epsilon_A)$ indicated in \eqref{diagram:fibU5}, and let $\eta : X \ra (X^\I)_\I$ be the unit of the root adjunction at $X$.  Then define $\Fib(A)\ra X$ by 
\[
 \Fib(A) :=  \eta^*{f_\I}
 \]
 as indicated in the following pullback diagram.
\begin{equation}\label{diagram:fibU6}
\xymatrix{
  \Fib(A)  \ar[d] \ar[r] \pbcorner & (F_A)_\I \ar[d]^{f_\I}\\
  X \ar[r]_{\eta} &  (X^\I)_\I
}
\end{equation}
By adjointness, sections of $\Fib(A)\ra X$ then correspond bijectively to sections of  $f : F_A \ra X^\I$.  
\end{enumerate}

\begin{lemma}\label{lemma:classtypebiasedfibstruct}
For any map $A\ra X$ in $\cSet/_\I$, the map $\Fib(A)\ra X$  in \eqref{diagram:fibU6} is a \emph{classifying type for $\delta$-biased fibration structures}: sections of $\Fib(A)\ra X$ correspond bijectively to $\delta$-biased fibration structures on $A\ra X$, and the construction is stable under pullback in the sense that for any $f:Y\ra X$, we have two pullback squares,
\begin{equation}\label{diagram:Fibstable}
\xymatrix{
f^*A \ar[d] \ar[r]  & A \ar[d]\\
Y \ar[r]_{f} &X\\
\Fib(f^*A)\ar[u] \ar[r] & \Fib(A)\ar[u].
}
\end{equation}
\end{lemma}

\begin{proof}
It is clear from the construction that fibration structures on $A\ra X$ correspond bijectively to sections of $\Fib(A)\ra X$. We show that $\Fib(A)\ra X$ is also stable under pullback.  To that end, the relevant steps of the construction are recalled schematically below.
\begin{equation}\label{diagram:fibU8}
\xymatrix{
&A^\I \ar[d]_{\epsilon_A} & &\\
\TFib(\epsilon_A) \ar[r] & A_\epsilon\ar[d]_{p_\epsilon} \ar[r] \pbcorner & A \ar[d]^{p} &\\
F_A \ar[r] & X^\I \ar[r]_{\epsilon} &  X & \ar[l] \Fib(A)
}
\end{equation}

Now consider the following diagram, in which the right hand side consists  of the data from  \eqref{diagram:fibU8}, and the front, central square is a pullback.
\begin{equation}\label{diagram:fibU9}
\xymatrix{
 & B^\I \ar[rrr] \ar[d]_{\epsilon_B} &&& A^\I \ar[d]^{\epsilon_A} &\\
 \TFib(\epsilon_B) \ar[r] & B_\epsilon \ar[d] \ar[rd] \ar@{.>}[rrr] &&& A_\epsilon \ar[d] \ar[ld] &  \TFib(\epsilon_A) \ar[l] \\
 F_B  \ar[r] & Y^\I \ar[rd] \ar@{.>}@/_.75pc/[rrr] & B \ar[r] \ar[d] & A \ar[d] &  X^\I \ar[ld]  & \ar[l] F_A \\
 &\Fib(B) \ar[r] & Y \ar[r]_f & X &\Fib(A) \ar[l] &
 }
\end{equation}
As in the proof of Lemma \ref{lemma:fibrationspullback}, on the left side we repeat the construction with $B \ra Y$ in place of $A\ra X$.  The left face of the indicated (distorted) cube is then also a pullback, whence the back (dotted) face is a pullback, since the two-story square in back is the image of the front pullback square under the right adjoint $(-)^\I$. Finally, the top rectangle in the back is therefore also a pullback.
 
It follows that $\TFib(\epsilon_B)$ is a pullback of $\TFib(\epsilon_A)$ along the upper dotted arrow, as in Lemma \ref{lemma:TFibstable}, and so the pushforward $F_B$ is a pullback of the corresponding $F_A$, along the lower dotted arrow (which is $f^\I$), by the Beck-Chevalley condition for the dotted pullback square.  Let us record this for later reference:
\begin{equation}\label{eq:pbFB}
F_B \cong (f^\I)^*F_A.
\end{equation}

It remains to show that $\Fib(B)$ is a pullback of $\Fib(A)$ along $f:Y\ra X$, and now it is good that we did not take these to be pushforwards of $F_B$ and $F_A$, because the floor of the cube need not be a pullback, and so the Beck-Chavalley condition would not apply.  Instead, consider the following diagram.
\begin{equation}\label{diagram:fibU10}
\xymatrix{
& \Fib(B) \ar[ld] \ar[dd]  \ar[r] & \Fib(A) \ar[dd] \ar[rd] & \\
(F_B)_\I \ar[dd] \ar[rrr] &&& (F_A)_\I \ar[dd] \\
& Y \ar[ld]_\eta \ar[r]_f & X \ar[rd]^\eta &\\
(Y^\I)_\I \ar[rrr]_{(f^\I)_\I} &&& (X^\I)_\I 
 }
\end{equation}
The sides of the cube are pullbacks by the construction of $\Fib(A)$ and $\Fib(B)$. The front face is the root of the pullback \eqref{eq:pbFB} and is thus also a pullback, since the root is a right adjoint. The base commutes by naturality of the unit of the adjunction, and so the back face is also a pullback, as required.  
\end{proof}

Now let us apply the foregoing construction of $\Fib(A)$ to the universal family $\VV\to\V$ to get $\Fib(\VV) \ra \V$, and define the universal small ($\delta$-biased) fibration in $\cSet/_\I$ by setting $\Fib : =\Fib(\VV)$ and $\FFib\fib\Fib$ by pulling back the universal family,
\begin{equation}\label{diagram:universalfib}
\xymatrix{
\FFib \ar[d] \ar[r]  \pbcorner & \VV\ar[d]^p\\
\Fib \ar[r] & \V.
}
\end{equation}
The proof of the following then proceeds just as that given for $\TTFib \to \TFib$ in Proposition \ref{prop:classTFib}.

\begin{proposition}\label{prop:UniversalFib}
The map $\FFib\ra\Fib$ constructed in \eqref{diagram:universalfib} is a \emph{universal small $\delta$-biased fibration} in $\cSet/_\I$: every small $\delta$-biased fibration $A \fib X$ in $\cSet/_\I$ is a pullback of $\FFib \fib \Fib$ along a canonically determined classifying map $X\ra \Fib$.
\begin{equation}\label{diagram:classifyfib}
\xymatrix{
A \ar@{>>}[d] \ar[r]  \pbcorner & \FFib \ar@{>>}[d] \\
X \ar[r] & \Fib
}
\end{equation}
\end{proposition}

\begin{remark}
Proposition \ref{prop:UniversalFib} made no use of the fact that we were working in the slice category $\cSet/_\I$ with $\delta : 1\to\I$ the generic point. It holds equally for $\delta$-biased fibrations with respect to any point $\delta: 1\to\I$ of a tiny object $\I$.  Thus \eg\ it could be used (with obvious adjustment) to construct a classifier for the $\{\delta_0, \delta_1\}$-biased fibrations of Section \ref{sec:biasedfibration} in (Cartesian, Dedekind, or other varieties of) cubical sets $\cSet$.  
\end{remark}

\subsubsection*{The classifying type of unbiased fibration structures.}\label{par:classifyunbiasedfib} 

In order to classify \emph{unbiased} fibration structures on maps $A\to X$ in $\cSet$, we first apply the pullback $\I^*: \cSet \ra \cSet/_\I$ and take the classifier $\Fib(\I^*A)\ra \I^*X$ for $\delta$-biased fibration structures, then apply the pushforward $\I_*: \cSet/_\I \ra \cSet$ and pull the result   $\I_*\Fib(\I^*A)\ra \I_*\I^*X$ back along the unit  $X \to \I_*\I^*X$.  

To show that this indeed classifies unbiased fibration structures on $A\to X$, let us first rename the classifying type from Lemma \ref{lemma:classtypebiasedfibstruct}, which was constructed over $\I$, to $\Fib_i(\I^*A) \ra \I^{*}X$, and then apply $\I_*$ to get the map,
\[
\Pi_{i:\I}\Fib_i(\I^*A) := \I_*(\Fib_i(\I^*A)) \too X^\I
\]
in $\cSet$.  Then, as just said, we define the desired map $\Fib(A)\ra X$ as the pullback along the unit $\rho : X \ra X^\I$ of  $\I^*\dashv \I_*$ as indicated below.
\begin{equation}\label{diagram:fibU7}
\xymatrix{
 \Fib(A) \ar[d] \ar[r] \pbcorner & \ar[d] \Pi_{i:\I}\Fib_i(\I^*A) \\
 X \ar[r]_{\rho} &  X^\I
}
\end{equation}
It now follows immediately from the adjunction $\I^*\dashv \I_*$ that sections of $\Fib(A)\ra X$ correspond bijectively to sections of $\Fib_i (\I^*A)\ra \I^{*}X$ over~$\I$, and thus to \emph{unbiased} fibration structures on $A\to X$.

\begin{lemma}\label{lemma:classtypeunbiasedfibstruct}
For any map $A\ra X$ in $\cSet$, the map $\Fib(A)\ra X$ in \eqref{diagram:fibU7} is a \emph{classifying type for unbiased fibration structures}: sections of $\Fib(A)\ra X$ correspond bijectively to unbiased fibration structures on $A\ra X$, and the construction is stable under pullback in the expected sense (as in Lemma \ref{lemma:classtypebiasedfibstruct}).
\end{lemma}

\begin{proof} It remains only to check the stability, but since both of the adjoints in $\I^* \dashv \I_* : \cSet/_\I \ra\cSet$ preserve pullbacks, this follows easily from the fact that the classifying types $\Fib_i$ are stable under pullback by Lemma \ref{lemma:classtypebiasedfibstruct}.
\end{proof}

Finally, we can again take $\Fib := \Fib(\VV)$ to now obtain a universal small \emph{unbiased} fibration $\FFib\ra\Fib$ in $\cSet$, as in \eqref{diagram:universalfib}, and the proof can conclude just as in that for Proposition~\ref{prop:classTFib}.  

\begin{proposition}\label{prop:UniversalunbiasedFib}
The map $\FFib\ra\Fib$ just constructed is a \emph{universal small unbiased fibration} in $\cSet$: every small unbiased fibration $A \fib X$ is a pullback of $\FFib \fib \Fib$ along a canonically determined classifying map $X\ra \Fib$.
\begin{equation}\label{diagram:classifyfib2}
\xymatrix{
A \ar@{>>}[d] \ar[r]  \pbcorner & \FFib \ar@{>>}[d] \\
X \ar[r] & \Fib
}
\end{equation}
\end{proposition}

\begin{remark}\label{sliceuniversalfibrations}
Recall from Proposition \ref{prop:familyclassifierforslices} that the universe in the slice category $\cSet/_\I$ is the pullback of the universe $\V$ from $\cSet$ along the base change $\I^* : \cSet \to \cSet/_\I$.  Thus in the construction just given of the classifier $\FFib\ra\Fib$ for unbiased fibrations in $\cSet$  we are first building the classifying type 
\[
\Fib_i(\I^*\VV) \to \I^*\V
\]
 for $\delta$-biased fibration structures on the universal family in $\cSet/_\I$, and then taking a pushforward $\I_*: \cSet/_\I \to \cSet$ to obtain the (base of the) classifier for unbiased fibrations as the pullback along the unit:
\begin{equation}\label{diagram:fibUuniversal}
\xymatrix{
 \Fib(\VV) \ar[d] \ar[r] \pbcorner & \ar[d] \Pi_{i:\I}\Fib_i(\I^*\VV) \\
 \V \ar[r]_{\rho} &  \V^\I
}
\end{equation}
We remark for later reference that this classifying type $\Fib = \Fib(\VV) \to \V$ for unbiased fibration structures can therefore be constructed as the pushforward of the classifier $\Fib_i(\I^*\VV) \to \I^*\V$ for $\delta$-biased fibration structures along the projection $q : \I^*\V =  \I\times \V \to \V$ indicated below.
\begin{equation}\label{diagram:fibUuniversal2}
\xymatrix{
 \Fib_i(\I^*\VV) \ar[d] &  \Fib(\VV) \ar[d] \ar[r] \pbcorner & \ar[d] \Pi_{i:\I}\Fib_i(\VV) \\
 \I^*\V \ar[r]_q \ar[d] \pbcorner & \V \ar[r]_{\rho} \ar[d] &  \V^\I\\
 \I \ar[r] & 1
}
\end{equation}
We record this fact as:
\begin{corollary}\label{cor:Fibaspushforwardalongq}
$\Fib = \Sigma_{\V}\,q_* \Fib_i(\I^*\VV).$
\end{corollary}

The reader may also find it illuminating to reconsider the construction of the universal small unbiased fibration in more type theoretic terms.  
It was defined to be $\FFib\fib\Fib = \Fib(\VV)$, for the universal family $\VV \to \V$, with $\FFib$ the pullback of $\VV \to \V$ along the canonical projection $\Fib(\VV)\to\V$. Since, type theoretically, we have $\VV = \Sigma_{A:\V}A$, by the stability of the classifying type $\Fib(-)$ we can  write $\Fib = \Sigma_{A:\V}\Fib(A)$ so that:
\[
\FFib = \Sigma_{A:\V}\Fib(A)\times A\ \too\ \Sigma_{A:\V}\Fib(A)\ =\ \Fib \,.
\]
\end{remark}


\subsection*{Realignment for fibration structure}\label{sec:realignment}

The realignment for families of Proposition \ref{prop:realignment} will need to be extended to (structured) fibrations. Our approach makes use of the notion of a \emph{weak proposition}.  Informally, a map $P\to X$ may be said to be a weak proposition if it is ``conditionally contractible'', in the sense that it is contractible if it has a section (recall that a proposition may be defined as a fibration that is ``contractible if inhabited'').  More formally, we have the following.

\begin{definition}
A map $P\to X$ is said to be a \emph{weak proposition} if the projection $P\times_X P\to P$ is a trivial fibration.
\begin{equation}\label{diagram:weakprop}
\begin{tikzcd}
P^2 \ar[d,two heads,swap,"{\sim}"] \ar[r]  \pbmark & P \ar[d] \\
P \ar[r] &  X.
\end{tikzcd}
\end{equation}
Note that if either projection is a trivial fibration, then both are.
\end{definition}

As an object over the base, a weak proposition is thus one that ``thinks it is contractible''.  The key fact needed for realignment is the following.

\begin{lemma}\label{lemma:Fibweakprop} For any $A\to X$, the classifying type $\TFib(A) \to X$ is a weak proposition.  Moreover, the same is true for $\Fib(A)\to X$ (both the biased and unbiased versions) if the cofibrations are closed under exponentiation by the interval $\I$.
\end{lemma}

\begin{proof}
Let $A\to X$ and consider the following diagram, in which we have written $A' = \TFib(A)\times_X A$ and $\TFib(A)^2 = \TFib(A)\times_X \TFib(A)$.
\begin{equation}\label{diagram:TFibweakprop}
\begin{tikzcd}
A' \ar[dd]  \ar[rr] \pbmark && A \ar[dd]  \\
& \TFib(A)^2 \ar[ld] \ar[rr]  \pbmark && \TFib(A) \ar[ld] \\
\TFib(A) \ar[rr] && X
\end{tikzcd}
\end{equation}
Since $\TFib$ is stable under pullback (by Lemma \ref{lemma:TFibstable}), we have $\TFib(A)^2 \cong \TFib(A')$, and since $\TFib(A)^2$ has a canonical section, $A' \to \TFib(A)$ is therefore a trivial fibration.  Inspecting the definition of $\TFib(A) = +\mathsf{Alg}(A)$ in \eqref{diag:plualgstru}, we see that if a map $A\to X$ is a trivial fibration, then so is $\TFib(A)\to X$ (since $\eta : A\to A^+$ is always a cofibration). Thus $\TFib(A)^2 \cong \TFib(A') \to\TFib(A)$ is also a trivial fibration.

For $\Fib(A)\to X$, with reference to the construction \eqref{diagram:fibU8} we use the foregoing to infer that $\TFib(\epsilon_A) \to A_\epsilon$ is a weak proposition, and so therefore is its pushforward $F_A = (p_\epsilon)_*\TFib(\epsilon_A)\to X^\I$ along the projection $p_\epsilon : A_\epsilon = X^I\times_X A \to X^\I$, since pushforward clearly preserves weak propositions.  Applying the root $(-)_\I$ preserves trivial fibrations, by the assumption that its left adjoint $(-)^\I$ preserves cofibrations, and so, as a right adjoint, it also preserves weak propositions. Therefore $(F_A)_\I \to (X^\I)_\I$ is a weak proposition, but then so is its pullback along the unit $X\to (X^\I)_\I$, which is $\Fib_i(A)\to X$, the classifier for $\delta$-biased fibration structures.  The same reasoning shows that $\Fib(A) = \rho^*\Pi_{i:\I}\Fib_i(\I^*A)$ (as in \eqref{diagram:fibU7}) is also a weak proposition.
\end{proof}

In light of Lemma \ref{lemma:Fibweakprop} we shall assume as a final axiom on cofibrations:
 
\begin{enumerate}
\item[(C8)] The pathobject functor preserves cofibrations: thus $c:A \cof B$ implies $c^\I:A^\I \cof B^\I$.
\end{enumerate}

Now, by Propositions \ref{prop:UniversalFib} and \ref{prop:UniversalunbiasedFib} we have universal small $\delta$-biased and unbiased fibrations, the former in $\cSet/_\I$, the latter in $\cSet$.  The following remarks apply to both, which we refer to neutrally as $\UU \fib \U$.  The base object $\U$ is (the domain of) the classifying type $\mathsf{Fib}(\VV) \to \V$, where $\VV \to \V$ is the universal small family.  Type theoretically, this object can be written as
\[
\U = \Sigma_{E:\V}\Fib(E)\,,
\] 
which comes with the canonical projection
\[
\U = \Sigma_{E:\V}\Fib(E) \too \V\,.
\] 
In these terms, a fibration $E \fib X$ is a pair $\langle E, e \rangle$, consisting of the underlying family $E\to X$, equipped with a fibration structure $e :\Fib(E)$.
\noindent Lemma \ref{lemma:Fibweakprop}  then allows us to establish the following, which was first isolated in \cite{Shu:15}  (as condition (2'), also see \cite{GSS:22}). It holds for both biased and unbiased fibrations, and will be used in the sequel to ``correct'' the fibration structure on certain maps.

\begin{lemma}[Realignment for fibrations]\label{lemma:realignmentforfibrations}
Given a fibration $F\fib X$ and a cofibration $c : C\cof X$, let $f_c : C \to \U$ classify the pullback $c^*F\fib C$.  Then there is a classifying map $f: X \to \U$ for $F$ with $f\circ c = f_c$.
\begin{equation}\label{diagram:fibrationrealignment}
\begin{tikzcd}
c^*F \ar[dd, two heads] \ar[rd] \ar[rr] && \UU \ar[dd, two heads] \\
& F \ar[dd, two heads] \ar[ru, dotted] & \\
C  \ar[rd, tail,swap, "c"] \ar[rr, near start, "f_c"] && \U  \\
& X \ar[ru, dotted, swap, "f"] &
\end{tikzcd}
\end{equation}
\end{lemma}
\begin{proof}
First, let $|f_{c}|: C\to \V$ be the composite of $f_c: C\to \U$ with the canonical projection $\U\to \V$, thus classifying the underlying family $c^*F \to C$.  Next, let $f_0 : X\to \V$ classify the underlying family $F\to X$.  We may assume that $f_0\circ c = |f_{c}|$ by realignment for families, Proposition \ref{prop:realignment}.
\begin{equation}\label{diagram:fibrationrealignment2}
\begin{tikzcd}
c^*F \ar[dd, two heads] \ar[rd] \ar[rr] && \UU \ar[dd, two heads] \ar[r] & \VV \ar[dd] \\
& F \ar[dd, two heads] \ar[rru, dotted] & &\\
C  \ar[rd, tail,swap, "c"] \ar[rr, near start, "f_c"] && \U \ar[r] & \V \\
& X \ar[rru, dotted, swap, "f_0"] &
\end{tikzcd}
\end{equation}
Since $F\fib X$ is a fibration, there is a lift $f_1 : X \to \U$ of $f_0$ classifying the fibration structure.
We thus have the following commutative diagram in the base of \eqref{diagram:fibrationrealignment2}.
\begin{equation}\label{diagram:realignmentproof}
\begin{tikzcd}
C  \ar[d, tail,swap, "c"] \ar[r,swap, "{f_c}"] \ar[rr, bend left, "{|f_{c}|}"] & \U \ar[r] & \V \ar[d, equals]  \\
X \ar[r, "f_1"]  \ar[rr,swap,bend right, "f_0"] & \U \ar[r]  & \V
\end{tikzcd}
\end{equation}
Now pull $\U\to \V$ back against itself and rearrange the previous data to give (the solid part of) the following, which also commutes.
\begin{equation}\label{diagram:realignmentproof2}
\begin{tikzcd}
C \ar[d, tail, swap, "{c}"] \ar[rr, swap, "{\langle f_{1}c, f_c \rangle}"]  \ar[rrr, bend left, "{f_{c}}"] 
	&& \U\times_\V \U \ar[r, swap, "{\pi_2}"] \ar[d, two heads, swap, "{\pi_1}"] \pbmark & \U \ar[d]  \\
X \ar[rr,swap, "{f_1}"] \ar[rru, dotted, swap, "{f_2}"]  \ar[rrr,swap, bend right, "{f_0}"] && \U \ar[r]  & \V
\end{tikzcd}
\end{equation}
Since $\U=\Fib(\VV)\to\V$ is a weak proposition by Lemma \ref{lemma:Fibweakprop} and (C8), the projection $\pi_1: \U\times_\V \U \fib \U$ is a trivial fibration, so there is a diagonal filler $f_2 : X\to \U\times_\V \U$ as indicated.  Taking $f := \pi_2\circ f_2 : X\to \U\times_\V \U \to \U$ gives another classifying map for the fibration structure on $F\to X$, for which $f\circ c =  f_c$ as required.
\end{proof}

\section{The equivalence extension property}\label{sec:EEP}

The equivalence extension property is closely related to the \emph{univalence} of the universal fibration $\UU\fib\U$ constructed in Section \ref{sec:universalfibration} (see \cite{Shu:15}).  It will be used in Section \ref{sec:FEP} to show that the base object $\U$ is fibrant.  The proof of the equivalence extension property given here is a reformulation of a type-theoretic argument due to Coquand \cite{CCHM:2018ctt}, which in turn is a modification of the original argument of Voevodsky \cite{KLV:21}.  See \cite{Sattler:2017ee} for another reformulation.  

\subsection*{The sliced premodel structure}

We begin by recalling some basic facts and making some simple observations that are well-known in general model categories, but need to be checked again here, because we do not yet have a full model structure. The reader is reminded that the word ``fibration'' unqualified always refers to \emph{unbiased} fibrations as in Definition \ref{def:unbiasedfibration}.  First, for any object $Z \in\cSet$, the slice category $\cSetZ$  inherits the premodel structure of Proposition \ref{prop:FWC} from $\cSet$ via the forgetful functor 
\[
Z_! : \cSetZ  \too \cSet\,.
\]
In more detail:
\begin{definition}\label{def:slicepremodelstructure}
 A map $f : X \to Y$ over $Z$ is a \emph{(trivial) cofibration or (trivial) fibration} over $Z$ just if it is one in $\cSet$ after forgetting the $Z$-indexing via  $Z_!:\cSetZ \to \cSet$. This will be called the (\emph{relative} or) \emph{sliced premodel structure} on $\cSetZ$.  Accordingly, a map $f : X \to Y$ over $Z$ will be called a \emph{weak equivalence} over $Z$ just if it factors over $Z$ as a trivial fibration over $Z$ after a trivial cofibration over $Z$, which therefore holds just if it is a weak equivalence in $\cSet$ after forgetting the $Z$-indexing.
\end{definition}

That the specification in Definition \ref{def:slicepremodelstructure} actually does determine a premodel structure is a consequence of Proposition \ref{prop:FWC}, and the well-known fact that (pre-)model structures are stable under slicing in this way \cite{Hirschhorn:2003mc}.   In more detail:

\begin{lemma}\label{lem:slicepremodellifting1}
A map $f : X \to Y$ over $Z$  is a fibration (respectively, a trivial fibration) over $Z$ if, and only if, it lifts on the right in the slice category $\cSetZ$ against all trivial cofibrations (respectively, cofibrations) over $Z$.
\end{lemma}
\begin{proof}
Let $X \stackrel{f}{\to} Y \stackrel{p_Y}{\too} Z$, regarded as a map in the slice category over $Z$, with $p_X = p_Y \circ f : X \to Z$.  Then by definition $f$ is a fibration in $\cSetZ$ just if $f : X\to Y$ is a fibration in the total category $\cSet$, which holds just if $f$ lifts on the right against all trivial cofibrations $t : A \cof B$ in $\cSet$.  But every lifting problem of the form $t\pitchfork f$ in $\cSet$,
\[\begin{tikzcd}
A \ar[d,swap,tail, "t"] \ar[r, "x" ] & X \ar[d, "f"] \\  
B  \ar[r, swap, "y" ] \ar[ru, dotted] & Y  \ar[d, "{p_Y}"] \\
& Z
 \end{tikzcd}\]
gives rise to a corresponding one over $Z$, just by composing everything with $p_Y : Y \to Z$. Moreover, the evident resulting map $A\stackrel{t}{\cof} B \to Z$ is then a trivial cofibration over $Z$,  and every such lifting problem for $f$ over $Z$ arises in this way.  Finally, the diagonal fillers for the resulting lifting problem in $\cSetZ$ are exactly the diagonal fillers for the original one in $\cSet$. Thus the map $f$ over $Z$ is a fibration over $Z$ just in case it lifts on the right over $Z$ against all trivial cofibrations over $Z$, as claimed.  The case of trivial fibrations and cofibrations is exactly analogous.
\end{proof}

\begin{lemma}\label{lem:slicepremodellifting2}
A map $f : X \to Y$ over $Z$  is a cofibration (respectively, a trivial cofibration) over $Z$ if, and only if, it lifts on the left in the slice category $\cSetZ$ against all trivial fibrations (respectively, fibrations) over $Z$.
\end{lemma}
\begin{proof}
Let $X \stackrel{f}{\to} Y \stackrel{p_Y}{\too} Z$, regarded as a map in the slice category over $Z$, with $p_X = p_Y \circ f : X \to Z$.  Then by definition $f$ is a cofibration in $\cSetZ$ just if $Z_! f : Z_! X\to Z_! Y$ is a cofibration in the total category $\cSet$, which holds just if $Z_! f$ lifts on the left against all trivial fibrations $t : E \fib F$ in $\cSet$.  
But every lifting problem of the form $Z_! f\pitchfork t$ in $\cSet$,
\[\begin{tikzcd}
X \ar[d,swap, "f"] \ar[r, "x" ] & E \ar[d, two heads, "t"] \\  
Y  \ar[r, swap, "y" ] \ar[ru, dotted] \ar[d,swap, "{p_Y}"] & F  \\
Z &
 \end{tikzcd}\]
gives rise to a corresponding one over $Z$ of the form $f\pitchfork Z^*t$, by pulling $t$ back along $Z \to 1$. Moreover, since trivial fibrations are stable under pullback in $\cSet$, the pullback $Z^*t$ is a trivial fibration, and so $Z^*t$ is a trivial fibration over $Z$.  Thus $f$ is a cofibration in $\cSetZ$ if and only if $f\pitchfork Z^*t$ in $\cSetZ$ for all trivial fibrations $t : E\fib F$ in $\cSet$.  

Now observe that for any map $A \stackrel{g}{\to} B \stackrel{p_B}{\too} Z$ over $Z$, with $p_A = p_B\circ g$, the following unit square is a pullback, as indicated below,
\begin{equation}\label{eq:unitgraphpullback}
\begin{tikzcd}
A \ar[d,swap, "g"] \ar[r, "\eta_A" ] \pbmark & Z^*Z_!A  \ar[d, "Z^*Z_!g"] \ar[r, equals] & Z\times A \ar[d, "Z\times g"] \\  
B  \ar[r, swap, "\eta_B" ]   \ar[d] & Z^*Z_!A   \ar[d] \ar[r, equals] & Z\times B  \ar[d] \\
Z  \ar[r, equals] & Z  \ar[r, equals] & Z,
 \end{tikzcd}
 \end{equation}
because the graph $\eta_A = \pair{p_A, 1_Z} : A \to Z\times A$ is a pullback of $\Delta_A = \pair{1_Z, 1_Z} : Z \to Z\times Z$ along $1_Z \times p_A : Z\times A \to Z\times Z$, and similarly for $\eta_B$.  Thus in particular, every trivial fibration $A \stackrel{g}{\fib} B \stackrel{p_B}{\too} Z$ over $Z$ is a pullback over $Z$ of one of the form $Z^*g : Z^*A \fib Z^*B$ for a trivial fibration $g : A\fib B$ in $\cSet$.  Therefore $f$ is a cofibration in $\cSetZ$ if and only if $f\pitchfork g$ in $\cSetZ$ for all trivial fibrations $g : A\fib B$ in $\cSetZ$, as claimed. 
The case of trivial cofibrations and fibrations is exactly analogous.
\end{proof}

Since factoring a map in the slice category is evidently given simply by factoring it after forgetting the indexing, we now have:

\begin{proposition}\label{prop:slicepremodelstructure}
The specification in Definition \ref{def:slicepremodelstructure} determines a premodel structure on $\cSetZ$ for any object $Z \in \cSet$.
\end{proposition}

The reader is warned that when $Z = \I$ there is a possibility of confusion with the $\delta$-biased fibrations in $\cSet/_\I$, which do not in general agree with the $\I$-sliced (unbiased) fibrations.  

In order to verify the axioms (C1)-(C8) for cofibrations, let $Z^*1 \rightrightarrows Z^*\I$ in $\cSetZ$ be the result of pulling the interval $1 \rightrightarrows \I$ back along $Z\to 1$, to obtain a bipointed object in $\cSetZ$ that we shall write as,
\begin{equation}\label{eq:relativeinterval}
\delta_0, \delta_1 : 1_Z \rightrightarrows \I_Z\,.
\end{equation}
Observe that $1_Z + 1_Z \cong Z^*1 + Z^*1 \cof Z^*\I$ since the pullback functor  $Z^* : \cSet\to \cSetZ$ preserves (co)limits and cofibrations.   

\begin{proposition}
Taking $\delta_0, \delta_1 : 1_Z \rightrightarrows \I_Z$ as an interval, the axioms  (C1)-(C8) for cofibrations are satisfied in $\cSetZ$ 
\end{proposition}

\begin{proof}(sketch)
The (relative) cofibration classifier in $\cSetZ$ is the pullback $Z^*t : Z^*1 \cof Z^*\Phi$, which we shall write as
\begin{equation}\label{eq:relativecofclassifier}
t_Z : 1_Z \cof \Phi_Z\,.
\end{equation}
For axiom (C8), observe that for a map $c : A \to B$ in $\cSetZ$, the exponential $c^{\I_Z} : A^{\I_Z}  \to A^{\I_Z} $ in $\cSetZ$ fits into a unit pullback square of the form \eqref{eq:unitgraphpullback},
\begin{equation}\label{eq:unitgraphpullback2}
\begin{tikzcd}
A^{\I_Z}  \ar[d,swap, "c^{\I_Z}"] \ar[r, "\eta_A" ] \pbmark & Z^*Z_!(A^{\I_Z})   \ar[d, "Z^*Z_!(c^{\I_Z})"] \ar[r, "\cong"] 
	& Z^*\big(Z_!(A)^{\I} \big) \ar[d, "Z^*\big(Z_!(c)^{\I} \big)"] \\  
B^{\I_Z}   \ar[r, swap, "\eta_B" ]   & Z^*Z_!(B^{\I_Z})   \ar[r,swap, "\cong"] & Z^*\big(Z_!(B)^{\I} \big)
 \end{tikzcd}
 \end{equation}
So if $c$ is a cofibration, so is $c^{\I_Z}$.
The other axioms are routine.
\end{proof}

\begin{lemma}\label{lemma:JTslice}
For any cubical set $Z$, we have the following relative versions of the pushout-product and pullback-hom conditions involving the interval in the slice category $\cSetZ$.
\begin{enumerate}
\item If $c : A \cof B$ is a cofibration in $\cSetZ$, then the pushout-product formed in $\cSetZ$ with $\delta_0 : 1_Z \cof \I_Z$, written
\[
c\otimes_Z \delta_0 : B +_{A} (A\times_Z \I_Z) \too B\times_Z \I_Z\,,
\]
is a trivial cofibration (and similarly for $\delta_1 : 1_Z\cof \I_Z$). 
\item If $f : X \fib Y$ is a fibration in $\cSetZ$, then the pullback-hom formed in $\cSetZ$ with $\delta_0 : 1_Z \cof \I_Z$, written
\[
\delta_0\Rightarrow_Z f : X^{\I_Z} \too Y^{\I_Z}\times_Z X \,,
\]
is a trivial fibration (and similarly for $\delta_1 : 1_Z\cof \I_Z$). 
\end{enumerate}
\end{lemma}

\begin{proof}
For (1),  the pushout-product $c \otimes_Z \delta_0 : D \to B \times_Z \I_Z$ over $Z$ is equal to the (non-relative) pushout-product $c \otimes \delta_0 : D \to B \times \I$, because $Z^*\delta_0 : Z^*1 \to Z^*\I$  is constant over $Z$, so
\[
B\times_Z \I_Z \cong B \times \I\,,
\]  
and similarly for $A$ (and pushouts in the slice are created by the forgetful functor $\cSetZ \to \cSet$).  Thus, briefly,
\[
c \otimes_Z Z^*\delta_0 = c \otimes \delta_0 \,,
\]
which is indeed a trivial cofibration.

(2) follows from (1) and lemma \ref{lem:slicepremodellifting2}, together with the usual adjunction between $\otimes_Z$ and $\Rightarrow_Z$.
\end{proof}

In order to apply the results on weak equivalences from Section \ref{sec:weq} in arbitrary slice categories $\cSetZ$ we shall also require  the notions of homotopy equivalence over $Z$ and weak homotopy equivalence over $Z$.  We first use the relative interval $\delta_0, \delta_1 : 1_Z \rightrightarrows \I_Z$ \eqref{eq:relativeinterval} to define homotopy between maps over $Z$ in the expected way, namely:

\begin{definition}\label{homotopyinaslice}
For any object $Z$ and maps $f, g: X\rightrightarrows Y$ in $\cSetZ$, a \emph{homotopy over $Z$}, written 
\[
\vartheta : f \sim_Z g\,,
\]
 is a map over $Z$,
\[
\vartheta : \I_Z\times_Z {X} \too Y,
\]
such that $\vartheta \circ \iota_0 = f$ and $\vartheta \circ \iota_1 = g$, 
\begin{equation}\label{diagram:defhomotopy2}
\xymatrix{
X \ar[r]^-{\iota_0} \ar[rd]_f & \I_Z\times_Z\!{X} \ar[d]^-{\vartheta} & X, \ar[l]_-{\iota_1} \ar[ld]^g \\
& Y &
}
\end{equation}
where, as usual, $\iota_0, \iota_1$ are the canonical inclusions into the ends of the cylinder,
\[
\xymatrix{
\iota_\epsilon : X \cong 1_Z\times_Z X \ar[rr]^-{\delta_\epsilon\times_Z X} && \I_Z\times_Z X\,,} \qquad \epsilon = 0,1 .
\]
\end{definition}

\begin{lemma}\label{lem:homotopyinthesliceishomotopy}
For any object $Z$ and maps $f, g: X\rightrightarrows Y$ in $\cSetZ$, a homotopy over $Z$ determines a homotopy of the underlying maps by applying the functor $Z_! : \cSetZ \to \cSet$ that forgets the $Z$-indexing,
\[
\vartheta : f\sim_Z g \quad\ \mapsto\quad  Z_! \vartheta : Z_! f\sim Z_! g\,.
\]
\end{lemma}
\begin{proof} 
 Consider the following diagram depicting a homotopy $\vartheta : f\sim_Z g$ over $Z$. 
\[\begin{tikzcd}
&& X  \ar[d, shift left, "\iota_1"] \ar[d, shift right, swap,  "\iota_0"] \ar[r,swap, "f"] \ar[r, shift left=1ex, "g"] & Y \ar[dd] \\  
\I \ar[d]  & \ar[l]  \I \times Z \ar[d] &  \ar[l]   \I_Z \times_Z X  \ar[d] \ar[ru,swap, "\vartheta"] &  \\
1 & Z\ar[l] & \ar[l] X  \ar[r] & Z
 \end{tikzcd}\]
 Since the lower left two squares are pullbacks, we have $\I_Z \times_Z X \cong \I\times X$. So applying $Z_!$ to $\vartheta$ results in a homotopy $Z_!\vartheta : Z_! f \sim Z_!g$.
 
Note that an arbitrary homotopy $\varphi : f \sim  g$ will not result in one over $Z$, however, unless $\varphi$ commutes with the indexing maps to $Z$.
\end{proof}

\begin{proposition}\label{prop:homotopybetweenfibequivrel}
For any object $Z$, the relation of homotopy over $Z$ between maps $f, g: X\rightrightarrows Y$ over $Z$ is preserved by pre- and post-composition. If $X\fib Z$ and $Y\fib Z$ are both fibrations, then the relation $f \sim_Z g$ of maps between them is an equivalence relation.
\end{proposition}

The proof is essentially the same as the corresponding one for homotopy over $1$, Proposition \ref{prop:homotopyintofibequivrel}, with the exception that both $X$ and $Y$ are required to be fibrant objects over $Z$, so that the exponential $Y^X$ over $Z$ is also a fibration $Y^Z  \fib Z$ (by Corollary \ref{cor:unbiasedPi}).

Next we define a \emph{connected components} functor on the full subcategory $\Fib_Z \hook \cSetZ$ of fibrations over $Z$,
\[
(\pi_0)_Z : \Fib_Z \ra\Set\,,
\]
by taking the global sections of a fibration $F\fib Z$, modulo the relation $\sim_Z$ of homotopy over $Z$. In more detail,  for $F\fib Z$ in $\Fib_Z$ let $(\pi_0)_Z(F)$ be the coequalizer,
\begin{equation}\label{eq:coeqpi0}
\Hom_Z(\I_Z , F) \rightrightarrows \Hom_Z(1_Z , F) \to (\pi_0)_Z(F) \,,
\end{equation}
where the two maps are given by precomposition with the interval $1_Z \rightrightarrows \I_Z$ over $Z$, and the $\hom$-sets are those in $\cSetZ$.

For fibrations $X\fib Z$ and $F\fib Z$ we then again have 
\[
(\pi_0)_Z(F^X) = \hom_{Z}(X,F)/\!\!\sim_Z\,,
\]  
so $(\pi_0)_Z(F^X)$ is the set  $[X, F]_Z$ of $Z$-homotopy equivalence classes of maps $X\ra F$ over $Z$.
For maps over a base object $Z$, we can then define the notions of homotopy equivalence over $Z$ and, between fibrations, weak homotopy equivalence over $Z$ as before (cf.~Section \ref{sec:weq}):

\begin{definition}\label{def:slicedhewe}
Let $Z$ be any object in $\cSet$, and let $X \to Z$ and $Y \to Z$ be regarded as objects over $Z$.
\begin{enumerate}
\item A map $f : X\to Y$ over $Z$ is a \emph{homotopy equivalence over $Z$} 
if there is a map $g : Y\to X$ over $Z$ and two homotopies over $Z$,
\[
\vartheta : g\circ f \sim_Z 1_X,\qquad \varphi : f\circ g \sim_Z 1_Y\,.
\]
\item For $X\fib Z$ and $Y\fib Z$ fibrations, a map $f : X\to Y$ over $Z$ is a \emph{weak homotopy equivalence over $Z$} if for every fibration $F\fib Z$, the precomposition map over $Z$,
\[
F^f : F^Y \to F^X\,,
\]
is bijective on connected components, 
\[
(\pi_0)_Z(F^f) : (\pi_0)_Z(F^Y) \cong (\pi_0)_Z(F^X)\,,
\]
where the indicated exponentials are taken in the slice category. 
\end{enumerate}
\end{definition} 

The proof of the following is analogous to that of the corresponding facts for the case $Z=1$ (Lemmas \ref{lemma:HE342} and \ref{lemma:WHE342}).

\begin{lemma}\label{lemma:HE342sliced}
The  homotopy equivalences over any object $Z$ satisfy the 3-for-2 condition, as do the weak homotopy equivalences over $Z$.
\end{lemma}

\begin{proposition}\label{prop:weqisheoverZ}
For any object $Z$ and fibrations $X\fib Z$ and $Y\fib Z$, the following conditions are equivalent for any map $f : X\to Y$ over $Z$.
\begin{enumerate}
\item $f : X\to Y$ is a weak equivalence over $Z$,
\item $f : X\to Y$ is a homotopy equivalence over $Z$, 
\item $f : X\to Y$ is a weak homotopy equivalence over $Z$, 
 \end{enumerate}
\end{proposition}

\begin{proof}
Let $f : X\to Y$ be a weak equivalence.  Factor $f = tf\circ tc$ with a trivial cofibration $tc : X \cof F$ followed by a trivial fibration $tf : F \fib Y$, both of which are then also over $Z$.  
\[\begin{tikzcd}
	& F  \ar[dd, two heads]  \ar[dr, two heads, "tf"]  &  \\  
X  \ar[ur, tail, "tc"]  \ar[dr, two heads, swap, "p"] \ar[rr, pos=0.3, "f"] & & Y \ar[dl, two heads, "q" ]   \\
 & Z & 
 \end{tikzcd}\]
The proof of Proposition \ref{prop:CofbetweenFibWEWHEHE} $(1\!\Rightarrow\!2)$ now applies over $Z$, \emph{mutatis mutandis},  to show that $tc : X\cof F$ is a homotopy equivalence over $Z$. Similarly, the proof of Lemma \ref{lem:TFibisHE} also works over $Z$ to show that $tf: F\fib Y$ is a homotopy equivalence over $Z$. Thus $f = tf\circ tc$ is a homotopy equivalence over $Z$.

Any homotopy equivalence over $Z$ is clearly a weak homotopy equivalence over $Z$, by the same proof as for Lemma \ref{lemma:HEisWHE} (using the fact that $X\fib Z$ and $Y\fib Z$ in order to form the required exponentials).

If $f : X\to Y$ is a weak homotopy equivalence, then factor it as $f = tf\circ c$ with a cofibration $c : X \cof F$ followed by a trivial fibration $tf : F \fib Y$, both of which are also over $Z$.   We thus just need to show that $c : X \cof F$ is a trivial cofibration.  As in the first step, $tf : F \fib Y$ is a homotopy equivalence over $Z$, whence a weak homotopy equivalence by the second step, and so by 3-for-2 for weak homotopy equivalences over $Z$, Lemma \ref{lemma:HE342sliced}, $c : X \cof F$ is also a weak homotopy equivalence over $Z$.  Now, as in the proof of Proposition \ref{prop:CofbetweenFibWEWHEHE}, factor $c = f\circ tc$ as a trivial cofibration $tc : X \cof C$ followed by a fibration $f : C \fib F$, both over $Z$.   By steps 1 and 2, $tc : X \cof C$ is then a weak homotopy equivalence over $Z$. By 3-for-2 for weak homotopy equivalences, Lemma \ref{lemma:HE342sliced}, $f : C \fib F$ is also a weak homotopy equivalence over $Z$.  It remains to show that the fibration $f : C \fib F$ is a weak equivalence.  This follows by repeating the reasoning for Lemma \ref{lemma:FibWHEfibCodTFib}, and the results leading up to it, over $Z$.
\end{proof}

Using Lemma \ref{lemma:HE342sliced} we now have:

\begin{corollary}\label{cor:342weqfiboverZ}
For any object $Z$, the weak equivalences between fibrations into $Z$ satisfy the 3-for-2 condition.
\end{corollary}

\begin{remark}\label{remark:relativepremodelneeds342}
Our immediate goal has been to show Corollary \ref{cor:342weqfiboverZ}, which will be used to establish the equivalence extension property.  Proposition \ref{prop:slicedheweequiv} below, which assumes the fibration extension property, will emphatically \emph{not} be used in the sequel, but is included here simply to complete the study of the relative (pre)model structure.
\end{remark}

\begin{lemma}\label{lemma:heovertohe}
Let $f : X\to Y$ be any map over $Z$. 
\begin{enumerate}
\item If $f : X\to Y$ is homotopy equivalence over $Z$, then $Z_!f : Z_! X \to Z_! Y$ is a homotopy equivalence.
\item If $X\fib Z$ and $Y\fib Z$ are fibrations and $f : X\to Y$ is weak homotopy equivalence over $Z$, then $Z_!f : Z_! X \to Z_! Y$ is a weak homotopy equivalence.
\end{enumerate}
\end{lemma}
\begin{proof}
(1) is immediate from the fact that $Z_!$ preserves homotopies, Lemma \ref{lem:homotopyinthesliceishomotopy}.
For (2), let $f : X\to Y$ be a weak homotopy equivalence over $Z$ between fibrations $X\fib Z$ and $Y\fib Z$, and let $K$ be any fibrant object in $\cSet$.  
Consider the internal precomposition map, 
\[
K^{Z_!f} : K^{Z_! Y} \to K^{Z_! X}\,,
\]
which we would like to show is a bijection under $\pi_0 : \cSet \to \Set$.
Since $K\fib1$ is a fibration, so is its pullback $Z^*K \fib Z$.  Therefore, since $f : X\to Y$ is weak homotopy equivalence over $Z$, 
the precomposition map over $Z$,
\[
(Z^*K)^f : (Z^*K)^Y \to (Z^*K)^X\,,
\]
is bijective on connected components, 
\[
(\pi_0)_Z\big((Z^*K)^f\big) : (\pi_0)_Z\big((Z^*K)^Y\big) \cong (\pi_0)_Z\big((Z^*K)^X\big)\,.
\]
But now observe that in the coequalizer \eqref{eq:coeqpi0} that defines $(\pi_0)_Z\big((Z^*K)^X\big)$, we have 
\[\begin{split}
\Hom_Z(1_Z , (Z^*K)^X) \cong  \Hom_Z(X, (Z^*K))\\
 \cong \Hom(Z_!X, K)  \cong \Hom(1, K^{Z_!X})\,,
\end{split}\]
and similarly 
\[\begin{split}
\Hom_Z(\I_Z , (Z^*K)^X) \cong  \Hom_Z(Z^*\I \times X, Z^*K) \\
 \cong \Hom_Z(\I \times Z_! X, K) \cong \Hom(\I, K^{Z_!X})\,.
\end{split}\]
Thus $(\pi_0)_Z((Z^*K)^X) \cong (\pi_0)(K^{Z_!X})$, and the same is true with $Y$ in place of $X$.   So  $K^{Z_!f} : K^{Z_!Y} \to K^{Z_!X}$ is also bijective on connected components.
\end{proof}

\begin{proposition}\label{prop:slicedheweequiv}
Let $Z$ be any object in $\cSet$ and $X \fib Z$ and $Y \fib Z$ fibrations.  For any map $f : X\to Y$ over $Z$,  the following conditions are equivalent, \emph{assuming the fibration extension property, Corollary \ref{cor:FEP}}.
\begin{enumerate}
\item $f : X\to Y$  is a weak equivalence over $Z$. 
\item $f : X\to Y$   is a  homotopy equivalence over $Z$. 
\item $f : X\to Y$   is a weak homotopy equivalence over $Z$. 
\item $Z_!f : Z_!X\to Z_!Y$ is a weak equivalence. 
\item $Z_!f : Z_!X\to Z_!Y$ is a homotopy equivalence in $\cSet$. 
\item $Z_!f : Z_!X\to Z_!Y$ is a weak homotopy equivalence in $\cSet$. 
\end{enumerate}
\end{proposition} 

\begin{proof}
In Proposition \ref{prop:weqisheoverZ}  we showed the implications $1\Leftrightarrow 2 \Leftrightarrow 3$.
We also have $1\Leftrightarrow 4$ by definition, and $4\Rightarrow 6$ and $5 \Rightarrow 6$ by Lemmas \ref{lem:WEimpliesWHE} and \ref{lemma:HEisWHE}.  Moreover, by Lemma \ref{lemma:heovertohe} we have $2\Rightarrow 5$ and $3\Rightarrow 6$.  
Thus all 6 conditions will be equivalent once we have $6\Rightarrow 4$, which follows from Proposition \ref{prop:WHEiffWE} and the fibration extension property, Corollary \ref{cor:FEP}. 
\end{proof}

\subsection*{Pathobject factorizations.}

For any map $f :X\ra Y$ in $\cSet$, recall the \emph{pathobject factorization} $f = t\circ s$ indicated below.
\begin{equation}\label{diag:pathspace_factorization}
\xymatrix{
X \ar@/_1.5pc/[ddd]_= \ar[d]^s \ar[rr]^f \pbcorner && Y \ar[d]_r \ar[rrd]^=\\
P_f  \ar[dd]^{p_f} \ar[rr]^{p_0^*f} \ar@{.>}@/_1.5pc/[rrrr]_<<<<<<<<<<<<<<<t \pbcorner && Y^\I \ar[dd]_{p_0} \ar[rr]^-{p_1} && Y \\
\\
X \ar[rr]_f  &&  Y &
}
\end{equation}
Here $p_0, p_1$ are the evaluations $Y^{\delta_{0}}, Y^{\delta_{1}} : Y^\I \to Y$ at the endpoints $\delta_0, \delta_1 : 1\ra\I$, and let $r:= Y^!$ for $! : \I\ra 1$, so that $p_0r = p_1r = 1_Y$.  Then let  $p_f := f^*p_0 : P_f \ra Y$, the pullback of $p_0$ along $f$, and $s:= f^*r : X\ra P_f$ (as a map over $X$).  Finally, let $t:= p_1\circ p_0^*f : P_f \ra Y$ be the indicated horizontal composite. 

We then have the following facts:
\begin{enumerate}
\item The retraction $p_0\circ r = 1_Y$ pulls back along $f$ to a retraction ${p_f}\circ{s}=1_X$.

\item If $Y$ is a fibrant object, then $p_0 , p_1 : Y^\I \ra Y$ are both trivial fibrations, by Proposition \ref{prop:sanitycheck}.  

\item If $X$ and $Y$ are both fibrant then $t= p_1\circ p_0^*f : P_f \ra Y$ is a fibration.  This can be seen by factoring the maps $p_0, p_1 :Y^\I \rightrightarrows  Y$ through the product projections as
\[
\pi_0\circ p,\ \pi_1\circ p : Y^\I \ra Y\times Y \rightrightarrows Y
\]
where $p = (p_0, p_1)$, and then interpolating the pullback $(f,1_Y) : X\times Y\ra Y\times Y$ into \eqref{diag:pathspace_factorization} as indicated below.
\begin{equation}\label{diag:pathspace_factorization2}
\xymatrix{
& X \ar[d]^s \ar[rr]^f \pbcorner && Y \ar[d]_r \\
& \ar[ld]_t P_f  \ar[d]^{f^*p} \ar[rr]^{p_0^*f} \pbcorner && Y^\I \ar[d]_{p} \ar[rd]^{p_1}\\
Y & \ar[l]^-{\pi'_1} X\times Y\ar[d]_-{\pi'_0} \ar[rr]_{(f,1_Y)} \pbcorner  &&  Y\times Y \ar[d]^{\pi_0} \ar[r]_-{\pi_1} & Y \\
& X \ar[rr]_f  &&  Y &
}
\end{equation}
The second factor $t = p_1\circ p_0^*f : P_f \ra Y$ now appears also as $\pi_1\circ(f,1_Y)\circ f^*p$, which is equal to the pullback $f^*p: P_f \ra X\times Y$ followed by the second projection $\pi'_1 : X\times Y \ra Y$ (which is not a pullback).  But if $Y$ is fibrant, then $p: Y^\I \ra Y\times Y$ is a fibration by the ${\otimes} \dashv {\Rightarrow}$ adjunction, since $p = \del \Rightarrow Y$ (this is just as in Proposition \ref{prop:sanitycheck}, but with the cofibration $\del : 1+1 \cof \I$ in place of the trivial cofibration $\delta_\epsilon : 1 \to \I$).
Therefore the pullback  $f^*p$ is also a fibration.  And if $X$ is fibrant, then the second projection $\pi'_1:X\times Y \ra Y$ is a fibration. Thus in this case, $t = \pi'_1\circ f^*p : P_f \ra Y$ is a fibration, as claimed.
\end{enumerate}

Summarizing (1)-(3):

\begin{lemma}\label{lemma:pathspace_factorization} 
For any map $f : X\ra Y$ 
there is a factorization $f = t\circ s$,
\begin{equation}\label{diag:pathspace_factorization3}
\xymatrix{
X \ar[r]^s \ar[rd]_f & \ar@{.>}@/_1.5pc/[l]_{p_f} P_f \ar[d]^t \\
& Y
}
\end{equation} 
in which
\begin{enumerate}
\item $s$ is a section of a map $p_f : P_f \to X$,
\item  if $Y$ is fibrant, then $p_f$ is a trivial fibration,
\item if both $X$ and $Y$ are fibrant, then $t$ is a fibration.
\end{enumerate}
Note that the retraction $p_f : P_f \ra X$ of $s$ is not over $Y$.
\end{lemma}

Next, if $f:X\ra Y$ is a map over any base object $Z$ in $\cSet$, we can use the same factorization $f = t\circ s$ to get a factorization in the slice category over $Z$,
\begin{equation}\label{diag:pathobjectoverabase}
\xymatrix{
&  P_f  \ar[rd]^t & \\
X \ar[rr]^f \ar[rd]_{p_X} \ar[ru]^{s} && Y \ar[ld]^{p_Y} \\
& Z
}
\end{equation}
with $p_Y \circ t : P_f \to Z$; however, the maps $s, t$ will no longer have the properties stated in Lemma \ref{lemma:pathspace_factorization}, because \eg\ $p_0: Y^\I \to Y$ need not be a trivial fibration, even when $p_Y : Y\to Z$ is a fibration, since the object $Y$ need not be fibrant if the base $Z$ is not fibrant.

To remedy this, we can instead build a \emph{fiberwise pathobject factorization} by using the relative pathobject $Y^{\I_Z} \to Z$, where the indicated exponential is taken in the slice over $Z$, and the interval object $\I_Z$ occurring in the exponent is the relative one from \eqref{eq:relativeinterval}, \ie\ the result of pulling the interval $\I$ back from $\cSet$ along $Z\to 1$.  The pathspace factorization is then constructed as in \eqref{diag:pathspace_factorization}, but now in the slice $\cSetZ$, using the pulled back interval $1_Z\rightrightarrows \I_Z$. Moreover, the resulting factorization $f=t\circ s:X\ra P_f \ra Y$ is then stable under pullback along any map $g : Z' \ra Z$, in the sense that $g^*(Y^{\I_Z}) \cong g^*(Y)^{\I_{Z'}}$ and so $g^*P_f = P_{g^*f}$, where $g^*f : g^*X \ra g^* Y$, and similarly for the factors $g^*s$ and~$g^*t$.  

In more detail, let us review the foregoing steps in the relative case, with reference to the following diagram.
\begin{equation}\label{diag:sliced_pathspace_factorization}
\xymatrix{
X \ar@/_1.5pc/[ddd]_= \ar[d]^s \ar[rr]^f \pbcorner && Y \ar[d]_r \ar[rrd]^=\\
P_f  \ar[dd]^{p_f} \ar[rr]^{p_0^*f} \ar@{.>}@/_1.5pc/[rrrr]_<<<<<<<<<<<<<<<t \pbcorner && Y^{\I_{Z}} \ar[dd]_{p_0} \ar[rr]^-{p_1} && Y \\
\\
X \ar[rd]_{p_X} \ar[rr]_f  && \ar[ld]^{p_Y} Y &\\
& Z &
}
\end{equation}
\begin{enumerate}

\item The exponential of $p_Y : Y\to Z$, taken in $\cSetZ$, by the \emph{constant} maps $Z^*\delta_\epsilon : Z^*1\to Z^*\I$, which we write as $\delta_\epsilon : 1_Z\to \I_Z$, are now maps $p_\epsilon := Y^{\delta_\epsilon} : Y^{\I_{Z}} \to Y$ over $Z$, for $\epsilon = 0,1$.  The retraction $p_0\circ r = 1_Y$ (with $r$ defined accordingly) is now also over $Z$, and it still pulls back along $f$ to a retraction $p_f \circ s =1_X$, also over $Z$.

\item If $p_Y : Y\fib Z$ is a fibration, then the maps $p_0 , p_1 : Y^{\I_Z} \ra Y$ over $Z$ are again trivial fibrations by Lemma \ref{lemma:JTslice}, since these are pullback-homs  over $Z$ of the form $\delta_\epsilon \Rightarrow_Z Y$. 

\item If $X\fib Z$ and $Y\fib Z$ are both fibrations, then for the same reason $t= p_1\circ p_0^*f : P_f \ra Y$ is a fibration. 
\end{enumerate}

Again, summarizing (1)-(3) in the relative case:

\begin{lemma}\label{lemma:relative_pathspace_factorization} 
For any map $f : X\ra Y$  over any base $Z\in \cSet$, 
there is a stable factorization $f = t\circ s$ over $Z$,
\begin{equation}\label{diag:pathspace_factorization4}
\xymatrix{
X \ar[d] \ar[r]^s \ar[rd]_f & \ar@{.>}@/_1.5pc/[l]_{p_f} P_f \ar[d]^t \\
Z & \ar[l] Y
}
\end{equation} 
in which
\begin{enumerate}
\item $s : X \to P_f$ is a section of a map $p_f : P_f \to X$ over $Z$,
\item  if $Y\to Z$ is a fibration, then $p_f : P_f \to X$ is a trivial fibration,
\item if both $X\to Z$ and $Y\to Z$ are fibrations, then $t : P_f \to Y$ is a fibration.
\end{enumerate}
Note that the retraction $p_f : P_f \ra X$ of $s$ is not over $Y$.
\end{lemma}

The following fact concerning just the cofibration weak factorization system will also be needed.

\begin{lemma}\label{lemma:etaTF}
Let $p: E \onto B$ be a trivial fibration and $c : C\mono B$ a cofibration.  Then the unit $\eta:E \ra c_*c^*E$ over $B$ of the base change along $c$,
\[
c^*\dashv c_* : \cSet/_C \too \cSet/_B
\]
is also a trivial fibration.
\end{lemma}

\begin{proof}
Regarding $c : C \cof B$ as a subobject $C\cof 1_B$ in $\cSet/_B$, the unit map $\eta : E \ra c_*c^*E = E^C$ is the pullback-hom $c\!\Rightarrow_B\! p$ in the slice category over $B$, as shown below.
\begin{equation}\label{diagram:etaispghom}
\xymatrix{
E^{1_B} \ar@/_4ex/ [rdd]_{p^{1_B}} \ar@{.>}[rd]^{c\, \Rightarrow_B\, {p}} \ar@/^4ex/ [rrd]^{E^c} && \\
& B^\I \times_{B^C} E^C \ar[d] \ar[r] & E^C \ar[d]^{p^C} \\
& B^{1_B} \ar[r]_{B^c} &  B^C 
}
\end{equation}
We use the fact that in $\cSet/_B$ we have $B^c : B^{1_B} \cong 1_B \cong B^C$ and so
\[
(c\Rightarrow_B p) = E^c : E \too E^C\,,
\]
which is indeed $\eta : E \ra c_*c^*E = E^C$.

Now for any cofibration $A \stackrel{a}{\cof} X\to B$ over $B$, by Lemma \ref{lemma:Leibniz} we have an equivalence of diagonal filling conditions in $\cSet/_B$,
\[
a \pitchfork (c\!\Rightarrow_B\!p)  \quad\text{iff}\quad (a\!\otimes_B\!c) \pitchfork p.
\]
But since $c : C\mono B$ is a cofibration, $a\otimes_B c$ is also a cofibration, since $a : A \cof X$ is one, and by axiom (C6), cofibrations are closed under pushout-products.  Thus $(a\otimes_B c) \pitchfork p$ indeed holds, since $p$ is a trivial fibration.
\end{proof}

\begin{proposition}[Equivalence extension property]\label{prop:EEP}
Weak equivalences extended along cofibrations in the following sense: given a cofibration $c:C' \mono C$ and fibrations $A'\onto C'$ and $B\onto C$, and a weak equivalence $w':A' \simeq c^*B$ over $C'$,
\begin{equation}\label{diag:EEP}
\xymatrix{
A' \ar@{->>}[dd] \ar[rd]_{\sim}^{w'} \ar@{..>}[rr] && A \ar@{..>>}[dd] \ar@{..>}[rd]_{\sim}^{w} \\
& c^*B \ar@{->>}[ld] \ar[rr]  && B \ar@{->>}[ld] \\
C' \ar@{>->}[rr]_c && C
}
\end{equation}
there is a fibration $A\onto C$ and a weak equivalence $w: A \simeq B$ over $C$ that pulls back along $c:C' \mono C$ to $w'$, so $c^*w = w'$.

\end{proposition}
\begin{proof}
Call the given fibration $q:B\ra C$ and let $b:=  q^*c : c^*B \ra B$ be the indicated pullback, which is thus also a cofibration. Let $w := b_*w' : A\ra B$ be the pushforward of $w'$ along $b$.  Composing $w$ with $q$ gives the map $p:= q\circ w:A\ra C$.  Since $b$ is monic, we indeed have $b^*w = w'$, thus filling in all the dotted arrows in \eqref{diag:EEP}.  Note moreover that $c^*w =  b^*w = w'$, as required. It remains to show that $p:A\ra C$ is a fibration and $w : A\ra B$ is a weak equivalence.   
\begin{equation}\label{diag:EEPfactored}
\xymatrix@=1.5em{
A' \ar@{->>}[dddd] \ar[rrdd]_{\sim}^{w'} \ar@{>->}[rrrr] && && A \ar[dddd]_>>>>>>>>>>>>p \ar[rrdd]^{w} \ar[rr]^s && \ar@{.>}@/_1.5pc/[ll]_{p_w} P_w \ar[dd]^t \\
\\
&& c^*B \ar@{->>}[lldd] \ar@{>->}[rrrr]_<<<<<<<<<<b  &&&& B \ar@{->>}[lldd]^q  \\
\\
C' \ar@{>->}[rrrr]_c &&&& C &&
}
\end{equation}
Let us name $p':= c^*p : A' \ra C'$ and $B' := c^*B$ and $q':= c^*q$.
Now let $w=t\circ s$ be the (relative) pathspace factorization \eqref{diag:sliced_pathspace_factorization} of $w$, as a map over $C$.  Since $q:B \ra C$ is a fibration, by Lemma \ref{lemma:relative_pathspace_factorization}, we know that $s : A\ra P_w$ has a retraction $p_w : P_w \ra A$ over $C$ which is a trivial fibration.  

The pathspace factorization $w=t\circ s : A \ra P_w \ra B$ is stable under pullback along $c$, providing a pathspace factorization of $c^*w = w'=t'\circ s' : A' \ra P_{w'} \ra B'$ over $C'$.  Since both $p'$ and $q'$ are fibrations, the retraction $p_{w'}: P_{w'} \ra A'$ is a trivial fibration, and now $t' : P_{w'} \ra B'$ is a fibration.
\begin{equation}\label{diag:EEPfactored2}
\xymatrix@=1.5em{
A' \ar@{->>}[dddd]_{p'} \ar[rrdd]_{\sim}^{w'}  \ar[rr]^{s'} && \ar@{.>>}@/_1.5pc/[ll]_{p_{w'}} P_{w'} 
	\ar@{->>}[dd]^{t'} \\
\\
&& B' \ar@{->>}[lldd]^{q'} \\
\\
C' &&
}
\end{equation}
Thus the composite $q'\circ t' : P_{w'} \ra B' \ra C'$ is a fibration and therefore, by the retraction over $C'$ with the trivial fibration $p_{w'}$, we have that $s' : A' \ra P_{w'}$ is a weak equivalence, by 3-for-2 for weak equivalences between fibrations, Corollary \ref{cor:342weqfiboverZ}.  For the same reason, $t'$ is then a weak equivalence, and therefore a trivial fibration.

Since $t' = c^*t = b^*t$ is a trivial fibration, its pushforward $b_*b^*t$ along $b$ is also one, by Corollary \ref{cor:plusalgprops}.  Moreover, $b_*b^*t : b_*b^*P_w \ra B$ admits a unit $\eta : P_w \ra  b_*b^*P_w$ (over~$B$).  
\begin{equation}\label{diag:EEPfactored3}
\xymatrix@=1.5em{
A' \ar@{->>}[dddd] \ar[rrdd]_{\sim}^{w'} \ar@{>->}[rrrr] && && A \ar[dddd]_>>>>>>>>>>>>p \ar[rrdd]^{w} \ar[rr]^s && \ar@{.>}@/_1.5pc/[ll]_{p_w} P_w \ar[dd]^t \ar[rr]^-\eta && b_*b^*P_w \ar[lldd]^{b_*b^*t}\\
\\
&& B' \ar@{->>}[lldd] \ar@{>->}[rrrr]_<<<<<<<<<<b  &&&& B \ar@{->>}[lldd]^q &&&& \\
\\
C' \ar@{>->}[rrrr]_c &&&& C &&&&
}
\end{equation}
We now \emph{claim} that $\eta : P_w \ra  b_*b^*P_w$ is a trivial fibration.  Given that,  the composite $t = b_*b^*t \circ \eta$ is also a trivial fibration, whence $q\circ t : P_w \ra C$ is a fibration, and so its retract $p:A\ra C$ is a fibration.  Moreover, since $s$ is a section of the trivial fibration $p_w: P_w\ra A$ between fibrations, again by Corollary \ref{cor:342weqfiboverZ} it is also a weak equivalence. Thus $w=t\circ s$ is a weak equivalence, and we are finished. 

To prove the remaining claim that $\eta : P_w \ra  b_*b^*P_w$ is a trivial fibration, we shall use lemma \ref{lemma:etaTF}.  It does not apply directly, however, since $t : P_w \ra B$ is not yet known to be a trivial fibration.  Instead, we show that $\eta$ is a pullback of the corresponding unit at the trivial fibration $p_1 : B^\I \ra B$.

Consider the following cube (viewed with $b:B'\ra B$ at the front).
\begin{equation}\label{diag:ppcube}
\xymatrix@=1.5em{
P_{w'} \ar@{->>}[ddd]_{p_{w'}} \ar[rrd]^{(p'_0)^*w'} \ar@{>->}[rrr]^{\overline{a}}
		&&& P_w \ar[ddd]^{p_w} \ar[rrd]^{p_0^*w} &&\\
&& B'^\I \ar@{->>}[ddd]_{p'_0} \ar@{>->}[rrr]^{\overline{b}}  
		&&& B^\I \ar@{->>}[ddd]^{p_0} \\
\\
A' \ar@{>->}[rrr]_>>>>>>>a  \ar[rrd]_{w'}
		&&& A \ar[rrd]^{w} &&\\
&& B'\ar@{>->}[rrr]_b &&& B
}
\end{equation}
The right hand face is a pullback by definition, and the remainder results from pulling the entire right face back along $b$, by the stability of the pathspace factorization, Lemma \ref{lemma:relative_pathspace_factorization}. Thus all faces in the cube \eqref{diag:ppcube}
are pullbacks.  The base is also a pushforward, $b_*w'=w$, again by definition.  Thus the top face is also a pushforward, $\overline{b}_*((p'_0)^*w')=p_0^*w$. Indeed, since the front face is a pullback, the Beck-Chevalley condition applies, and so we have $\overline{b}_*(p'_0)^*(w') = p_0^*\,b_*(w') = p_0^*w$.

Now consider the following, in which the top square remains the same as in \eqref{diag:ppcube}, but $p_0$ has been relaced by $p_1 : B^\I \ra B$, so the composite at right is by definition $t = p_1\circ p_0^*w$.
\begin{equation}\label{diag:twounits}
\xymatrix@=1.5em{
P_{w'} \ar@{->>}[ddddrr]_{t'} \ar[rrd]^>>>>>>{(p'_0)^*w'} \ar@{>->}[rrr]^{\overline{a}}
		&&& P_w \ar[ddddrr]_t \ar[rrd]^{p_0^*w} &&\\
&& B'^\I \ar@{->>}[ddd]^{p'_1} \ar@{>->}[rrr]^<<<<<<<<{\overline{b}}  
		&&& B^\I \ar@{->>}[ddd]^{p_1} \\
\\
\\
&& B'\ar@{>->}[rrr]_b &&& B
}
\end{equation}
The horizontal direction is still pullback along $b$; let us rename $p_0^*w=:u$ so that $(p'_0)^*w' = b^*u$ and $t' = b^*t$ and $p'_1 = b^*p_1$ to make this clear. We then add the pushforward along $b$ on the right, in order to obtain the two units $\eta$.
\begin{equation}\label{diag:twounits2}
\xymatrix@=1.5em{
b^*P_{w} \ar@{->>}[ddddrr]_{b^*t} \ar[rrd]^{b^*u} \ar@{>->}[rrr]^{\overline{a}}
		&&& P_w \ar[ddddrr]_t \ar[rrd]^{u} \ar[rr]^{\eta_t}
		&& b_*b^* P_w  \ar[rrd]^{b_*b^* u} \ar[ddddrr]_{b_*b^* t}&&\\
&& b^*B^\I \ar@{->>}[ddd]^{b^*p_1} \ar@{>->}[rrr]^<<<<<<<<<{\overline{b}}  
		&&& B^\I \ar@{->>}[ddd]^{p_1} \ar[rr]^{\eta_{p_1}} 
		&& b_*b^* B^\I \ar[ddd]^{b_*b^*p_1}\\
\\
\\
&& B'\ar@{>->}[rrr]_b &&& B  \ar[rr]_= && B
}
\end{equation}
By the usual calculation of pushforwards in slice categories, $\overline{b}_* \cong \eta_{p_1}^*\circ b_*$, and so for $b^*u$ we have $\overline{b}_*b^* u = \eta_{p_1}^*b_*b^* u$.  But as we just determined in \eqref{diag:ppcube} the top left square is already a pushforward, and therefore $u = \eta_{p_1}^*b_*b^* u$,  so the top right naturality square is a pullback. 

To finish the proof as planned, $p_1 : B^\I \ra B$ is a trivial fibration because $q : B\ra C$ is a fibration, and $b : B' \mono B$ is a cofibration because it is a pullback of $c : C'\mono C$.  Thus by lemma \ref{lemma:etaTF}, we have that $\eta_{p_1}: B^\I \ra  b_*b^*B^\I$ is a trivial fibration, and so its pullback $\eta_t : P_w \ra  b_*b^*P_w$ is a trivial fibration, as claimed.
\end{proof}

\begin{remark}
Note that $p : A \ra C$ is small if $q : B\ra C$ is small.
\end{remark}

\section{The fibration extension property}\label{sec:FEP}

Given a universal fibration $\UU\fib\U$, such as $\FFib\fib\Fib$ of Proposition \ref{prop:UniversalunbiasedFib}, the fibration extension property (Definition \ref{def:fibextreplace}) is closely related to the statement that the base object $\U$ is fibrant.  For Kan simplicial sets, Voevodsky proved the latter directly, using the theory of minimal fibrations \cite{KLV:21}.   In a more general (but still simplicial) setting, Shulman \cite{Shu:15} gives a proof using univalence, in the form of the equivalence extension property of Section \ref{sec:EEP}, but that proof also uses the 3-for-2 property for weak equivalences, which we do not yet have. For cubical sets, Coquand \cite{CCHM:2018ctt} uses the equivalence extension property to prove that $\U$ is fibrant without assuming 3-for-2 for weak equivalences, via a neat type theoretic argument reducing box-filling to an operation of \emph{Kan-composition}.   We shall prove that $\U$ is fibrant using the equivalence extension property, also without assuming 3-for-2 for weak equivalences, but via a different argument than that in \cite{CCHM:2018ctt} not using (type theory or)  Kan composition.

Returning to the relation between the fibration extension property and the fibrancy of the base object of the universal fibration $\UU\fib\U$, it is easy to see that the latter implies the former.  Indeed, let $t : X\cof X'$ be a trivial cofibration and $Y \fib X$ a fibration.  To extend $Y$ along $t$, take a classifying map $y : X \ra \U$, so that $Y \cong y^*\UU$ over $X$. If $\U$ is fibrant then we can extend $y$ along $t : X\cof X'$ to get $y' : X' \ra \U$ with $y = y'\circ t$.  The pullback $Y' = (y')^*\UU \fib X'$ is then a (small) fibration such that $t^*Y' \cong t^*(y')^*\UU\cong y^*\UU \cong Y$ over $X$.  
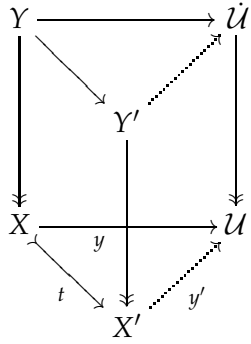
\[
\xymatrix{
Y \ar@{->>}[dd] \ar[rd] \ar[rr] && \UU \ar@{->>}[dd] \\
& Y' \ar@{->>}[dd] \ar@{..>}[ru] & \\
X  \ar@{>->}[rd]_t \ar[rr]_<<<<<<<<{y}  && \U  \\
& X' \ar@{..>}[ru]_{y'} &
}
\]
Thus, for the record, we have:

\begin{proposition}\label{prop:UfibtoFEP}
If the base object $\,\U$ of the universal fibration $\UU\fib\U$ is fibrant, then the fibration weak factorization system has the fibration extension property.
\end{proposition}

Conversely, given the Realignment Lemma \ref{lemma:realignmentforfibrations}, the fibration extension property also implies the fibrancy of $\U$:

\begin{corollary}
The fibration extension property implies that the base $\U$ of the universal fibration $\UU\fib\U$ is fibrant: given any $y : X \ra \U$ and trivial cofibration $t : X\cof X'$, there is a map $y' : X' \ra \U$  with $y'\circ t = y$.
\end{corollary}

\begin{proof}
Take the pullback of $\UU\fib\U$ along $y : X \ra \U$ to get a (small) a fibration $Y\fib X$, which extends along the (trivial) cofibration $t : X\cof X'$ by the fibration extension property,  to a (small) fibration $Y'\onto X'$ with $Y \cong t^*Y'$ over $X$.  By realignment there is a classifying map $y' : X' \ra \U$ for $Y'$ with $y'\circ t = y$.
\end{proof}

Now let us show the following.
\begin{proposition}\label{prop:Ufibrant}
The base $\U$ of the universal fibration $\UU\fib\U$ in $\cSet$, as constructed in Section \ref{sec:universalfibration}, is a fibrant object.
\end{proposition}

\begin{proof}
By Corollary \ref{cor:unbiasedfibrant}, $\U$ is an unbiased fibrant object if the canonical map $u = \langle \proj_2, \mathsf{eval}\rangle$ in the following diagram in $\cSet$, is a trivial fibration.
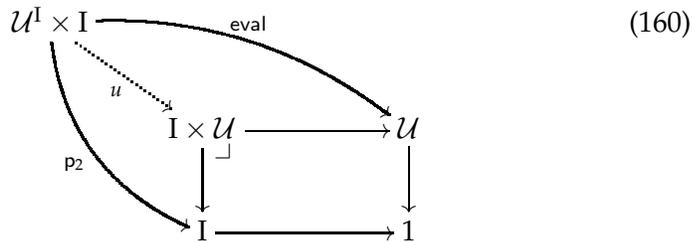
\begin{equation}\label{diagram:Uunbiasedfibrant}
\xymatrix{
\U^\I\times \I \ar@/^3ex/ [rrrd]^{\mathsf{eval}} \ar@/_3ex/ [rdd]_{\proj_2} \ar@{..>}[rd]_{u}  && \\
& \I\times \U \pbcorner  \ar[d] \ar[rr] && \U \ar[d] \\
& \I \ar[rr] && 1
}
\end{equation}
Thus consider a filling problem of the following form, with an arbitrary cofibration $c:C \cof Z$.
\begin{equation}\label{diag:Ufib1}
\xymatrix@=3em{
C \ar@{>->}[d]_c \ar[r] & \U^\I\times \I \ar[d]^{\langle \proj_2, \mathsf{eval}\rangle} \\
Z  \ar@{..>}[ru] \ar[r]  & \I\times\U  \\
}
\end{equation}
The horizontal maps may be written in the form $\langle i, b\rangle : Z \to \I\times \U$ and $\langle \tilde{a} , ic\rangle : C \to \U^\I\times\I$, regarding $i : Z \to \I$ as an $\I$-indexing.

Transposing $\tilde{a}$ to $a : C\times \I \ra \U$ we obtain the new problem
\begin{equation}\label{diagram:unbiasedpushoutproduct4}
\xymatrix{
C \ar@{>->}[d]_{c} \ar[rr]^{\gph{ic}} && C\times \I \ar[d]^{c\times\I} \ar@/^4ex/ [rdd]^{a} \\
Z \ar@/_4ex/ [rrrd]_{b} \ar[rr]_{\gph{i}} &&  Z\times\I  \ar@{.>}[rd]_{d} \\
&&& \U
}
\end{equation}
in which we recall from \eqref{eq:graphdef} the notation $\gph{i} = \gph{1_Z, i} : Z \to Z\times \I$ for the graph of a map $i : Z\to \I$.
Given a map $d$ as shown in \eqref{diagram:unbiasedpushoutproduct4}, we can obtain the indicated diagonal filler in \eqref{diag:Ufib1} as  $\langle \tilde{d} , i\rangle : Z \to \U^\I\times\I$.

As a sanity check, note that $b\circ c = a\circ \langle ic\rangle$ turns the problem \eqref{diagram:unbiasedpushoutproduct4} into that of extending the copair $[b,a]$ along the unique map
\[
Z +_C (C\times \I)\too Z\times \I\,,
\]
which is exactly the (trivial cofibration) pushout-product $c\otimes_i\delta$ from \eqref{diagram:unbiasedpushoutproduct}, recalled below for the reader's convenience.

\begin{equation}\label{diagram:unbiasedpushoutproduct5}
\xymatrix{
C \ar@{>->}[d]_{c} \ar[r]^{\gph{ic}} & C\times \I \ar[d] \ar@/^4ex/ [rdd]^{c\times\I}\\
Z \ar@/_4ex/ [rrd]_{\gph{i}} \ar[r] &  Z +_C (C\times\I) \ar@{.>}[rd]_{c\, \otimes_i \delta} \\
&& Z\times\I
}
\end{equation}

Returning to \eqref{diagram:unbiasedpushoutproduct4}, take pullbacks of $\UU\fib \U$ along $a$ and $b$ to get fibrations $p_a : A\fib C\times \I$ and $p_b : B\fib Z$ respectively, and let 
\[
p_c :=  \gph{ic}^*p_a : A_c \too C
\]
be the corresponding ``fiber of $A$ over the graph of $ic$''.  We then have $c^*B \cong A_c$ over $C$ by the commutativity of the outer square of \eqref{diag:Ufib1}.
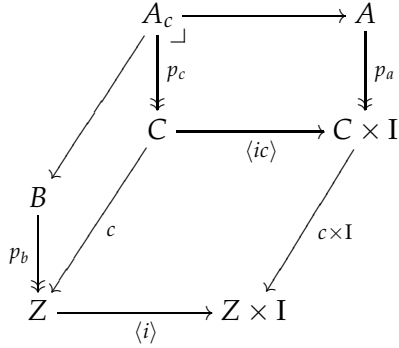
\[
\xymatrix@=1em{
&& \ar[llddd] A_c \ar@{->>}[dd]^{p_c} \ar[rr]  \pbcorner &&  A \ar@{->>}[dd]^{p_a} \\
&& && \\
&& \ar[llddd]^c C \ar[rr]_{\gph{ic}}  &&  \ar[lddd]^{c\times \I} C\times\I \\
B \ar@{->>}[dd]_{p_b} &&&& \\
&&&& \\
Z \ar[rrr]_{\gph{i}} &&& Z\times\I &
}
\]

The diagonal filler sought in \eqref{diag:Ufib1} now corresponds, again by transposition and pullback of $\UU\fib\U$, to a fibration $p_d : D\fib Z\times \I$ with $\langle i\rangle^*D \cong B$ over $Z$ and $(c\times \I)^*D \cong A$ over $C\times \I$, as indicated below.
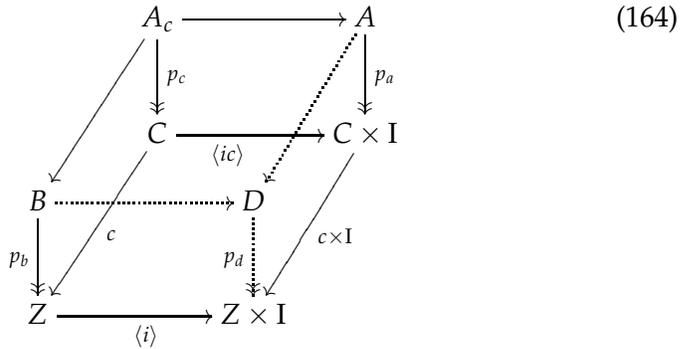
\begin{equation}\label{diag:Ufib1.5}
\xymatrix@=1em{
&& \ar[llddd] A_c \ar@{->>}[dd]^{p_c} \ar[rr]  &&  \ar@{..>}[lddd] A \ar@{->>}[dd]^{p_a} \\
&& && \\
&& \ar[llddd]^c C \ar[rr]_<<<<<<<{\gph{ic}}  &&  \ar[lddd]^{c\times\I} C\times\I \\
B \ar@{->>}[dd]_{p_b} \ar@{..>}[rrr] &&& D \ar@{..>>}[dd]_{p_d} & \\
&&&& \\
Z \ar[rrr]_{\gph{i}} &&& Z\times\I &
}
\end{equation}
We shall construct  $p_d : D\fib Z\times \I$ using the equivalence extension property (Proposition \ref{prop:EEP}) as follows.  First apply the functor $(-)\times\I$ to the left vertical (pullback) face of the cube in \eqref{diag:Ufib1.5} to get the following, with a new pullback square on the right with the indicated fibrations.
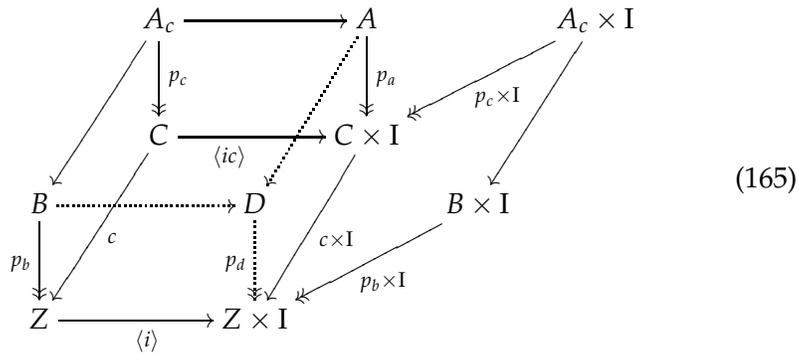
\begin{equation}\label{diag:Ufib2}
\begin{gathered}
\xymatrix@=1em{
&& \ar[llddd] A_c \ar@{->>}[dd]^{p_c} \ar[rr]  &&  \ar@{..>}[lddd] A \ar@{->>}[dd]^{p_a} &&  \ar@{->>}[lldd]^{p_c\times \I} A_c\times\I \ar[lddd]  \\
&& && &&\\
&& \ar[llddd]^c C  \ar[rr]_<<<<<<<{\gph{ic}}  &&  \ar[lddd]^{c\times\I} C\times\I && \\
B \ar@{->>}[dd]_{p_b} \ar@{..>}[rrr] &&& D \ar@{..>>}[dd]_{p_d} && \ar@{->>}[lldd]^{p_b\times\I} B\times\I &\\
&&&& &&\\
Z \ar[rrr]_{\gph{i}} &&& Z\times\I &&&
}
\end{gathered}
\end{equation}
We now \emph{claim} that  there is a weak equivalence $e:A \simeq A_c \times \I$ over $C\times \I$. From this it  follows by the equivalence extension property (Proposition \ref{prop:EEP}) that there are:
\begin{enumerate}
\item[(i)] a fibration $p_d : D\fib Z\times \I$ with $(c\times\I)^*D \cong A$ over $C\times\I$, and 
\item[(ii)] a weak equivalence $f:D\simeq B\times\I$ over $Z\times \I$ with $(c\times\I)^*f \cong e$ over $C\times\I$. 
\end{enumerate}
It then remains only to show that $B\cong \gph{i}^*D$ over $Z$ to complete the proof.

To obtain the claimed weak equivalence $e$, consider the following square, 
\begin{equation}\label{diag:Ufibn}
\xymatrix{
A_c \ar[d] \ar[rr]^{\langle icp_c \rangle} && A_c\times\I \ar@{>>}[d]^{p_c\times \I} \\
A  \ar@{>>}[rr]_{p_a} &&  C\times\I \,,
}
\end{equation}
in which the top horizontal map is the graph of the composite,
\[
A_c \stackrel{p_c}{\fib} C \stackrel{c}{\mono} Z \stackrel{i}{\to} \I\,,
\]
and the others are the evident ones from \eqref{diag:Ufib2}.
The square is easily seen to commute, and the top map is a trivial cofibration (by Remark \ref{rem:specialtrivcofs}), because it is the graph of a map into $\I$.  The left map is also a trivial cofibration by Frobenius (Proposition \ref{prop:Frobenius}), because by its definition in \eqref{diag:Ufib1.5} it is the pullback of another such graph $\gph{ic}$ along the fibration $p_a$.  A simple lemma (Lemma \ref{lem:simple} below) provides the claimed weak equivalence $e:A \simeq A_c\times \I$ over $C\times \I$.  

To see that $B\cong \gph{i}^*D$ over $Z$, recall from the proof of the equivalence extension property that the map $f:D\cong B\times \I$ is the pushforward of  $e:A \simeq A_c\times \I$ along the cofibration $b_c\times\I :A_c\times\I \cof B\times\I$, where we are calling the evident map in \eqref{diag:Ufib2} $b_c : A_c\cof B$.  Thus by construction $f = (b_c\times\I)_*\,e$.  We can then apply the Beck-Chevalley condition for the pushforward using the pullback square on the left below.
\begin{equation}
\xymatrix{
A_c \ar@{>->}[d]_{b_c} \ar@{>->}[rr]^{\gph{icp_c}} \pbcorner && A_c\times\I \ar@{>->}[d]^{b_c\times \I} & \ar[l]_-e A\\
B  \ar@{>->}[rr]_{\gph{ip_b}}  &&  B\times\I  & \ar[l]^-f D
}
\end{equation}
The pullback of $e$ along the top of the square is the identity on $A_c$, as can be seen by pulling back $e$ as a map over $C\times \I$ along $\gph{ic} :  C\ra C\times\I$. Thus the same is true up to isomorophism for the pullback of $f$ along the bottom.

An application of the Realignment Lemma \ref{lemma:realignmentforfibrations} along the trivial cofibration $c\otimes_i\delta$ completes the proof.
\end{proof}

\begin{lemma}\label{lem:simple}
Suppose the following square commutes and the indicated cofibrations are trivial.
\begin{equation}\label{diag:standard1}
\xymatrix{
A \ar@{>->}[d] \ar@{>->}[r] & C \ar@{>>}[d] \\
B  \ar@{>>}[r] &  D \\
}
\end{equation}
Then there is a weak equivalence $e : B \simeq C$ over $D$ (and under $A$).
\end{lemma}
\begin{proof}
Use the fact that any two diagonal fillers are homotopic to get a homotopy equivalence $e : B \simeq C$ filling the square.
\end{proof}

\begin{remark}
The foregoing proof of Proposition \ref{prop:Ufibrant}, the fibrancy of the universe $\U$, also works, \emph{mutatis mutandis}, for the universe of \emph{biased} fibrations, as used in the setting of \cite{CCHM:2018ctt}.
\end{remark}

Applying proposition \ref{prop:UfibtoFEP} now yields the following.

\begin{corollary}[Fibration extension property]\label{cor:FEP}
The fibration weak factorization system has the fibration extension property (definition \ref{def:fibextreplace}). 
\end{corollary}

By Theorem \ref{theorem:QMSmodFEP}, finally, we have the following.

\begin{theorem}\label{theorem:QMS}
There is a Quillen model structure $(\CC,\WW,\FF)$ on the category $\cSet$ of cubical sets for which:
\begin{enumerate}
\item the cofibrations $\CC$ are any class of maps satisfying (C0)-(C8) (equivalently, the simplified axioms in  Appendix A),
\item the fibrations $\FF$ are the maps $f : Y\ra X$ for which the canonical map 
\[
(f^\I \times \I, \eval_Y) : Y^\I \times \I\too (X^\I \times \I)\times_X Y
\]
lifts on the right against $\CC$.
\item the weak equivalences $\WW$ are the maps $w : X\ra Y$  for which the internal precomposition $K^w : K^Y \to K^X$ is bijective on connected components for every fibrant object $K$.
\end{enumerate}
\end{theorem}

\begin{remark}
We note that in terms of the universal fibration $\UU\fib\U$ constructed in Section \ref{sec:U} the equivalence extension property Proposition \ref{prop:EEP} says that the second projection from the classifying type of equivalences $A\simeq B$ between small families, 
\[
\pi_2 : \Sigma_{A,B}\mathsf{Eq}(A,B) \too \U \,,
\]
is a trivial fibration.  From this, it follows that the canonical transport map 
\begin{equation}\label{eq:UA}
* : \U^\I \too \Sigma_{A,B}\mathsf{Eq}(A,B)
\end{equation}
 is an equivalence over the base $\U$ via $p_2: \U^\I \to \U$, which is a trivial fibration because $\U$ is fibrant by Proposition \ref{prop:Ufibrant}.  In type theory, the pathobject $\U^\I$ of course interprets the identity type $A=B$, so the equivalence \eqref{eq:UA} can be expressed as 
 \[
 (A=B)\simeq(A\simeq B)\,.
 \]
 \end{remark}

\section*{Appendix A: Axioms for Cartesian cofibrations}\label{appendix:Cofibrations}

A system of maps satisfying the axioms (C0)-(C8) above for the cofibrations in a cartesian cubical model category will be called \emph{cartesian cofibrations}. The axioms may be restated equivalently as follows.
\begin{enumerate}
\item[(A0)] All cofibrations are monomorphisms.
\item[(A1)] All isomorphisms are cofibrations.
\item[(A2)] The composite of two cofibrations is a cofibration.
\item[(A3)] Any pullback of a cofibration is a cofibration.
\item[(A4)] The join of two cofibrant subobjects is a cofibration.
\item[(A5)] The diagonal of the interval $\I\ra\I\times\I$ is a cofibration.
\item[(A6)] Cofibrations are preserved by the pathobject functor $(-)^\I$.
\item[(A7)] The category of cofibrations and Cartesian squares has a terminal object .
\end{enumerate}

\begin{example}
Consider the cartesian cubical presheaves $\cEE = \EE^{\op{\C}}\!= \EE^{\B}$ in a topos $\EE$.
For such (internal) discrete opfibrations $(A,\, \alpha)$,
\[\begin{tikzcd}
A  \ar[d] & \ar[l,swap,"\alpha"] \ar[d] \B_1 \times_{\B_0} A \ar[r] \pbmark  &  \ar[d] A \\
\B_0 & \ar[l, "\cod" ] \B_1 \ar[r,swap,"\dom"] & \B_0
\end{tikzcd}\]
over the (internal) category $\B = (\B_1\rightrightarrows\B_0)$ of finite bipointed sets, call a subpresheaf $c : (C,\gamma)\to (A,\alpha)$ \emph{locally complemented} if the underlying map $c:C\to A$ over $\B_0 = \N$ is a complemented subobject in $\EE/_{\!\N}$, \ie\ $C+\neg{C} \cong A$ over $\N$.  Internally, this means that
\begin{equation}\label{pointwisecomplemented}
C_n + \neg(C_n) \cong A_n\qquad\text{for all $n\in\N$},
\end{equation}
which is a weaker condition than $(A,\alpha)+\neg (A,\alpha) \cong (B,\beta)$ as presheaves (unless $\EE=\Set$, in which case it is trivial).
\begin{proposition}
 For any topos  $\EE$, the locally complemented subobjects in the category $\cEE$ of cubical $\EE$-objects satisfy the axioms for cartesian cofibrations.
\end{proposition}
\begin{proof}
Axioms (A0)-(A4)  are satisfied by the complemented subobjects in $\EE/_{\!\N}$, and the forgetful functor $U : \cEE \to \EE/_{\!\N}$ creates the monos, isos, composites, pullbacks, and joins in question.  
For (A5), we use the fact that the equality relation on $\N$ is decidable to infer that, for each $[n] = \{0, x_1, \dots, x_n, 1\}$, the finite set $\{i = j \,|\, 0\leq i,j \leq n+1\}$ is complemented in $\B([1],[n])\times \B([1],[n])$,  and so for the subpresheaf $\delta : \I \to \I\times\I$ we indeed have,
\begin{equation}
\delta_n + \neg(\delta_n) \cong \I_n\times \I_n \qquad\text{for all $n\in\N$}\,.
\end{equation}
For (A6) we use the fact that the pathobject $A^{\I}$ is a shift by one dimension, together with $(A^{\I})_{n} = A_{n+1}$, together with \eqref{pointwisecomplemented}.
The cofibration classifier in (A7) is given by applying the right adjoint $U \dashv R : \EE/_\N \to \cEE$ to the complemented subobject classifier $\N^*2$ of $\EE/_{\!\N}$.
\end{proof}
\end{example}

\section*{Appendix B: Cartesian cubical sets classifies intervals}\label{appendix:classtopos}

Recall from Section \ref{sec:cSet} that the objects of the \emph{Cartesian cube category} $\Box$ may be taken concretely to be finite, strictly bipointed sets, written
\[
[n] = \{0, x_1, ..., x_n, 1\},
\]
 and the arrows $f : [n] \to [m]$ to be all bipointed maps $[m]\to [n]$ (note the direction). 
The category of \emph{(Cartesian) cubical sets} is then the presheaf topos
\[
\cSet =  \psh{\Box}.
\]
It is generated by the \emph{$n$-cubes} $\I^n = \yon[n]$, with  $1 = \I^0$, $\I = \yon[1]$; and $\I^n \times \I^m \cong \I^{n+m}$ by preservation of products by the Yoneda embedding $\yon : \op{\Box} \hook\cSet$.  For a cubical set $X: \op{\Box} \to \Set$ we have the usual Yoneda correspondence for the set $X_n$ of  \emph{$n$-cubes in $X$},
\[
\{c \in X_n \} \ \cong\ \{c : \I^n \to X\}.
\]
In particular, $\I^n_m = \C([m],[n])$ is the set of $m$-cubes in the $n$-cube.\footnote{Note that the cardinality of $\I^n_m$ is therefore just $(m+2)^n$, in comparison to the \emph{Dedekind} cubes $\C_{\wedge,\vee}$ used in \cite{CCHM:2018ctt,orton-pitts}, for which \eg\ $\C_{\wedge,\vee}([n],[1])$ is the $n^{th}$ \emph{Dedekind number}, the number of elements in the free distributive lattice on $n$ generators, which is in general a number so large that it is unknown for values of $n>8$.} 

\begin{proposition}\label{prop:csetclassifiesintervals}
The category $\cSet$ of Cartesian cubical sets is the classifying topos for \emph{intervals}: objects $\mathcal{I}$ with points $i,j:1\rightrightarrows \mathcal{I}$ the pullback of which is $0$:
\[
\xymatrix{
0 \ar[d] \ar[r] \pbcorner & 1 \ar[d]^j  \\
1 \ar[r]_i & \mathcal{I}
}
\]
\end{proposition}

\begin{proof}
Consider the covariant presentation $\cSet = \Set^\B$ where $\B$ is the category of finite, strictly bipointed sets and bipointed maps.  We can extend $\B\hook\B_=$ by freely adjoining coequalizers, making $\B_=$ the free finite \emph{colimit} category on a co-bipointed object. A concrete presentation of $\B_=$ is the finite bipointed sets, including those with $0=1$.  Let us write $(n)$ for the bipointed set $\{x_1, ..., x_n, *\}$, with $n$ (non-constant) elements and a further element $0=*=1$.  There is an evident coequalizer $[1]\rightrightarrows [n]\to(n)$, which (only) identifies the distinguished points, and every coqualizer in $\B_=$ has either the form $[m]\rightrightarrows [n]\to [k]$ or $[m]\rightrightarrows [n]\to(k)$, for a suitable choice of $k$.  Note that there are no maps of the form $(m)\to [n]$, and that every map $[m]\to (n)$ factors uniquely as $[m]\to (m)\to (n)$ with  $[m]\to(m)$ the canonical coequalizer of $0$ and $1$.  The category $\B_=$ can therefore be decomposed into two ``levels'', the upper one of which is essentially $\B$, and the lower one consisting of just the objects $(n)$, and thus essentially the finite \emph{pointed} sets, and for each $n$, there is the canonical coequalizer $[n]\to(n)$ going from the upper level to the lower one.
\[
\xymatrix{
\dots \ar[r] & [m] \ar@{->>}[d] \ar[r] & [n] \ar@{->>}[d] \ar[r] & \dots \\
\dots \ar[r] & (m) \ar[r] & (n) \ar[r] & \dots
}
\]
Write $u : \B\to\B_=$ for the upper inclusion, which is the classifying functor of generic co-bipointed object in $\B_=$.  

Now consider the induced geometric morphism:
\[
\xymatrix{
\Set^\B \ar@<-2ex>[rr]_{u_!} \ar@<2ex>[rr]^{u_*} && \ar[ll]|-{\,u^*} \Set^{\B_=}  & u_! \dashv u^* \dashv u_*
}
\]
Since $u^*$ is the restriction along $u$, the right adjoint $u_*$ must be ``prolongation by $1$'',
\begin{align*}
u_*(P)[n] &= P[n],\\
u_*(P)(n) &= \{*\},
\end{align*}
with the obvious maps,
\[
\xymatrix{
\dots \ar[r] & P[m] \ar@{->>}[d] \ar[r] & P[n] \ar@{->>}[d] \ar[r] & \dots \\
\dots \ar[r] & \{*\} \ar[r] & \{*\} \ar[r] & \dots
}
\]
as is easily seen by considering maps in $\Set^{\B_=}$ of the form
\[
\xymatrix{
 Q[n] \ar[d]_{\cdot} \ar[r] & P[n] \ar@{->>}[d]^{\cdot} \\
 Q(n) \ar[r]^{} & \{*\} .
}
\]
Since $u_* : \Set^{\B} \to \Set^{\B_=}$ is evidently full and faithful, it is the inclusion part of a sheaf subtopos $\mathsf{sh}(\B^\mathsf{op}_=, j) \hook \Set^{\B_=}$ for a suitable Grothendieck topology $j$ on $\B^\mathsf{op}_=$.  We claim that $j$ is the closed complement topology of the subobject $[0 = 1] \rightarrowtail 1$ represented by the coequalizer $[0]\to (0)$.  Indeed, in $\Set^{\B_=}$ we have the representable functors:
\begin{align*}
\I &= y[1],\\
1 &= y[0],\\
[0=1] &= y(0)
\end{align*}
fitting into an equalizer $[0=1] \to 1 \rightrightarrows \I$, which is the image under Yoneda of the canonical coequalizer $[1] \rightrightarrows [0]\to (0)$ in $\B_=$.  The closed complement topology for $[0=1] \mono 1$ is generated by the single cover $0\to [0=1]$, which can be described logically as 
forcing the sequent $(0=1 \vdash \bot)$ to hold.  Recall from \cite{JohnstoneTT}, Proposition 3.53, the following simple characterization of the sheaves for the closed complement topology of an object $U\mono 1$: an object $X$ is a sheaf iff $X\times U \cong U$. In the present case, it therefore suffices to show that for any $P:\B_= \to \Set$ we have:
\[
P\times [0=1] \cong [0=1] \quad\text{iff}\quad P(n) = 1\ \text{for all $n$}.
\]
For any object $b\in \B_=$, consider the map
\[
\hom(yb, P\times [0=1] ) \cong \hom(yb, P) \times \hom(yb, [0=1]) \to \hom(yb, [0=1]) .
\]
If $b = [k]$, then $\hom(yb, [0 = 1]) \cong \hom_{\B_=}((0), [k]) \cong 0$, and so we always have an iso
\[
\begin{split}
\hom(yb, P\times [0 = 1] ) \cong \hom(yb, P) \times \hom(yb, [0=1])\\
 \cong \hom(yb, P) \times 0 \cong 0.
 \end{split}
\]
If $b = (k)$, then $\hom(y(k), [0=1]) \cong \hom_{\B_=}((0), (k)) \cong 1$, and we have an iso
\[
\begin{split}
\hom(y(k), P\times [\bot=\top] ) \cong \hom(y(k), P) \times \hom(y(k), [0=1]) \\
\cong \hom(y(k), P) \times 1 \cong \hom(y(k), P) \cong P(k).
\end{split}
\]
Thus either way we will have an iso $P\times [0=1] \cong [0=1]$ iff $P(k) \cong 1$.

The presheaf topos $\Set^\B$ is therefore the closed complement of the open subtopos
\[
\Set^{\B_=}\!/_{\![0=1]}\ \hook\ \Set^{\B_=}\,,
\]
given by forcing the proposition $0 \neq 1$.  Since $\Set^{\B_=}$ is clearly the classifying topos for \emph{arbitrary} bipointed objects, say $\Set[B, 0,1]$, the sheaf subtopos 
\[
 \Set[B, 0\neq 1]\ \simeq\ \Set^\B\ \hook\ \Set^{\B_=}\ \simeq\ \Set[B, 0,1]
\]
classifies \emph{strictly} bipointed objects, \ie\ intervals, as claimed.
\end{proof}

\begin{corollary}
The geometric realization functor to topological spaces 
\[
R: \cSet \to \mathsf{Top}
\]
preserves finite products, $R(X\times Y) \cong R(X)\times R(Y)$ and $R(1) \cong \{*\}$.
\end{corollary}

\begin{proof}
Compose the inverse image of the classifying geometric morphism $\mathsf{sSets} \to \cSet$ of the $1$-simplex $\Delta^1$ with the standard geometric realization $\mathsf{sSets} \to \mathsf{Top}$, both of which preserve finite products.
\end{proof}

\begin{example}[P.~Aczel]
The cubical set $P$ of polynomials (say, over the integers), is defined by:
\[
P_n = \{ p(x_1, ...,  x_n)\ |\  \text{polynomials in at most}\  x_1, ..., x_n \}
\]
with the substitution map $s^* : P_n \to P_m$ taking $p(x_1, ...,  x_n)$ to 
\[
 s^*p(x_1, ...,  x_n) = p\big(s(x_1),\dots, s(x_n)\big) \,,
\]
for each bipointed map $s:[n] \to [m]$.

This cubical set $P$ underlies a ring object in $\cSet$, and the interval $\I = \y[1]$ embeds into it via the component maps
\[
\eta_n : \I_n \to P_n
\]
taking $v_i \in \C([n], [1])\cong \B\big([1],[n]\big) \cong \{0, x_1, \dots, x_n, 1\} $ to $0$, $1$, or the variable $x_i$, respectively, in $P_n$.
The same is true for any algebraic theory $\T$ with two constants, 
such as Boolean algebras: there is a distinguished cubical $\T$-algebra $\mathcal{A}$, and a natural map $\eta : \I \to |\mathcal{A}|$ in $\cSet$.

Indeed, let $\cSet = \Set[\mathcal{I}]$ as a classifying topos for intervals by Proposition \ref{prop:csetclassifiesintervals} with $\mathcal{I} = (1\rightrightarrows \I)$, and let
\[
\Set[\T,\mathsf{flat}]= \Set^{\op{\T}}
\]
 be the topos of presheaves on the Lawvere algebraic theory $\T$, which therefore classifies \emph{flat} $\T$-algebras.  There is a bipointed object $\mathcal{J} = (1\rightrightarrows \J)$ in $\T$, consisting of the generic $\T$-algebra and its two constants, which has a classifying functor $\J^\sharp : \C \to \T$, inducing adjoint functors on presheaves, 
\[
\J_! \dashv \J^* \dashv \J_* : \Set[\mathcal{I}] = \psh{\C} \too \psh{\T} = \Set[\T,\mathsf{flat}]\,,
\]
where $\J_! \circ \y_{\C} =\ \y_{\T} \circ \J^\sharp$,  with $\y$ the respective Yoneda embeddings.

We can then calculate,
\begin{equation}
\begin{split}
\J^* \J_! ( \I )( [n] ) &= \J^* \J_! ( \y[1] )( [n] )\\
 &= \J^* \y( \J^\sharp[1] )( [n] ) = \y( \J^\sharp[1] )( \J^\sharp[n] ) \\
 &= \T( \J^\sharp[n] , \J^\sharp[1] ) = \T\mathsf{Alg}( \J^\sharp[1] , \J^\sharp[n] ) \\
 &= \T\mathsf{Alg}( F(1) , F(n) ) = | F(n) |,
\end{split}
\end{equation}
where $|F(n)|$ is the underlying set of the free $\T$-algebra $F(n)$, the $n^{th}$ object of the Lawvere theory under its dual presentation $\op{\T}\hook \T\mathsf{Alg}$.  The unit of the $J_! \dashv J^*$ adjunction provides a natural map $\eta: \I \to \J^* \J_! ( \I )$, given pointwise by  $\I_n \cong   \{0, x_1, \dots, x_n, 1\} \too | F(n) | \cong \J^* \J_! ( \I )_n $.

The cubical set of polynomials $P = \J^* \J_! ( \I )$ is thus indeed a cubical ring, with a map $\I \to P$, since $\J_!(\I) \cong \y_\T\,\J^\sharp([1]) \cong \y_\T(\J)$ is a ring in $\Set[\T,\mathsf{flat}]$ and $J^*$ is left exact.  In fact, we learn thereby  that $P$ is flat.
\end{example}

\begin{definition}
Let $\Box \to \Cat$ be the unique product-preserving functor taking the interval $[1]$ to the one arrow  category $\bbtwo = (0\leq1)$.  This functor then takes $[n]$ to $\bbtwo^n$, the $n$-fold product in $\Cat$, and maps $[m] \to [n]$ to the corresponding monotone functions $\bbtwo^m \to \bbtwo^n$ of posets.\footnote{Thus factoring through the full subcategory $\C_{\wedge,\vee} \hook\Cat$ of \emph{Dedekind cubes}, mentioned above, which is the Lawvere algebraic theory of distributive lattices.}   The  \emph{cubical nerve} functor 
\[
N : \Cat \to \cSet
\]
 is then defined by:
\[
N(\bbC)_n = \Cat(\bbtwo^n, \bbC).
\]
Thus $N(\bbC)_0$  is the set of objects of \bbC; $N(\bbC)_1$ is the set of arrows; $N(\bbC)_2$ consists of all commutative squares; $N(\bbC)_3$ all commutative cubes, etc. 
\end{definition}

\begin{proposition}\label{prop:nervefull}
The cubical nerve $N : \Cat \to \cSet$ is full and faithful.
\end{proposition}
\begin{proof}
Given categories $\bbC$ and $\D$ and functors $F, G : \bbC\to\D$, suppose $F(f) \neq G(f)$ for some $f : A\to B$ in $\bbC$.  Take $f^{\sharp} : \bbtwo \to \bbC$ with image~$f$.  Then $N(F)_1(f^\sharp) = F(f) \neq G(f) = N(G)_1(f^\sharp)$, and so $N(F) \neq N(G) : N(\bbC)\to N(\D)$.  So $N$ is faithful.

For fullness, let $\varphi : N(\bbC) \to N(\D)$ be a natural transformation, and define a proposed functor $F : \bbC\to \D$ by  
\begin{align*}
F_0 = \varphi_0 &: \bbC_0 = N(\bbC)_0 \to N(\D)_0 = \D_0\\
F_1 = \varphi_1 &: \bbC_1 = N(\bbC)_1 \to N(\D)_1 = \D_1.
 \end{align*}
We just need to show that $F$ preserves identity arrows and composition.
Consider the following diagram.
\[
\xymatrix{
\Cat(\bbtwo^1, \bbC) = N(\bbC)_1 \ar[r]^{F_1} & N(\D)_1 = \Cat(\bbtwo^1, \D) \\
\Cat(\bbtwo^0, \bbC) = N(\bbC)_0 \ar[u]^{!^*} \ar[r]_{F_0} & N(\D)_0 = \Cat(\bbtwo^0, \D) \ar[u]_{!^*}.
}
\]
Here $!^* : \Cat(\bbtwo^0, \bbC) \to \Cat(\bbtwo, \bbC) $ is precomposition with $! : \bbtwo = \bbtwo^1 \rightarrow \bbtwo^0 = \mathbbm{1}$, so the diagram commutes.  But since $! : \bbtwo \rightarrow \mathbbm{1}$ is a functor, 
\[
\bbC_0 = \Cat(\mathbbm{1}, \bbC)  \stackrel{!^*}{\to}  \Cat(\bbtwo, \bbC) = \bbC_1
\]
 takes objects in $\bbC$ to their identity arrows.  Thus $F$ preserves identity arrows.  Similarly, for composition, consider 
\[
\xymatrix{
\Cat(\bbtwo^2, \bbC) = N(\bbC)_2 \ar[d]_{d^*} \ar[r]^{\varphi_2} & N(\D)_2 = \Cat(\bbtwo^2, \D) \ar[d]^{d^*} \\
\Cat(\bbtwo, \bbC) = N(\bbC)_1  \ar[r]_{F_1} & N(\D)_1 = \Cat(\bbtwo, \D).
}
\]
where $\varphi_2 : N(\bbC)_2 \to N(\D)_2$ is the action of $\varphi$ on commutative squares of arrows, and $d^* : \Cat(\bbtwo^2, \bbC) \to \Cat(\bbtwo, \bbC)$ is precomposition with the diagonal map $d : \bbtwo \rightarrow \bbtwo^2 = \bbtwo\times\bbtwo$, so the diagram commutes.  For any composable $f: A \to B$ and $g:B\to C$ in $\bbC$ there is a commutative square 
\[
\xymatrix{
A \ar[r]^{f} \ar[d]_{f} & B \ar[d]^{g} \\
B  \ar[r]_{g} & C,
}
\]
and the effect of $d^* : \Cat(\bbtwo^2, \bbC) \to \Cat(\bbtwo, \bbC)$ on this square is exactly $g\circ f: A\to C$, and similarly for $d^* : \Cat(\bbtwo^2, \D) \to \Cat(\bbtwo, \D)$.  Thus the commutativity of the above diagram implies that $F$ preserves composition.
Since clearly $N(F) = \varphi$, we indeed have that $N$ is also full.
\end{proof}

\bibliographystyle{alpha}
\bibliography{../references}

\end{document}